\documentclass[12pt,a4paper,reqno]{amsart}
\usepackage{amscd}
\usepackage{amssymb}
\usepackage[centering,text={14.5cm,22cm}]{geometry}
\usepackage{graphicx}
\usepackage{color}
\usepackage[all]{xy}
\usepackage{mathrsfs}
\usepackage{marvosym}
\usepackage{stmaryrd}
\usepackage{srcltx}
\usepackage{hyperref}

\usepackage{floatrow}
\newfloatcommand{capbtabbox}{table}[][\FBwidth]

\definecolor{shadecolor}{rgb}{1,0.9,0.7}

\newcommand\restr[2]{{
  \left.\kern-\nulldelimiterspace 
  #1 
  \vphantom{\big|} 
  \right|_{#2} 
  }}

\let\oldtocsection=\tocsection
 
\let\oldtocsubsection=\tocsubsection
 
\let\oldtocsubsubsection=\tocsubsubsection
 
\renewcommand{\tocsection}[2]{\hspace{0em}\oldtocsection{#1}{#2}}
\renewcommand{\tocsubsection}[2]{\hspace{1em}\oldtocsubsection{#1}{#2}}
\renewcommand{\tocsubsubsection}[2]{\hspace{2em}\oldtocsubsubsection{#1}{#2}}

\newtheorem{theorem}{Theorem}[section]
\newtheorem{lemma}[theorem]{Lemma}
\newtheorem{proposition}[theorem]{Proposition}
\newtheorem{corollary}[theorem]{Corollary}

\theoremstyle{definition}
\newtheorem{definition}[theorem]{Definition}

\newtheorem{construction}[theorem]{Construction}

\newtheorem{example}[theorem]{Example}

\theoremstyle{remark}
\newtheorem{remark}[theorem]{Remark}
\newtheorem{remarks}[theorem]{Remarks}

\numberwithin{equation}{section}
\numberwithin{figure}{section}

\newcommand{\NN} {\mathbb{N}}
\newcommand{\ZZ} {\mathbb{Z}}

\newcommand{\RR} {\mathbb{R}}

\newcommand{\PP} {\mathbb{P}}
\renewcommand{\AA} {\mathbb{A}}
\newcommand{\GG} {\mathbb{G}}

\newcommand{\VV} {\mathbb{V}}

\newcommand {\shE}  {\mathcal{E}}

\newcommand {\shM}  {\mathcal{M}}

\newcommand {\shP}  {\mathcal{P}}

\newcommand {\shX}  {\mathcal{X}}

\newcommand {\bsigma} {\boldsymbol{\sigma}}

\newcommand {\foD}  {\mathfrak{D}}

\newcommand {\foM}  {\mathfrak{M}}

\newcommand {\foT}  {\mathfrak{T}}

\newcommand {\foc}  {\mathfrak{c}}
\newcommand {\fod}  {\mathfrak{d}}

\newcommand {\foj}  {\mathfrak{j}}

\newcommand {\fom}  {\mathfrak{m}}

\newcommand {\fov}  {\mathfrak{v}}

\newcommand {\Aff}  {\operatorname{Aff}}

\newcommand {\as}  {\mathrm{as}}

\newcommand {\Aut}  {\operatorname{Aut}}

\newcommand {\codim} {\operatorname{codim}}

\newcommand {\coker} {\operatorname{coker}}

\newcommand {\Div}  {\operatorname{Div}}

\newcommand {\GL}  {\operatorname{GL}}

\newcommand {\gp}  {{\operatorname{gp}}}

\newcommand {\Hom}  {\operatorname{Hom}}

\newcommand {\id}  {\operatorname{id}}

\newcommand {\inc}  {\mathrm{in}}
\newcommand {\ind}  {\operatorname{ind}}
\newcommand {\Int}  {\operatorname{Int}}

\newcommand {\Interstices}  {\operatorname{Interstices}}

\newcommand {\Joints} {\operatorname{Joints}}

\renewcommand {\ker } {\operatorname{ker}}
\newcommand {\kk} {\Bbbk}

\newcommand {\Log} {\mathcal{L}og\hspace{1pt}}

\newcommand {\lra}  {\longrightarrow}

\newcommand {\M} {\mathcal{M}}

\renewcommand {\max} {{\operatorname{max}}}

\renewcommand{\O}  {\mathcal{O}}

\newcommand {\ol} {\overline}

\newcommand {\out}  {\mathrm{out}}
\newcommand{\trop}{\mathrm{trop}}

\renewcommand{\P}  {\mathscr{P}}

\newcommand {\Pic}  {\operatorname{Pic}}

\newcommand {\pr}  {\operatorname{pr}}

\newcommand {\scrM}  {\mathscr{M}}

\newcommand {\scrH}  {\mathscr{H}}
\newcommand {\scrP}  {\mathscr{P}}

\newcommand {\Sing} {\operatorname{Sing}}

\newcommand {\Spec} {\operatorname{Spec}}

\newcommand {\Supp} {\operatorname{Supp}}

\newcommand {\Scatter} {\operatorname{{\mathsf{S}}}}

\newcommand {\tors}  {\mathrm{tors}}

\newcommand {\ul} {\underline}

\newcommand {\virt} {\mathrm{virt}}

\newcommand {\T} {\mathfrak T}

\newcommand{\Bl} {\mathrm{Bl}}

\def\mydate{\ifcase\month \or January\or February\or March\or
April\or May\or June\or July\or August\or September\or October\or 
November\or December\fi \space\number\day,\space\number\year}

\begin{document}

\title[The Higher Dimensional Tropical Vertex]{The Higher Dimensional Tropical Vertex}
\author{H\"ulya Arg\"uz}
\address{Institute of Science and Technology Austria, Klosterneuburg \\3400, Austria}
\email{huelya.arguez@ist.ac.at}

\author{Mark Gross}
\address{Department of Mathematics\\University of Cambridge\\Cambridge, CB3 0WB\\United Kingdom}
\email{mgross@dpmms.cam.ac.uk }

\date{\today}

\begin{abstract}
We study log Calabi-Yau varieties obtained as a blow-up of a toric variety along hypersurfaces in its toric boundary. Mirrors to such varieties are constructed by Gross-Siebert from a canonical scattering diagram built by using punctured  Gromov-Witten invariants of Abramovich-Chen-Gross-Siebert. We show that there is a piecewise linear isomorphism between the canonical scattering diagram and a scattering diagram defined algorithmically, following a higher dimensional generalisation of the Kontsevich-Soibelman construction. We deduce that the punctured  Gromov-Witten invariants of the log Calabi-Yau variety can be captured from this algorithmic construction. This generalizes previous results of Gross-Pandharipande-Siebert on ``The Tropical Vertex'' to higher dimensions. As a particular example, we compute these invariants for a non-toric blow-up of the three dimensional projective space along two lines.
\end{abstract}
\maketitle

\tableofcontents
\section{Introduction}
\label{sec:intro}
\subsection{Background}
A log Calabi-Yau pair $(X, D)$ is an $n$-dimensional smooth 
variety $X$ over an algebraically closed field $\kk$ of characteristic zero together with a reduced simple normal crossing divisor $D$ in $X$ with $K_X+D=0$. In this article we study log Calabi--Yau pairs $(X,D)$
where $X$ is a projective variety
obtained by a blow-up
\begin{equation}
\label{Eq: blow up}
    X \longrightarrow X_{\Sigma}
\end{equation}
of a toric variety $X_{\Sigma}$ associated to a complete fan $\Sigma$ in $\RR^n$, and $D\subset X$ is given by the strict transform of the toric boundary divisor. We assume that the center of the blow-up is a union of disjoint general smooth hypersurfaces
\begin{equation}
    \label{Eq: hypersurfaces H}
  H=H_1\cup\cdots\cup H_s 
\end{equation}
where $H_i \subset D_{\rho_i}$, for the toric divisor $D_{\rho_i}$ corresponding to a choice of ray $\rho_i\in\Sigma$.

Log Calabi--Yau pairs are very widely studied in the mathematics literature, and they attract broad interest, especially from the point of view of mirror symmetry, see e.g., \cite{gross2015mirror,donagi2012weak,kollar2016dual,yu2016enumeration,pascaleff2019symplectic,takahashi2001log,GHS}. 
In particular, they play significant role in the Gross--Siebert program \cite{GS2}, which is aimed at an algebro-geometric understanding
of mirrory symmetry motivated
by the Strominger--Yau--Zaslow conjecture \cite{SYZ}.
This conjecture, roughly, suggests that mirror pairs of Calabi--Yau varieties admit dual (singular) torus fibrations. The dual to a torus fibration on a Calabi--Yau is obtained by a compactification of a semi-flat mirror that is constructed by dualizing the non-singular torus fibers. To obtain this compactification one considers appropriate corrections to the complex structure. It was shown by Kontsevich and Soibelman \cite{KS} in dimension two and by Gross and Siebert \cite{GS2} in higher dimensions  that the problem to determine these corrections could be reduced to a combinatorial algorithm encoded in a scattering diagram. In dimension two, Gross, Pandharipande and Siebert \cite{GPS} then showed that this 
combinatorial algorithm, producing a \emph{scattering diagram},
could be interpreted in terms of relative invariants of pairs $(X,D)$
as described above, with $X$ the blow-up of a toric surface.  

In the two-dimensional case, Gross, Hacking and Keel \cite{GHK}
then used \cite{GPS} to construct
the homogeneous coordinate ring for the mirror to a log Calabi--Yau surface $X\setminus D$.
This led Gross and Siebert to a general algebro-geometric mirror construction
in all dimensions \cite{GSUtah, GSAssoc, GSCanScat}. 

Before elaborating on the notion of a scattering diagram associated to pairs
$(X,D)$, we will describe more general versions of these relative invariants which are needed in higher dimension. These are punctured 
Gromov--Witten invariants constructed by Abramovich, Chen, Gross and Siebert \cite{ACGSII}.

\subsubsection{Punctured Gromov--Witten invariants.}
It is expected that the necessary corrections of the complex structure describing the mirror to $(X,D)$ can be obtained by counts of certain holomorphic discs with boundary on a fiber of the SYZ fibration \cite{SYZ,fukayamultivalued}. A key motivation in the approach to mirror symmetry we are interested in is to interpret analogues of such holomorphic disks algebro-geometrically. One of the main contributions in this direction is provided in \cite{GPS} showing that scattering diagrams carry enumerative information, and capture \emph{log Gromov--Witten invariants}. These invariants are obtained as counts of rational algebraic curves with prescribed tangency conditions relative to a (not necessarily smooth) divisor. This is in some sense a generalisation of relative Gromov--Witten theory developed in \cite{JunLi,li1998symplectic,ionel2003relative}, where tangency conditions are imposed relative to smooth divisors. Although in the two-dimensional case these invariants captured by scattering diagrams suffice to provide the necessary geometric data to construct mirrors to log Calabi--Yau surfaces \cite{GHK}, in higher dimensions, due to the more complicated nature of the structure in scattering diagrams, one needs to use a generalization of these invariants provided by punctured  Gromov-Witten theory \cite{ACGSII}. In this set-up, one allows negative orders of tangency 
at certain marked points, referred to as punctured points. For a precise
review of the definition of a punctured point, see 
\S\ref{Subsec: punctured}. Roughly, if a punctured map
$f:C\rightarrow X$ has negative order of tangency with a divisor
$D\subseteq X$ at a punctured point of $C$, then the irreducible
component of $C$ containing this punctured point must map into $D$, and the negative
order of tangency is encoded in the data of the logarithmic structure.

In general, to define a punctured invariant,
one considers a combinatorial type $\tau$ of tropical map to
the tropicalization of $X$, and a curve class $\ul{\beta}$.
With $\tilde\tau=(\tau,\ul{\beta})$, one can define a 
moduli space $\scrM(X,\tilde\tau)$ of punctured maps of class
$\ul{\beta}$ whose type is a degeneration of $\tau$.
The type $\tau$ in particular specifies the genus, the number of punctured
points and their contact orders with the components of $D$. This moduli space 
carries a virtual fundamental class. The precise definitions are
via the technology of \cite{ACGSII} and punctured maps to the
associated Artin fan \cite{abramovich2016skeletons}. 

The special case we are interested in here, still with $(X,D)$ a
log Calabi-Yau pair, is when $\tau$ is the type of a genus zero
tropical map with one puncture with fixed contact order.  When $\tau$ is
the type of a $(\dim X-2)$-dimensional family of tropical curves, 
the moduli space $\scrM(X,\tilde\tau)$ is virtually zero dimensional
and hence carries a virtual degree, which we write as $N_{\tilde\tau}$,
following the notation of \cite{GSUtah}. (We note that \cite{GSCanScat}
instead uses the notation $W_{\tilde\tau}$.) This provides the main ingredient to define the canonical scattering diagram, which was defined for higher dimensional log Calabi-Yau pairs in \cite{GSUtah} and whose properties were
determined in \cite{GSCanScat}. Although it is generally a challenging task to compute these virtual moduli spaces  $\scrM(X,\tilde\tau)$ and extract the numbers $N_{\tilde\tau}$ corresponding to punctured invariants, the key point in this paper is that these numbers can be captured from the combinatorics of scattering diagrams purely
algorithmically. 
\subsubsection{The canonical scattering diagram}
Associated to a log Calabi--Yau pair $(X,D)$ is its tropicalization $(B,\P)$, a polyhedral complex defined similarly as in the two dimensional case in \cite[\S1.2]{GHK}. This polyhedral complex carries the structure of an integral affine manifold with singularities, with singular locus $\Delta\subset B$. 
We fix a submonoid $Q\subset N_1(X)$ containing all effective curve classes, where $N_1(X)$ denotes the abelian group generated by projective irreducible curves in $X$ modulo numerical equivalence \cite[Defn 1.8]{GHS}. The canonical scattering diagram associated to $(X,D)$ is then given by pairs
\[ \foD_{(X,D)} := \{(\fod,f_{\fod})\} \]
of walls $\fod \subset B$, along with attached functions $f_{\fod}$ that are elements of the completion of $\kk[\shP_x^+]$ at the ideal generated by $Q\setminus\{0\}$, where $\shP_x^+=\Lambda_x\times Q$, 
$x\in \Int{\fod}$ is a general point and $\Lambda$ is the local
system of integral vector fields on $B\setminus\Delta$. These functions $f_{\fod}$ are concretely given by
\[ f_{\fod}= \exp(k_{\tau}N_{\tilde\tau}t^{\ul{\beta}} z^{-u}) \] 
where $\tilde\tau=(\tau,\ul{\beta})$ ranges over types of $\dim X-2$-dimensional families
of tropical maps associated to punctured curves with only one puncture, as described previously. Here $t^{\ul{\beta}}z^{-u}$ denotes
the monomial in $\kk[\shP_x^+]$ associated to $(-u,\ul{\beta})$.
The wall $\fod$ is swept out by the image of the leg of the
type $\tau$ corresponding to the punctured point, hence is determined
purely combinatorially from $\tau$. The contact order of the punctured point is recorded in the tangent vector $u\in \Lambda_x$ and $k_{\tau}$ is a positive integer depending
only on the tropical type $\tau$ \cite[\S2.4]{GSUtah} or 
\cite[(3.10)]{GSCanScat}.

\subsection{Main results}
In this article we show that punctured log invariants can be captured combinatorially from scattering diagrams. For this, we connect the canonical scattering diagram of $(X,D)$, defined using punctured  Gromov-Witten invariants of $X$, where $X$ is as in \eqref{Eq: blow up}, to an algorithmically constructed toric scattering diagram associated to $X_{\Sigma}$ and $H$. Here $H$ is given as in \eqref{Eq: hypersurfaces H} by hypersurfaces $H_i \subset D_{\rho_i}$. We denote the primitive generator of $\rho_i$ by $m_i$, and write the decomposition of $H_i$ into connected components as \[H_i=\bigcup_{j=1}^{s_i} H_{ij}.\] 

We define the toric scattering diagram $\foD_{(X_{\Sigma},H)}$ consisting of codimension one closed subsets $\fod \subset \RR^n$, 
referred to as walls, along with attached functions $f_{\fod}$. 
Here $\RR^n$ is
the vector space in which the fan $\Sigma$ lives.
To define these functions, first take
\begin{equation}
\nonumber
    P = M\oplus \bigoplus_{i=1}^s\NN^{s_i},
\end{equation}
where $M$ is the cocharacter lattice associated with $X_{\Sigma}$.
Let $P^{\times}$ be the group of units of $P$, and $\fom_P=P\setminus P^{\times}$. Denote by $\widehat{\kk[P]}$ the completion of $\kk[P]$ with respect to
$\fom_P$. We
write $e_{i1},\ldots,e_{is_i}$ for the generators of $\NN^{s_i}$, and $t_{ij}:=z^{e_{ij}}\in \widehat{\kk[P]}$. We describe the initial scattering diagram $\foD_{(X_{\Sigma},H),\mathrm{in}}$ whose walls are given by the tropical hypersurfaces associated to $H_{ij}$, forming the walls. We then set the attached functions to be powers of
\[  f_{ij} := 1+t_{ij}z^{m_i} \]
determined by the weights of the tropical hypersurface corresponding
to $H_{ij}$. See \S\ref{Subsec: scattering in MR} for more details.

For each wall $\fod$, we specify a primitive $m_0\in M\setminus\{0\}$ tangent to $\fod$,
called the direction of the wall. We say the wall is \emph{incoming} if $\fod=\fod-\RR_{\ge 0} m_0$. Moreover, for a sufficiently general path, we define the standard notion of a path ordered product by considering the composition of all the automorphisms attached to the walls $\fod$ which are crossed by this path. These automorphisms are determined from the attached functions $f_{\fod}$. A scattering diagram is said to be \emph{consistent} if all path ordered products around loops are the identity. We prove the higher dimensional analogue of the Kontsevich--Soibelman Lemma \cite{KS} in Theorem \ref{thm:higher dim KS}:
\begin{theorem}
There is a consistent scattering diagram
$\Scatter(\foD)=\foD_{(X_{\Sigma},H)}$ containing $\foD_{(X_{\Sigma},H),\mathrm{in}}$ such that $\foD_{(X_{\Sigma},H)}\setminus
\foD_{(X_{\Sigma},H),\mathrm{in}}$ consists only of non-incoming walls. Further, this scattering diagram is unique up to equivalence.
\end{theorem}
We note that in theory, this statement can be extracted from \cite{GS2},
but because of the complexity of the global arguments used in \cite{GS2},
it is essentially impossible to extract a precise reference, and we
find it better to present an independent proof here.

We then show that the associated consistent scattering diagram $\foD_{(X_{\Sigma},H)}$ indeed captures punctured log invariants. For this we compare $\foD_{(X_{\Sigma},H)}$ with the canonical scattering diagram associated to $(X,D)$. It follows essentially from the definition of the tropicalization of
$(X,D)$ that there is a natural piecewise-linear isomorphism
\[
\Upsilon:(M_{\RR},\Sigma)\rightarrow (B,\P).
\]
This allows us to compare the scattering diagrams $\foD_{(X_{\Sigma},H)}$
and $\foD_{(X,D)}$ as follows. 
We first introduce some additional notation.
Let $E_i^j$ denote an exceptional
curve of the blow-up \eqref{Eq: blow up} over the hypersurface $H_{ij}$. There is a natural splitting
$N_1(X)=N_1(X_{\Sigma})\oplus\bigoplus_{ij} \ZZ E_i^j$ in which
$N_1(X_{\Sigma})$ is identified with the set of curve classes
in $N_1(X)$ with intersection number zero with all exceptional
divisors.

A wall of $\foD_{(X_{\Sigma},H)}$ is necessarily of the form
\[
\left(\fod, f_{\fod}\bigg(\prod_{ij} (t_{ij}z^{m_i})^{a_{ij}}\bigg)\right)
\]
for some non-negative integers $a_{ij}$, where the notation $ f_{\fod}\bigg(\prod_{ij} (t_{ij}z^{m_i})^{a_{ij}}\bigg)$ indicates that $f_{\mathfrak{d}}$ is a power-series in the 
expression $\prod_{i,j}(t_{ij}z^{m_i})^{a_{ij}}$,
for some collection of exponents $a_{ij}\in\NN$.
We will define $\Upsilon(\fod,f_{\fod})$, a wall on $B$. The definition
depends on whether $(\fod,f_{\fod})$ is incoming or not. 
In any event,
we may assume, after refining the walls of $\foD_{(X_{\Sigma},H)}$,
that each wall of this scattering diagram is contained in some cone
$\sigma\in\Sigma$.

If the wall is incoming, then by construction of $\foD_{(X_{\Sigma},H)}$
it is of the form
$(\fod,(1+t_{ij}z^{m_i})^{w_{ij}})$ for some positive integer
$w_{ij}$. As $m_i$ is tangent to the cone of $\Sigma$ containing
$\fod$ and $\Upsilon$ is piecewise linear with respect to $\Sigma$,
$\Upsilon_*(m_i)$ makes sense as a tangent vector to $B$.
We then define
\[
\Upsilon(\fod,(1+t_{ij}z^{m_i})^{w_{ij}})
=(\Upsilon(\fod),(1+t^{E_i^j}z^{-\Upsilon_*(m_i)})^{w_{ij}}).
\]
If the wall is not incoming, with $\fod\subseteq\sigma\in\Sigma$, then
the data $\mathbf{A}=\{a_{ij}\}$ and $\sigma$ determine a curve class 
$\bar\beta_{\mathbf{A},\sigma}\in N_1(X_{\Sigma})$ -- see \S\ref{sec: at infinity} for the precise details of this curve
class. Under the inclusion $N_1(X_{\Sigma}) \hookrightarrow N_1(X)$ 
given by the above mentioned splitting, we may view
$\bar\beta_{\mathbf{A},\sigma}$ as a curve class in $N_1(X)$, which we also denote by $\bar\beta_{\mathbf{A},\sigma}$. We then obtain a curve class
\[
\beta_{\mathbf{A},\sigma}=\bar\beta_{\mathbf{A},\sigma} - \sum_{ij} a_{ij} E_i^j.
\]
Further, as $m_{\out}:=-\sum_{ij} a_{ij}m_i$ is tangent to the
cone of $\Sigma$ containing $\fod$, as before $\Upsilon_*(m_{\out})$ makes sense
as a tangent vector to $B$. We may thus define the wall
\begin{equation}
\label{Eq: wall upsilon canonical}
\Upsilon(\fod,f_{\fod})=(\Upsilon(\fod),
f_{\fod}(t^{\beta_{\mathbf{A},\sigma}}z^{-\Upsilon_*(m_{\out})})).
\end{equation}
We then define 
\[
\Upsilon(\foD_{(X_{\Sigma},H)}) :=
\{\Upsilon(\fod,f_{\fod})\,|\, (\fod,f_{\fod})\in \foD_{(X_{\Sigma},H)}\}.
\]
To compare this scattering diagram with the canonical scattering diagram $\foD_{(X,D)}$ we consider a degeneration $(\widetilde X, \widetilde D)$ obtained from a blow-up of the degeneration to the normal cone of $X_{\Sigma}$, with general fiber $(X,D)$. We view the canonical scattering diagram $\foD_{(X,D)}$ as a scattering diagram that is embedded into $\foD_{(\widetilde X, \widetilde D)}$, and denote it by $\iota(\foD_{(X,D)})$ in \eqref{Eq: ioata applied to canonical scattering}. 

Our main result, Theorem \ref{Thm: HDTV}, then states:

\begin{theorem}
$\Upsilon(\foD_{(X_{\Sigma},H)})$ is equivalent to $\foD_{(X,D)}$.
\end{theorem}
 Here two scattering diagrams are equivalent if they induce the same
wall-crossing automorphisms. To prove Theorem \ref{Thm: HDTV}, we consider a degeneration $(\widetilde X, \widetilde D)$ obtained from a blow-up of the degeneration to the normal cone of $X_{\Sigma}$, with general fiber $(X,D)$. This generalizes a degeneration used in \cite{GPS}.
We investigate the canonical scattering diagram associated to $(\widetilde X, \widetilde D)$, which has support in the tropicalization $\widetilde B$
of $(\widetilde X,\widetilde D)$. This tropicalization comes naturally
with a projection map
\[ \widetilde p: \widetilde B \to \RR_{\geq 0}.  \]
Hence, we obtain a scattering diagram $\foD^1_{(\widetilde X, \widetilde D)}$ supported on $\widetilde B_1:= \widetilde p^{-1}(1)$, which is an integral affine manifold with singularities away from the origin. 
Effectively, this process pulls apart the singular locus of
$B$ into pieces of singular locus which take a standard form.
As first hinted at in \cite[\S4]{GSICM} but carried out rigorously for
the first time here, 
the monodromy around these standard pieces of singular locus go a long
way to determining the structure of the canonical scattering diagram, see Lemmas \ref{Lem:functional equation}, \ref{lem: multiple covers}.
Localizing to the origin $0\in \widetilde B_1$ we obtain a scattering diagram 
\begin{equation}
T_0\foD^{1}_{(\widetilde X,\widetilde D)}:=\{(T_0\fod, f_{\fod})\,|\,
(\fod,f_{\fod})\in \foD^{1}_{(\widetilde X,\widetilde D)}, \quad 0\in \fod\} 
\end{equation}
defined in \eqref{Eq: asymptotic canonical }, where $T_0\widetilde B_1$ denotes the tangent space to $0\in \widetilde B_1$. 
We then define a scattering diagram $\nu(\foD_{(X_{\Sigma},H)})$ in \eqref{Eq: nu of scattering diagram}, where $\nu$ is induced by a linear isomorphism 
$M_{\RR} \to T_0\widetilde B_1$.  One of the key results is Theorem \ref{thm:key theorem}, which states that the scattering diagram $T_0\foD_{(\widetilde X,\widetilde D)}^1$ is equivalent to 
$\nu(\foD_{(X_{\Sigma},H)})$.

At a second step we consider the asymptotic scattering diagram $\foD_{(\widetilde X,\widetilde D)}^{1,\as}$, given by Definition \ref{def:asymptotic scattering}. Proposition \ref{Prop: Asymptotic equivalence} shows that $\foD_{(\widetilde X,\widetilde D)}^{1,\as}$ is equivalent to $\iota(\foD_{(X,D)})$. In \S \ref{sec: at infinity} we then define a piecewise-linear isomorphism 
\[ \mu: M_{\RR} \lra \widetilde B_0= \widetilde{p}^{-1}(0) \cong B\]
which induces the equivalence of scattering diagrams $\mu(T_0\foD^{1}_{(\widetilde X,\widetilde D)})$ and $\foD^{1,\as}_{(\widetilde X, \widetilde D)}$, and hence $\iota(\foD_{(X,D)})$. The map $\Upsilon$ is then given by the composition $\Upsilon= \mu \circ \nu$.

Using Theorem \ref{Thm: HDTV}, we capture the punctured  Gromov-Witten invariants of $X$ appearing in the canonical scattering diagram $\foD_{(X,D)}$, combinatorially from the toric scattering diagram $\foD_{(X_{\Sigma},H)}$. As a concrete example we analyse the punctured invariants of the log Calabi--Yau obtained from a non-toric blow-up of $\PP^3$ with center the union of two disjoint lines. The associated canonical scattering diagram in this case has infinitely many walls. However, after appropriate cancellations, we obtain a minimal scattering diagram equivalent to it. By the aid of computer algebra, in joint work with Tom Coates \cite{ACG} we obtain punctured log invariants of log Calabi--Yau varieties as well as explicit equations of the coordinate ring to Fano varieties in higher dimensions in greater generality. 
\subsection{Related work} 
In \cite{keel2019frobenius}, Keel and Yu obtain some analogous results
using non-Archimedean techniques. In particular, 
there the authors work with log Calabi-Yau pairs $(X,D)$ such that
$X\setminus D$ is affine and contains a copy of a $\dim X$-dimensional
torus. They construct a scattering diagram, not using punctured invariants,
but by detecting the influence of walls on broken lines. In the cluster variety case, when $(X,D)$ is obtained
as in our situation by blowing up a collection of hypertori in irreducible
components of $D$, it follows from \cite[\S 22]{keel2019frobenius} that this scattering diagram agrees with the canonical scattering diagram after pulling singularities to infinity as explained in \S \ref{Sec: Pulling singularities out}. In the cluster case if the power series expansion of the function $f_{\fod}$ attached to a wall $\fod$ of the minimal canonical scattering diagram associated to a log Calabi--Yau $(X,D)$ is given by
\[ f_{\fod} = \sum_m c_m z^m \]
as in \eqref{Eq: canonical fncs}, then in \cite{keel2019frobenius} it is shown that the coefficients $c_m$ are always non-negative integers which are counts of non-Archimedean cylinders. In our paper we give a direct enumerative interpretation of the coefficients of $\mathrm{log}(f_{\fod})$ in terms of punctured log invariants which are in general rational numbers (see \eqref{Eq: walls of canonical}). So, comparing these two interpretations, in the cluster case we obtain non-trivial relations between the counts of non-Archimedean cylinders and punctured log invariants. The integrality of the coefficients $c_m$ however is not geometrically clear from the point of view of punctured log invariants.

In \cite{bousseau2018quantum}, Bousseau shows that in the two dimensional case a refined version of the toric scattering diagram captures higher genus log Gromov-Witten
invariants of log Calabi-Yau surfaces as well. It is a natural question to ask if there is a higher genus version of our genus $0$ results.

Earlier work of Mandel \cite{mandelscattering} explored the
tropical enumerative interpretation of higher dimensional scattering diagrams
in the case arising from cluster varieties (corresponding here to only
blowing up hypertori). His interpretation was analogous to that given
in \cite[\S2]{GPS}. This description has connections with DT theories
\cite{arguz2021flow,mandelcheung}, and thus combining our results with
theirs should give a description of certain punctured invariants in
terms of a sum over certain tropical curves which in some cases also correspond to DT invariants.
Also related is the work of \cite{mandel2019theta}, where Mandel  
makes use, in
the cluster case, to the same type of degeneration used here to analyze
the Gromov-Witten invariants arising in the product rule for
theta functions.
\medskip

\emph{Acknowledgements:} This project had its genesis in discussions
with Rahul Pandharipande and Bernd Siebert in 2011 aiming to generalize the
results of \cite{GPS} to higher dimension. However, it became clear
at that time that the necessary Gromov-Witten technology had not yet been
developed, and the project became dormant.
The authors of this paper returned to that project in many useful conversations
with Tom Coates, who provided important help with the computational tools
that were necessary to understand concrete examples. We thank all of these collaborators, as well as Paul Hacking and Sean Keel, whose collaboration
with the second author led to the search for a canonical scattering
diagram in all dimensions.
In addition, we would like to thank Pierrick Bousseau for several insightful discussions. During the preparation of this project
H.A. received funding
from the European Research Council (ERC) under the European Union’s Horizon 2020
research and innovation programme (grant agreement No. 682603) and from Fondation
Math\'ematique Jacques Hadamard, and M.G.\ would like to acknowledge support of EPSRC grant
EP/N03189X/1 and a Royal Society Wolfson Research Merit Award. 


\section{Canonical scattering for log Calabi--Yau varieties}
\label{Sec: canonical scattering}
In this section we review the construction of the canonical scattering diagram,
as announced in \cite{GSUtah}, and developed in \cite{GSCanScat},
where it is called the canonical wall structure.
 This is a combinatorial gadget which provides a recipe for the construction of the coordinate ring of the mirror of a log Calabi-Yau pair $(X, D)$. 
\subsection{Tropical log Calabi--Yau spaces}

\subsubsection{The affine manifold $(B,\P)$}
\label{subsubsection:affine manifold}
For this discussion, we fix an $n$-dimensional log Calabi--Yau pair $(X,D)$.
We distinguish
between two cases, the \emph{absolute case}, where $X$ is assumed
to be projective over $\kk$, and the \emph{relative case},
where we are given a projective morphism 
\[ p:X \longrightarrow \AA^1 \]
with $p^{-1}(0)
\subset D$. We further assume in this case that $p$ is
a normal crossings morphism, that is, that \'etale locally on $X$, $p$ is given
by 
\begin{eqnarray}
\nonumber
\AA^n & \longrightarrow & \AA^1 \\
\nonumber
(x_1,\ldots,x_n) & \longmapsto & \prod_i x_i^{a_i}
\end{eqnarray}
with $a_i\ge 0$. Further, in this local description,
there exists an index set $J\subseteq\{1,\ldots,n\}$ with
\[D=V(\prod_{j\in J} x_j)\] 
and $p^{-1}(0)\subseteq D$. 
In the language of logarithmic geometry
used later, this makes $p$ a log smooth morphism.

We define the \emph{tropicalization} $(B,\P)$ of $(X,D)$, which we refer to as the \emph{tropical space associated to $(X,D)$}, as a polyhedral cone
complex as follows. Let $\Div(X)$ denote the group of divisors on $X$, and $\Div_D(X)\subseteq \Div(X)$ be the subgroup of divisors supported
on $D$, and
\[\Div_D(X)_{\RR}=\Div_D(X)\otimes_{\ZZ}\RR.\] 
Let $D=\bigcup_{i=1}^m D_i$ be the decomposition of $D$ into
irreducible components, and write $\{D_i^*\}$ for the dual basis of
$\Div_D(X)_{\RR}^*$. We assume throughout that for any index subset
$I\subseteq \{1,\ldots,m\}$, if
non-empty, $\bigcap_{i\in I} D_i$ is connected. Define $\P$ to be the collection of cones
\begin{equation}
    \label{Eq: polyhedral decomposition}
    \P:=\left
\{\sum_{i\in I} \RR_{\ge 0} D_i^*\,|\, \hbox{$I\subseteq \{1,\ldots,m\}$
such that $\bigcap_{i\in I}D_i\not=\emptyset$}\right\}.
\end{equation}
Then we have
\[
B:=\bigcup_{\tau\in\P}\tau \subseteq \Div_D(X)^*_{\RR}.
\]
In what follows, we only consider pairs $(X,D)$ which are
\emph{maximally degenerate}, that is, for which $\dim_{\RR} B=\dim X=n$.
This means there exists an index set $I$ such that $\bigcap_{i\in I} D_i \neq \emptyset$ 
and $\dim\bigcap_{i\in I} D_i=0$. 

Given any cone $\rho\in\P$, we can write $\rho=\sum_{i\in I} \RR_{\ge 0}
D_i^*$ for some index set $I$, and then 
\begin{equation}
\label{eq:Drho def}
D_{\rho}:=\bigcap_{i\in I} D_i
\end{equation}
is the stratum of $(X,D)$ corresponding to $\rho$. Hence, the polyhedral decomposition $\P$ induces a natural stratification on $X$, given by all the strata corresponding to cones in $\P$.

We continue with the notational conventions of \cite{GHS} from now on.
More generally in the sequel, we will be working with pairs $(B,\P)$
where $B$ is a topological space of dimension $n$ 
and $\P$ is a polyhedral decomposition of $B$.
We refer to cells of dimensions $0$, $1$ and $n$ as
\emph{vertices}, \emph{edges} and \emph{maximal cells}. The set of $k$-cells are denoted by $\P^{[k]}$ and we write $\P^\max:=\P^{[n]}$ for the set of maximal cells.
A cell
$\rho\in\P^{[n-1]}$ only contained in one maximal cell is said to lie on the \emph{boundary} of $B$, and we let $\partial B$ be the union of all $(n-1)$-cells lying on the boundary of $B$. Any cell of
$\P$ contained in $\partial B$ is called a \emph{boundary cell}.
Cells not contained in $\partial B$ are called \emph{interior},
defining $\mathring{\P}\subseteq \P$. Thus $\P_\partial:=\P\setminus\mathring{\P}$
is the induced polyhedral decomposition of $\partial B$. If $\tau\in\P$,
we denote the \emph{(open) star} of $\tau$ to be
\begin{equation}
\label{eq:star def}
\mathrm{Star}(\tau):=\bigcup_{\tau\subseteq\rho\in\P}\Int(\rho),
\end{equation}
where $\Int(\rho)$ denotes the interior of the cell $\rho$. 

For $\sigma\in\P$, we denote by $\Lambda_{\sigma}$ the lattice of
integral tangent vectors to $\sigma$. We do so similarly for any
rational polyhedral subset of $\sigma\in\P$. 

 The tuple $(B,\P)$, consisting of $B$ endowed with the polyhedral decomposition $\P$ defined in \eqref{Eq: polyhedral decomposition}, can be given
the structure of an affine manifold with singularities, following
\cite[\S2.4]{GSUtah}, \cite[\S1.3]{GSCanScat} and \cite{nicaise2019non}.
Recall that an \emph{affine manifold} is a topological manifold together with an atlas whose transition maps lie in the group of affine transformations  
\[
\Aff(M_\RR) := M_\RR \rtimes GL_n(\RR),
\] 
of $M_{\RR}=M\otimes_{\ZZ}\RR$ for a fixed lattice $M\cong\ZZ^n$. We refer to such an atlas as an \emph{affine structure}. We call an affine manifold \emph{integral} if the change of coordinate transformations lie in 
\[
\Aff(M) := M \rtimes GL_n(\ZZ).
\] 
An (integral) \emph{affine manifold with singularities} is a topological manifold $B$ which admits an (integral) affine structure on a subset $B\setminus \Delta$,
where $\Delta \subset B$ is a union of submanifolds of $B$ of codimension at least $2$. The union of these submanifolds is called the {\it discriminant locus} of $B$. 

In our situation, when $(B,\P)$ is the tropicalization of $(X,D)$, 
we take the discriminant locus
$\Delta\subseteq B$ to be the union of codimension $\ge 2$ cones of $\P$.
We describe an integral affine structure on $B\setminus
\Delta$ as follows. For every codimension
one cone $\rho$ contained in two maximal cones $\sigma$ and $\sigma'$,
we give
an affine coordinate chart on $\Int(\sigma\cup\sigma')=\mathrm{Star}(\rho)$,
induced by a piecewise linear embedding
\[
\psi_{\rho}:\sigma\cup\sigma'\rightarrow \RR^n.
\]
This embedding is chosen to satisfy the following conditions. 
Let $m_1,\ldots,m_{n-1}$ be the primitive generators of the
edges $\tau_1,\ldots,\tau_{n-1}$ of $\rho$, and let
$m_{n}, m_{n}'$ be the
primitive generators of the additional edges $\tau_{n},
\tau_{n}'$ of $\sigma$
and $\sigma'$ respectively. 
We then require that $\psi_{\rho}$ identifies
$\sigma\cup\sigma'$ with the support of a fan
$\Sigma_{\rho}$ in $\RR^n$, consisting of cones $\psi_{\rho}(\sigma),
\psi_{\rho}(\sigma')$ and their faces. Further, we require that
$\psi_{\rho}(\sigma),\psi_{\rho}(\sigma')$ are standard cones with integral
generators
\[  \{ \psi_{\rho}(m_1),\ldots,\psi_{\rho}(m_{n}) \} \,\ \,\ \mathrm{and} \,\ \,\ \{ \psi_{\rho}(m_1),
\ldots,\psi_{\rho}(m_{n-1}),\psi_{\rho}(m_{n}') \} \]
respectively. Finally, we require that
\begin{equation}
\label{eq:chart relation general}
\psi_{\rho}(m_{n})+\psi_{\rho}(m_{n}')=
- \sum_{j=1}^{n-1} \left( D_{\tau_j} \cdot D_{\rho}\right) \psi_{\rho}(m_j).
\end{equation}
Here $\psi_{\rho}$ is determined by \eqref{eq:chart relation general}
up to elements of $\GL_n(\ZZ)$. 

\begin{remark}
\label{Rem: Motuvation for toricy definition}
The motivation for \eqref{eq:chart relation general} 
is that the fan $\Sigma_{\rho}$ has the property that
in the toric variety $X_{\Sigma_{\rho}}$, 
\[
D_{\psi_{\rho}(\tau_j)}\cdot D_{\psi_{\rho}(\rho)}=
D_{\tau_j}\cdot D_{\rho},
\]
where on the left we write $D_{\omega}$ for the closed stratum
of $X_{\Sigma_{\rho}}$ corresponding to $\omega\in \Sigma_{\rho}$.
This follows from \cite{Oda}, pg.\ 52.
\end{remark}

\begin{remark}
\label{rem:relative case}
In the relative case, the morphism $p:X\rightarrow\AA^1$ induces a canonical map 
\[p_{\mathrm{trop}}: B\rightarrow
\RR_{\ge 0}.\] This arises functorially  from the logarithmic point of view introduced later. However, for now, we may view this as arising from the description
of $B$ as a subset of $\Div_D(X)^*_{\RR}$. We view the fibre
$X_0$ as a divisor on $X$ supported on $D$, hence
defining a functional on $\Div_D(X)^*_{\RR}$. Its restriction
to $B$ gives the map $p_{\mathrm{trop}}$. It is easy to see that this is an affine submersion where the affine structure of $B$ is defined, 
see \cite[Prop.~1.14]{GSCanScat}.

If $t\in \AA^1$ is general, then $(X_t,D_t)$ is also a log Calabi-Yau
variety with dual complex denoted as $(B_0,\P_0)$. 
It is easy to see that
$B_0=p_{\mathrm{trop}}^{-1}(0)$ and $\P_0=\{\sigma\cap p_{\mathrm{trop}}^{-1}(0)\,|\, \sigma\in\P\}$. 
Further, if $(X_t,D_t)$ is maximally degenerate, i.e., 
$\dim X_t=\dim B_0$, then $B_0=\partial B$,
and the affine structure on $B_0$ is
the natural affine structure inherited from the affine structure
on $B$.
See \cite[Prop.~1.18]{GSCanScat}.
\end{remark}

\medskip

While in the construction of the affine structure, we took
$\Delta$ to be the union of all codimension two cones of $\P$,
it will be an important point for us that sometimes the affine structure
on $B\setminus\Delta$ extends to a larger subset of $B$, i.e.,
that we only need to take $\Delta$ to be a union of \emph{some}
set of codimension $\ge 2$ cones of $\P$. Note in this case that if $\tau\in\P$
with $\tau\not\subseteq \Delta$, then $\mathrm{Star}(\tau)\cap\Delta=\emptyset$.  The following
will be useful for analyzing the structure of $B$ around $\tau$:
\begin{proposition}
\label{prop:local affine submersion}
Let $\tau\in \P$ be a cone satisfying the following conditions:
\begin{itemize}
    \item The corresponding stratum $D_{\tau}$ is toric with fan
$\Sigma_{\tau} \subseteq \RR^{d}$, where $d=\codim \tau$.
    \item The restriction of the stratification on $X$ induced by $\P$ to $D_{\tau}$ agrees with the stratification of $D_{\tau}$ into its toric strata. 
\end{itemize}
Then the affine structure on $\mathrm{Star}(\tau)\setminus\Delta$
extends to an affine structure on $\mathrm{Star}(\tau)$. Further,
there exists an affine submersion $\upsilon:\mathrm{Star}(\tau)\rightarrow \RR^d$
such that for each $\omega \in \Sigma_{\tau}$,  
\[ \upsilon^{-1}(\Int (\omega)) = \Int (\rho)  \]
for some element $\rho\in\P$ with $\tau\subseteq\rho$.
\end{proposition}

\begin{proof}
We will give an affine chart $\psi_{\tau}:\mathrm{Star}(\tau)
\rightarrow \RR^n$ as follows. First, with $n=\dim X$,
using \eqref{eq:Drho def}, we can write $D_{\tau}$ as the intersection of divisors
\[ D_{\tau}=\bigcap_{j=1}^{n-d} D_{i_j} \]
for some distinct indices $i_1,\ldots,i_{n-d}$. 
If $\tau\subseteq\rho\in\P$, then
by assumption $D_{\rho}$ is a toric strata of $D_{\tau}$, and there
is a corresponding cone $\rho_{\tau}\in\Sigma_{\tau}$. 
In particular, if $D_{\rho}$ is a divisor on $D_{\tau}$, let 
$m_{\rho}\in \ZZ^d$ be the primitive generator of the ray $\rho_{\tau}$.

Now as $\O_X(D_{i_j})|_{D_{\tau}}$ is a line bundle on $D_{\tau}$,
it is determined, up to a linear function, by 
a function $\psi_j:\RR^d\rightarrow \RR$ piecewise linear with
respect to $\Sigma_{\tau}$ and with integral slopes. 
Choose such a function for
each $j$. We may then define an embedding
\[
\psi_{\tau}:\bigcup_{\tau\subseteq\rho\in\P} \rho\rightarrow \RR^n
\]
as follows. Split $\RR^n=\RR^{n-d}\oplus\RR^d$, and
let $e_1,\ldots,e_{n-d}$ be the standard basis for $\RR^{n-d}$.
Then take
\[
\psi_{\tau}(D_{i_j}^*)=-e_j, \quad 1\le j\le n-d.
\]
Further, if $k$ is an index such that $D_k\cap D_{\tau}$ is a 
non-empty divisor $D_{\rho}$ on $D_{\tau}$, then define
\[
\psi_{\tau}(D_k^*)=\sum_{j=1}^{n-d} \psi_j(m_{\rho})e_j+(0,m_\rho),
\]
where for the last term, we are using the splitting of $\RR^n$.
We then extend $\psi_{\tau}$ linearly on each cone containing 
$\rho$.

It is easy to check that this map is injective and 
thus defines an affine structure on $\mathrm{Star}(\tau)$.
Restricting to $\mathrm{Star}(\tau)$ the 
composition
of $\psi_{\tau}$ with the projection onto the $\RR^d$ factor gives an affine submersion. We claim that this is the desired map $\upsilon$. For this it suffices to check that the affine structure defined by $\psi_{\tau}$
is compatible with the affine structure on $B$. To do this, fix cones $\sigma,\sigma'\in\P^{\max}$ with
$\tau\subseteq \rho=\sigma\cap\sigma'$, and $\rho\in\P^{[n-1]}$.
We have corresponding cones $\sigma_{\tau},\sigma'_{\tau},\rho_{\tau}
\in\Sigma_{\tau}$. Let $i_{n-d+1},\ldots,i_n$ be indices so that
$D_{i_1}^*,\ldots,D_{i_n}^*$ span $\sigma$, with
$D_{i_n}^*\not\in\sigma'$ and let $i_n'$ be an index
so that $D_{i_1}^*,\ldots,D_{i_{n-1}}^*,D_{i_n'}^*$
span $\sigma'$. Then by \eqref{eq:chart relation general},
we need to show that 
\[
\psi_{\tau}(D_{i_n}^*)+\psi_{\tau}(D_{i_n'}^*)
=-\sum_{j=1}^{n-1} (D_{i_j} \cdot D_{\rho})\psi_{\tau}(D_{i_j}^*).
\]

The right-hand side is 
\begin{equation}
    \label{Eq: RHS}
\sum_{j=1}^{n-d}\left(D_{i_j}\cdot D_{\rho} - \sum_{k=n-d+1}^{n-1}
(D_{i_k}\cdot D_{\rho})\psi_j(\upsilon(D_{i_k}^*))\right) e_j
-\sum_{j=n-d+1}^{n-1}(D_{i_j}\cdot D_{\rho})(0, \upsilon(D_{i_j}^*)).
\end{equation}
The left-hand side is 
\begin{equation}
    \label{Eq: LHS}
    \sum_{j=1}^{n-d} \big(\psi_j(\upsilon(D_{i_n}^*))+\psi_j(\upsilon(D_{i_n'}^*))
\big)e_j
+(0,\upsilon(D_{i_n}^*)+\upsilon(D_{i_n'}^*)).
\end{equation}
In the fan $\Sigma_{\tau}$ we have the equation, 
as in Remark \ref{Rem: Motuvation for toricy definition},
\[\upsilon(D_{i_n}^*)+\upsilon(D_{i_n'}^*)+\sum_{j=n-d+1}^{n-1}(D_{i_j}\cdot D_{\rho})\upsilon(D_{i_j}^*)=0,\] so the equality of the last terms in \eqref{Eq: RHS} and \eqref{Eq: LHS} follows immediately. The result then follows since
\[
\psi_j(\upsilon(D_{i_n}^*))+\psi_j(\upsilon(D_{i_n'}^*))+\sum_{k=n-d+1}^{n-1}
(D_{i_k}\cdot D_{\rho})\psi_j(\upsilon(D_{i_k}^*)) = \deg\O(D_{i_j})|_{D_{\rho}} = D_{i_j} \cdot D_{\rho} \]
by standard toric geometry.
\end{proof}

\begin{remark}
\label{rem:normal bundles}
The chart constructed in the proof of Proposition \ref{prop:local affine submersion} identifies the collection
of cones 
\[
\scrP_{\tau}:=\{\rho\in\P\,|\, \hbox{there exists $\omega\in\P$ with 
$\tau,\rho\subseteq\omega$}\}
\]
with a fan $\Sigma$ in $\RR^n$. For a cone $\rho\in\P_{\tau}$, write
$\psi_{\tau}(\rho)$ for the corresponding cone in $\Sigma$.
The relationship between $X$ and the toric variety $X_{\Sigma}$ is then
as follows. First, $D_{\tau}\cong D_{\psi_{\tau}(\tau)}$, the
stratum of $X_{\Sigma}$ corresponding to the cone $\psi_{\tau}(\tau)$.
Second, if $\rho\in\P_\tau$ is an edge, then
$\O_X(D_{\rho})|_{D_{\tau}}\cong 
\O_{X_{\Sigma}}(D_{\psi_{\tau}(\rho)})|_{D_{\psi_{\tau}(\tau)}}$.
This can be checked easily via the construction and standard facts
of toric geometry.
\end{remark}

\subsubsection{The MVPL function $\varphi$} \label{Subsec: the MVPL fnc }

Let us define $N_1(X)$ to be the abelian group generated by projective
irreducible curves in $X$ modulo numerical equivalence.
We fix a finitely generated, saturated submonoid $Q \subset N_1(X)$ 
containing all effective curve classes, such that 
\[Q^{\times}:= Q\cap (-Q)=\{0\}.\]
Let $\fom\subseteq \kk[Q]$ be the 
monomial ideal generated by monomials in $Q\setminus Q^{\times}$, 
and fix an ideal $I\subseteq\kk[Q]$ with $\sqrt{I}=\fom$.
We write $A_I:=\kk[Q]/I$. For $\beta \in Q$, we write $t^{\beta}\in
\kk[Q]$ for the corresponding monomial.

Another key piece of data is a \emph{multivalued
piecewise linear} (MVPL) function $\varphi$ on $B\setminus\Delta$ with values in $Q^{\gp}_{\RR}$, see \cite[Def.\ 1.8]{GHS}. This is determined by specifying a piecewise
linear function $\varphi_{\rho}$ on $\mathrm{Star}(\rho)$ for
each $\rho\in\mathring\P$ of codimension one, well-defined up to
linear functions.
Such a function is determined by specifying its kinks $\kappa_{\rho}
\in Q^{\gp}$ for each codimension one cone $\rho\in \mathring{\P}$, defined as follows \cite[Def.\ 1.6, Prop.\ 1.9]{GHS}. 
\begin{definition}
\label{def:kinky}
Let $\rho\in \mathring{\P}$ be a codimension one cone and let $\sigma,\sigma'$ be the two maximal cells
containing $\rho$, and let $\varphi_{\rho}$ be a piecewise linear
function on $\mathrm{Star}(\rho)\subset
B\setminus\Delta$. An affine chart at $x\in\Int\rho$ thus
provides an identification $\Lambda_\sigma=
\Lambda_{\sigma'}=:\Lambda_x$. Let $\delta:\Lambda_x\to\ZZ$ be the quotient by
$\Lambda_\rho \subseteq\Lambda_x$. Fix signs by requiring that $\delta$
is non-negative on tangent vectors pointing from $\rho$ into
$\sigma'$. Let $n,n'\in \check\Lambda_x\otimes Q^\gp$ be the slopes of
$\varphi_{\rho}|_\sigma$, $\varphi_{\rho}|_{\sigma'}$, respectively. Then
$(n'-n)(\Lambda_\rho)=0$ and hence there exists $\kappa_{\rho}\in Q^\gp$
with
\begin{equation}
\label{Eqn: kink}
n' -n =\delta \cdot\kappa_{\rho}.
\end{equation}
We refer to $\kappa_{\rho}$ as the \emph{kink} of $\varphi_{\rho}$ along $\rho$.
Thus if $\varphi$ is an MVPL function, it has a well-defined kink
$\kappa_{\rho}$ for each such $\rho$, and these kinks determine $\varphi$.
\end{definition}
We assume in what
follows that $\varphi$ is \emph{convex}, i.e., $\kappa_\rho
\in Q\setminus \{0\}$ for all $\rho$. The choice of $\varphi$ gives rise (see \cite[Def.\ 1.15]{GHS})
to a local system $\shP$ on $B\setminus\Delta$ fitting into an exact sequence
\begin{equation}
\label{eq:Q exact}
0\rightarrow \ul{Q}^{\gp}\rightarrow \shP\rightarrow \Lambda\rightarrow 0.
\end{equation}
Here $\ul{Q}^{\gp}$ is the constant sheaf with stalk $Q^{\gp}$, while
$\Lambda$ is the sheaf of integral tangent vectors on $B\setminus\Delta$.
For an element $m\in \shP_x$, we write $\bar m\in\Lambda_x$ for its image
under the projection of \eqref{eq:Q exact}. 
We may give an alternative description of $\shP$ as in \cite[Const.\ 2.2]{GHK}.
If $\rho\in\mathring{\scrP}$ is codimension one, 
choose a local representative $\varphi_{\rho}:\mathrm{Star}(\rho)\rightarrow
Q^{\gp}_{\RR}$ for $\varphi$, that is, $\varphi_{\rho}$ is a single-valued 
piecewise linear function with kink $\kappa_{\rho}$.
If $\rho,\sigma\in\mathring{\P}$ with $\rho$ codimension one and
$\sigma$ maximal, 
then there exists trivializations 
\begin{equation}
\label{Eq:trivial one}
\chi_{\rho}:\shP|_{\mathrm{Star}(\rho)}\cong 
\Lambda|_{\mathrm{Star}(\rho)}\times Q^{\gp}
\end{equation}
and 
\begin{equation}
\label{Eq:trivial two}
\chi_{\sigma}:\shP|_{\mathrm{Star}(\sigma)}\cong
\Lambda|_{\mathrm{Star}(\sigma)}\times Q^{\gp}.
\end{equation}
If $\rho\subseteq\sigma$,
so that $\Int(\sigma)=\mathrm{Star}(\sigma)\subseteq\mathrm{Star}(\rho)$,
we have on $\Int(\sigma)$ that
\[
\chi_\rho\circ\chi_{\sigma}^{-1}(m, q) = \big(m, 
q+(d\varphi_{\rho}|_{\sigma})(m)\big).
\]
Alternately, this identification can be viewed as describing parallel
transport between stalks of $\shP$:
if $x\in\rho$ a codimension one cell and $y\in\sigma$
an adjacent maximal cell, then parallel transport from $\shP_y$
to $\shP_x$ along a path contained in $\sigma$ is given by
\[
(m,q)\mapsto \big(m,(d\varphi_{\rho}|_{\sigma})(m)+q\big).
\]
We will make frequent use of this particular description of parallel transport to compare germs of $\P$ at nearby points.

The sheaf $\shP$ contains via \cite[Def.\ 1.16]{GHS},
a subsheaf $\shP^+\subseteq\shP$. We can describe the
stalk of $\shP^+$ at a point $x\in B$ lying in the interior of a maximal
cell, using the trivalization \eqref{Eq:trivial two}, as
\begin{equation}
\label{eq:P+x max}
\shP^+_x=\Lambda_x\times Q.
\end{equation}
If $x$ lies in the interior of a codimension one cell
$\rho$ which is not a boundary cell, we may, using the trivialization
\eqref{Eq:trivial one},
write $\shP^+_x$ as
\begin{equation}
\label{Eq: Pplus description}
\shP^+_x=\big\{\big(m,(d\varphi_{\rho}|_{\sigma})(m)+q\big)\,\big |\,\rho\subseteq\sigma
\in\scrP^{\max},\, m\in T_x\sigma\cap\Lambda_x,\,q\in Q\big\}.
\end{equation}
Here $T_x\sigma$ denotes the tangent wedge to $\sigma$ at $x$. Alternatively, $\shP^+_x$ has the following description.
Let $\rho\subseteq \sigma,\sigma'$, and let $\xi\in\Lambda_x$
be a choice of vector pointing into $\sigma$ such that its image in 
$\Lambda_{\sigma}/\Lambda_{\rho}\cong\ZZ$ is a generator. Then
\begin{equation}
\label{eq:P+x codim one}
\shP^+_x=(\Lambda_{\rho}\oplus\NN Z_+\oplus \NN Z_-\oplus Q)/\langle Z_+ + Z_-
= \kappa_{\rho}\rangle.
\end{equation}
For this, see the discussion in \cite[\S2.2]{GHS}. Here, the projections
to $\Lambda_x$ in \eqref{eq:Q exact} 
of $Z_+$ and $Z_-$ are $\xi$ and
$-\xi$ respectively. 
Under this description,
if $y\in\Int(\sigma)$, $y'\in\Int(\sigma')$,
then the parallel transport from $\shP^+_x$ to $\shP_y^+$ and 
$\shP_{y'}^+$ takes the form
\begin{equation}
\label{eq:parallel transport}
(\lambda_{\rho},a Z_+, bZ_-,q)\mapsto \begin{cases}
(\lambda_{\rho}+(a-b)\xi, b\kappa_{\rho}+q)\in \shP_y^+,&\\
(\lambda_{\rho}+(a-b)\xi, a\kappa_{\rho}+q)\in \shP_{y'}^+.&
\end{cases}
\end{equation}

In \S\ref{subsubsection:affine manifold}, we took $\Delta\subseteq B$
to be the union of codimension two cells of $\P$. However, it
is frequently the case, as in the sequel, that the affine structure
on $B\setminus\Delta$ extends to a larger subset of $B$. Thus
in general we will only assume that $\Delta$ is a union of \emph{some}
of the codimension two cells of $\P$. In this case, $\varphi$ may
also have a single-valued representative in larger open sets.

In detail, suppose that $\tau\in\scrP$ is a cell
with $\Int(\tau)\cap\Delta=\emptyset$,
and $\varphi$ can be expressed as a single-valued function on
$\mathrm{Star}(\tau)$. 
This means that there exists a piecewise affine function
$\varphi_{\tau}:\mathrm{Star}(\tau)\rightarrow Q^{\gp}_{\RR}$ whose kink
along any codimension one cells $\rho\supseteq\tau$ agrees with the
kink of $\varphi$. In this case, the sheaf $\shP$ extends to a sheaf
on $\mathrm{Star}(\tau)$ via a trivialization on $\mathrm{Star}(\tau)$
given again by $\chi_{\tau}:\shP|_{\mathrm{Star}(\tau)}\rightarrow
\Lambda|_{\mathrm{Star}(\tau)}\times Q$, and if $\tau\subseteq\sigma
\in\P^{\max}$, then
\begin{equation}
\label{eq:psi tau sigma}
\chi_{\tau}\circ \chi_{\sigma}^{-1}(\lambda,q)
= \big(\lambda, q+(d\varphi_{\tau}|_{\sigma})(\lambda)\big).
\end{equation}
As before, for $x\in\Int(\tau)$, we may define $\shP^+_x \subseteq \shP_x$
as:
\begin{equation}
\label{eq:P+x final}
\shP^+_x:=\big\{\big(m,(d\varphi_{\tau}|_{\sigma})(m)+q\big)\,\big |\,\tau\subseteq\sigma
\in\scrP^{\max},\, m\in T_x\sigma\cap\Lambda_x,\,q\in Q\big\}.
\end{equation}
Here $T_x\sigma$ denotes the tangent wedge to $\sigma$ at the point $x$.

Alternatively, for each $y\in\Int(\sigma)$ for $\sigma$
a maximal cell containing $\tau$, there is a canonical identification
of $\shP_y$ with $\shP_x$ via parallel transport inside $\mathrm{Star}(\tau)$.
We then may define
\[
\shP^+_x=\bigcap_{\tau\subseteq\sigma\in\scrP^{\max}} \shP^+_y.
\]
With this notation,
given a monoid ideal $I\subseteq Q$,
we obtain an ideal $I_y\subseteq \kk[\shP_y^+]=\kk[Q][\Lambda_y]$
generated by $I$ under the inclusion of $\kk[Q]$ in $\kk[\shP_y^+]$.
Since $\shP_x^+\subseteq \shP^+_y$ by the above equality, we may then
define 
\begin{equation}
\label{eq:Ix def}
I_x = \sum_{\tau\subseteq\sigma\in\scrP^{\max}} I_y\cap\kk[\shP^+_x].
\end{equation}

We finally record here for future use a key observation which underpins \cite{GS2}, analogous to \cite[Prop.\ 2.6]{GS2}.

\begin{definition}
Let $x\in B$ lie in the interior of a maximal cell, $m\in \shP_x^+$.
If $m=(\bar m,q)\in \Lambda_x\oplus Q$ under the splitting 
\eqref{eq:P+x max}, then we
define the \emph{$Q$-order} of $m$ at $x$ to be $\mathrm{ord}_x(m):=q$.
\end{definition}

\begin{proposition}
\label{prop:order increasing}
Let $\tau\in\mathring\P$ with $\tau\not\subseteq\Delta$
and suppose that $B\setminus\Delta$ carries an MVPL function
$\varphi$ with a single-valued representative on $\mathrm{Star}(\tau)$. Let
$\sigma, \sigma'\in\P^{\max}$ be two cells containing $\tau$. Let $y\in \Int(\sigma)$, $y'\in\Int(\sigma')$,
$x\in \Int(\tau)$,
and let $m\in \shP^+_y$ be such that the parallel transport of $\bar m
\in\Lambda_y$ to $\Lambda_x$ is a tangent vector pointing into $\sigma$
and the parallel transport of $-\bar m$ to $\Lambda_x$ is a tangent
vector pointing into $\sigma'$.
Let $m'$ be the parallel transport of $m$ to $\shP_{y'}$ along a path
in $\mathrm{Star}(\tau)$. Then $m'\in \shP^+_{y'}$, and there
exists a non-zero $q\in Q$ such that $\mathrm{ord}_{y'}(m')
= \mathrm{ord}_{y}(m)+q$. 
\end{proposition}

\begin{proof}
We may choose $y$ and $y'$ to lie close to a line passing through
$\Int(\tau)$ with tangent direction $\bar m$. In particular, these
may be chosen so that the straight line path between $y$ and
$y'$ still has tangent vector $\bar m$ but 
only passes through codimension zero and one cells of $\P$. Thus
using parallel transport along this straight line path,
it is sufficient to show the claim when $\tau$ is codimension one.
The result then follows from \eqref{eq:parallel transport}. Indeed, 
if $m=(\bar m,q')\in \shP_y^+$, we may write $\bar m=\lambda_{\tau}+a\xi$
for some $\lambda_{\tau}\in\Lambda_{\tau}$ and $a>0$ by the assumption
on $\bar m$. Then the parallel transport of $m$ to $\shP^+_x$ can
be written as $(\lambda_{\tau},aZ_+,0,q')\in \shP^+_x$, and
the further transport to $\shP^+_{y'}$ then takes the form
$(\bar m,a\kappa_{\tau}+q')$, giving the claim with $q=a\kappa_{\tau}$.
\end{proof}

\subsubsection{The canonical choice of MVPL function}

Here we will make use of a canonical choice of MVPL function
$\varphi$:

\begin{construction}
\label{constr:phi}
Following \cite{GHK}, \cite{GSCanScat}, we define
a canonical choice of MVPL function $\varphi$ 
from the log Calabi-Yau pair $(X,D)$ as follows.
We simply take,
for $\rho\in\mathring{\P}$ a codimension one cell,
the kink of $\varphi$ to be $\kappa_{\rho}=[D_{\rho}]$. Note that
$\kappa_{\rho}\in Q$ by the assumption that $Q$ contains all effective
curve classes, and that $\varphi$
is a convex MVPL function provided $D_{\rho}$ is not numerically
trivial for any such $\rho$.
\end{construction}

As we saw in Proposition \ref{prop:local affine submersion},
under certain situations, the affine structure extends across 
$\mathrm{Star}(\tau)$. Similarly, a single-valued representative
for $\varphi$ as constructed in Construction \ref{constr:phi}
may be found on this open subset of $B$. Before doing so, we review
for this purpose and future use later in the paper certain toric facts.

If $\Sigma$ is a
complete, non-singular fan in $M_{\RR}$, there is a standard
description of $N_1(X_{\Sigma})$.
Set $T_{\Sigma}:=\ZZ^{\Sigma(1)}$, the lattice with basis $t_\rho$, for $\rho\in\Sigma(1)$.
There is a canonical map
\begin{eqnarray}
\label{Eq: s}
s:T_{\Sigma} & \longrightarrow & M \\
\nonumber
t_{\rho} &   \longmapsto & m_{\rho}
\end{eqnarray}
where $m_\rho$ denotes the primitive generator of $\rho$. Because $\Sigma$ is
assumed to be non-singular, this map is surjective, and there is a canonical
isomorphism $N_1(X_{\Sigma})\cong \ker s$, with the isomorphism given
by
\[N_1(X_{\Sigma})\ni\beta \longmapsto \sum_{\rho\in\Sigma(1)} (D_{\rho}\cdot \beta)t_{\rho}.\]

More generally, suppose given a collection of 
cones $\sigma_1,\ldots,\sigma_p\in\Sigma$ and elements $m_1,\ldots,m_p\in M$
with $m_i$ tangent to $\sigma_i$. Suppose further that the $m_i$
satisfy the balancing condition $\sum_i m_i=0$. We then also obtain
a curve class $\beta\in N_1(X_{\Sigma})$ as follows. 
Let $\rho_{i1},\ldots,\rho_{in_i}$ be the one-dimensional faces of
$\sigma_i$. Thus we may write
$m_i=\sum_j a_{ij}m_{\rho_{ij}}$ with the $a_{ij}$ (possibly negative)
integers. Take
\[
a_{\rho}:=\sum_{i,j\,\,\mathrm{s.t.}\, \rho=\rho_{ij}} a_{ij}.
\]
Then $\sum a_{\rho}t_{\rho}\in\ker s$, hence represents a curve class
$\beta$. Explicitly, this is characterized by the condition
\begin{equation}
\label{eq:Drho dot beta}
D_{\rho}\cdot \beta = 
\sum_{i,j\,\,\mathrm{s.t.}\, \rho=\rho_{ij}} a_{ij}.
\end{equation}

\begin{lemma}
\label{lem:beta recovery} 
Let $\Sigma$, $X_{\Sigma}$ be as above. Then:
\begin{enumerate}
\item There exists a piecewise
linear function $\psi:M_{\RR}\rightarrow N_1(X_{\Sigma})\otimes_{\ZZ}\RR$
with kink along a codimension one cone $\rho\in\Sigma$ being the
class of $D_{\rho}$, the corresponding one-dimensional stratum of
$X_{\Sigma}$.
\item This function $\psi$ is universal,
in the following sense. If $D$ is a divisor
on $X_{\Sigma}$, let $\psi_D$ be the composition of $\psi$ with
the map $N_1(X_{\Sigma})\otimes\RR\rightarrow \RR$ given by
$\beta\mapsto \beta\cdot D$. Then up to linear functions, 
$\psi_D$ is the piecewise linear function determined by the divisor $D$.
\item
If $\sigma_1,\ldots,\sigma_p\in\Sigma$
and $m_1,\ldots,m_p\in M$ satisfy $m_i$ tangent to $\sigma_i$ and
$\sum_i m_i=0$, then the
corresponding class $\beta\in N_1(X_{\Sigma})$ is
\[
\beta=\sum_{i=1}^p (d\psi|_{\sigma_i})(m_i).
\]
\end{enumerate}
\end{lemma}

\begin{proof}
(1) is \cite[Lem.\ 1.13]{GHK}, while (3)
follows immediately from the argument of the proof of
\cite[Lem.\ 1.13]{GHK}, generalizing \cite[Lem.\ 3.32]{GHK}.
Finally, (2) follows 
as it is standard toric geometry that the kink of the piecewise
linear function corresponding to a divisor $D$
at a codimension one cone $\rho$ is $D\cdot D_{\rho}$.
\end{proof}

\begin{proposition}
\label{prop:phi tau}
If $\tau\in\P$ is such that $D_{\tau}$ satisfies the hypotheses
of Proposition \ref{prop:local affine submersion}, then
the MVPL function $\varphi$ has a single-valued representative
$\varphi_{\tau}$ on $\mathrm{Star}(\tau)$.
\end{proposition}

\begin{proof}
By Lemma \ref{lem:beta recovery}, there is a
piecewise linear function 
$\psi:\RR^d\rightarrow N_1(D_{\tau}) \otimes \RR$ with respect to the fan 
$\Sigma_{\tau}$ whose kink
along a codimension one $\rho\in\Sigma_{\tau}$ is the class of
the corresponding one-dimensional stratum of $D_{\tau}$. 
Now we have a canonical map $\iota:N_1(D_{\tau})\rightarrow N_1(X)$
induced by the inclusion $D_{\tau}\hookrightarrow X$.
Thus, using Proposition \ref{prop:local affine submersion}, we may
take $\varphi_{\tau}=\iota \circ \psi\circ\upsilon$.
\end{proof}

\subsection{Wall structures on $(B,\P)$}
\label{Sec: wall structures}
The main ingredient of the construction of the family mirror to $(X,D)$ is a \emph{scattering diagram}, or \emph{wall structure}, which lives on $(B,\P)$. 
We will review 
this data from \cite[\S2.3]{GHS} and its consequences in the remaining part of 
this section. While \cite{GHS} and \cite{GSCanScat} use the term 
``wall structure,'' we
will instead use the term ``scattering diagram" here.

We fix a general situation of a polyhedral affine manifold $(B,\P)$
as in \cite[Constr.\ 1.1]{GHS}. For our purposes, the reader may think
of $(B,\P)$ as having come from a log Calabi-Yau variety $(X,D)$,
or, in the relative case $X\rightarrow\AA^1$, as a fibre $p_{\mathrm{trop}}^{-1}(1)$ of the
induced map of Remark \ref{rem:relative case}. We also fix
a monoid $Q$ with $Q^{\times}=\{0\}$, an MVPL function $\varphi$ 
with values in $Q^{\gp}_{\RR}$, and a monoid ideal $I\subseteq Q$
with $\sqrt{I}=\fom=Q\setminus \{0\}$.

\begin{definition}
\label{Def: canonical scattering walls}
A \emph{wall} on $B$ is a codimension one rational polyhedral cone
$\fod\not\subseteq\partial B$ contained in some maximal cone $\sigma$ of $\P$, 
along with an element
\begin{equation}
\label{Eq: canonical fncs}
    f_{\fod}=\sum_{m\in\shP^+_x, \bar m\in \Lambda_{\fod}} c_m z^m 
\in \kk[\shP^+_x]/I_x,
\end{equation}
where $I_x$ is the ideal defined in \eqref{eq:Ix def}.
Here $x\in \Int(\fod)$ and
$\Lambda_{\fod}$ is the lattice of integral tangent vectors to
$\fod$. We further require that for every $m\in\shP^+_x$ with
$c_m\not=0$ and every $y\in \fod\setminus\Delta$, we have 
$m\in\shP_y$ lying in $\shP_y^+$ under parallel transport under a path in 
$\fod$. Finally, we require that $f_{\fod}\equiv 1\mod \fom_x$.

We say a wall $\fod$ has \emph{direction} $v \in \Lambda_{\fod}$ if the attached function $f_\fod$, given as in \eqref{Eq: canonical fncs}, satisfies 
$\bar m=-kv$ for some $k\in \NN$ whenever $c_m\not=0$. In this case we denote $f_\fod$ by $f_{\fod}(z^{-v})$ 

A \emph{scattering diagram} on $(B,\P)$ is a finite set $\foD$ of walls.

If $\foD$ is a scattering diagram, we define
\begin{eqnarray}
\nonumber
\Supp(\foD) & := & \bigcup_{\fod\in \foD} \fod, \\
\nonumber
\Sing(\foD) & := & \Delta\cup \bigcup_{\fod\in\foD} \partial\fod
\cup \bigcup_{\fod,\fod'\in\foD} (\fod\cap\fod'),
\end{eqnarray}
where the last union is over all pairs of walls $\fod,\fod'$ with
$\fod\cap\fod'$ codimension at least two. In particular,
$\Sing(\foD)$ is a codimension at least two subset of $B$.

A \emph{joint} $\foj$ of $\foD$ is a codimension two polyhedral
subset of $B$ contained in the intersection of $\Sing(\foD)$ with some $\sigma\in\P^{\max}$,
such that for $x\in\Int(\foj)$, 
the set of walls $\{\fod\in\foD\,|\,x\in \fod\}$ is independent of
$x$. Further, a joint must be a maximal subset with this property.

If $x\in B\setminus \Sing(\foD)$, we define
\begin{equation}
\label{eq: fx def}
f_x:=\prod_{x\in\fod\in\foD} f_{\fod}.
\end{equation}
We say that two scattering diagram $\foD$, $\foD'$ are
\emph{equivalent} if $f_x=f'_x$ for all $x\in B\setminus (\Sing(\foD)
\cup\Sing(\foD'))$.
\end{definition}
\begin{remark}
\label{Rem: minimal D}
Every scattering diagram $\foD$ is equivalent to a \emph{minimal} scattering diagram
which does not have two walls with the same support and contains no trivial walls $(\fod, f_{\fod})$ with $f_{\fod} = 1$. Indeed, if $(\fod_1, f_1),(\fod_2, f_2) \in \foD$ with
$\fod_1 = \fod_2$, then we can replace these two walls with the single wall $(\fod, f_1\cdot f_2)$
without affecting the wall crossing automorphisms. In addition we can remove any trivial wall. 
\end{remark}

\begin{remark}
\label{rem:equivalence}
The definition of a scattering diagram depends on the choice of an ideal
$I$. If $\foD$, $\foD'$ are scattering diagram defined using ideals
$I, I'$ and $I''\supseteq I,I'$, then we will say $\foD$ and $\foD'$
are \emph{equivalent modulo $I''$} if $f_x=f'_x\mod I_x''$
for all $x\in B\setminus (\Sing(\foD)\cup\Sing(\foD'))$. This allows us to compare scattering diagrams with different ideals.

The reader should bear in mind that in general we will be dealing
with a compatible system of scattering diagrams, i.e., a scattering diagram
$\foD_I$ for each ideal $I$, such that $\foD_I$ and $\foD_{I'}$
are equivalent modulo $I'$ whenever $I\subseteq I'$. Alternatively,
one may pass to the completion by taking the limit over all $I$, but this 
tends to cause additional notational complexities.
\end{remark}

\begin{remark}
The definition of scattering diagram \cite[Def.\ 2.11]{GHS} is somewhat
more restrictive on the set of walls. In particular, it includes an
additional condition that the underlying
polyhedral sets of $\foD$ are the codimension one cells of a rational
polyhedral decomposition $\P_\foD$ of $B$ refining $\P$. Further,
$\P_{\foD}$ needs to satisfy some properties we do not enumerate here. 
These conditions make for easier technical definitions and constructions
in \cite{GHS}, and given any scattering diagram, it is easy to construct
an equivalent scattering diagram which does satisfy these properties.
However, the scattering diagrams constructed naturally in this paper
do not satisfy these conditions, and it is convenient not to impose
them.

The definition of wall we give here is also slightly more restrictive
than that in \cite{GHS}, which does not insist on $f_{\fod}\equiv 1$ mod $\fom_x$ 
if $\fod$ is contained in codimension one cell of $\P$.
Allowing such a possibility is useful to incorporate the
type of discriminant locus $\Delta$ allowed in \cite{GS2}, where $\Delta$
does not lie in the union of codimension two cells. However, that is not
the case here.
\end{remark}

In the remaining part of this subsection we will describe the notion of \emph{consistency} of a wall-structure.
 Roughly put, one uses consistent scattering diagrams to glue together some standard open charts to build the coordinate ring for the mirror to a log Calabi--Yau \cite[\S2.4]{GHS}. 
This definition involves testing a property for each joint,
and we break into cases based
on the codimension of the joint:

\begin{definition}
\label{Def: joints}
The codimension $k \in \{0, 1, 2 \}$ of a joint $\foj$ is the codimension of 
the smallest cell of $\P$ containing $\foj$.
\end{definition}

We would like to emphasize that the terminology ``codimension of a joint'' in Definition \ref{Def: joints}, which is commonly used in the literature, does not refer to the codimenson of the joint itself, viewed as a subset of $B$ (as in that case, a joint is always of codimension two in $B$ by definition).

We now give slight variants of the definition of consistency
in \cite{GHS}; it is easy to check the definitions given here
are equivalent to the definitions given there, but the form
given here is more suited for our purposes.

\subsubsection{Consistency around codimension zero joints } 
\label{Subsubsect: Consistency around codim zero joints}
For $\sigma\in\P^{\max}$, set, with $x\in\Int(\sigma)$,
\begin{equation}
\label{Eqn: R_fou}
R_\sigma:= (\kk[Q]/I)[\Lambda_\sigma]=\kk[\shP^+_x]/I_x.
\end{equation}
Let $\gamma:[0,1]\rightarrow \Int(\sigma)$ be a piecewise smooth path
whose image is disjoint from $\Sing(\foD)$. Further, assume that
$\gamma$ is transversal to $\Supp(\foD)$, in the sense that
if $\gamma(t_0)\in\fod\in\foD$, then there is an $\epsilon>0$ such
that $\gamma((t_0-\epsilon,t_0))$ lies on one side of $\fod$
and $\gamma((t_0,t_0+\epsilon))$ lies on the other. 

Assuming that $\gamma(t_0)\in \fod$, we associate a \emph{wall-crossing}
homomorphism $\theta_{\gamma,\mathfrak{d}}$ as follows.
Let $n_\fod$ be a generator of $\Lambda_\fod^\perp
\subseteq\check\Lambda_x=\Hom(\Lambda_x,\ZZ)$ for some $x\in\Int\fod$, with
$n_\fod$ positive on $\gamma((t_0-\epsilon,t_0))$ as a 
function on $\sigma$ in an affine chart mapping $x$ to the origin.
Then define 
\begin{equation}
\label{Eqn: theta_fop}
\theta_{\gamma,\mathfrak{d}}: R_\sigma \lra R_\sigma,\quad
z^m\longmapsto f_\fod^{\langle n_\fod,\ol m\rangle} z^m.
\end{equation}
As $f_\fod\equiv 1\mod \fom_x$, we see that $f_{\fod}$ is
an element of $R_\sigma^\times$, so that this formula makes sense. 
We refer to $\theta_{\gamma,\mathfrak{d}}$
as the automorphism associated to \emph{crossing the wall $\fod$}.

We may now define the \emph{path-ordered product}
\begin{equation}
    \label{Eq: path ordered product}
    \theta_{\gamma,\foD}:=\theta_{\gamma,\fod_s}\circ\cdots\circ
\theta_{\gamma,\fod_1},
\end{equation}
where $\fod_1,\ldots,\fod_s$ is a complete list of walls traversed
by $\gamma$, in the order traversed. Note that if $\gamma(t_0)\in
\fod,\fod'$, then since $\gamma(t_0)\not\in \Sing(\foD)$,
necessarily $\fod\cap\fod'$ is codimension one so that $\Lambda_{\fod}
=\Lambda_{\fod'}$. From their definitions,
$\theta_{\gamma,\mathfrak{d}}$ and $\theta_{\gamma,\fod'}$ commute, so the
ordering of $\fod, \fod'$ in the path-ordered product is irrelevant.

\begin{definition}
\label{Def: consistency in codim zero}
Let $\foj$ be a codimension zero joint contained in a maximal
cell $\sigma$, and let $\gamma$ be a small loop in $\sigma$ around
$\foj$. Then $\foD$ is said to be \emph{consistent at $\foj$} if
\[
\theta_{\gamma,\foD}=\id
\]
as an automorphism of $R_\sigma$.
\end{definition}

A scattering diagram $\foD$ on $(B,\P)$ is
\emph{consistent in codimension zero} if it is consistent at any
codimension zero joint $\foj$.

\subsubsection{Consistency around codimension one joints}
\label{Subsubsec: Consistency around codim one joints}
There is no condition for a codimension one joint contained
in $\partial B$. Otherwise, let $\rho$ be a codimension one cell of $\P$
not contained in the boundary of $B$. Let $x\in\Int(\rho)$.
We define the ring
\[
R_{\rho}:= \kk[\shP^+_x]/I_x.
\]
Note this is independent of the choice of $x$.
It follows from the definition of wall that if $\fod\in\foD$ is a wall 
with $\fod\cap\Int(\rho)\not=\emptyset$, then $f_{\fod}$ may
be intepreted as an element of $R_{\rho}$. 

Let $\gamma:[0,1]\rightarrow \mathrm{Star}(\tau)$ be path as
in the previous subsection, transversal to $\Supp(\foD)$ and with
image disjoint from $\Sing(\foD)$. Assume further that 
if $\gamma(t)\in\fod$ for some $\fod\in \foD$, then $\fod\cap
\Int(\rho)\not=\emptyset$. Then necessarily $f_{\fod}\in R_{\rho}$.
Thus, we may define $\theta_{\gamma,\mathfrak{d}}:R_\rho\rightarrow R_{\rho}$
using the formula \eqref{Eqn: theta_fop}, and hence
define the path-ordered product $\theta_{\gamma,\foD}:R_{\rho}\rightarrow
R_{\rho}$ as in the codimension zero case.

\begin{definition}
\label{Defn: consistency along codim one}
We say a codimension one joint $\foj$ contained in a codimension
one cell $\rho$ is
\emph{consistent} if, for a sufficiently small loop $\gamma$ in
$\mathrm{Star}(\rho)$ around $\foj$, we have
\[
\theta_{\gamma,\foD}=\id
\]
as an automorphism of $R_{\rho}$.
\end{definition}

\begin{remark}
This definition may appear to be quite different from the one
given in \cite[Def.\ 2.14]{GHS}. The definition given there
was designed to include a more complicated case where $\Delta
\cap\Int(\rho)\not=\emptyset$ and $f_{\fod}\not\equiv 1\mod \fom$
if $\fod\subseteq\rho$. The reader may verify that in our
simpler case, the definition given here implies \cite[Def.\ 2.14]{GHS}.
\end{remark}

\subsubsection{Consistency around codimension two joints}
\label{Subsubsec: Consistency around codim two joints}
To check consistency in codimension two we first will review the theory of broken lines. For details we refer to \cite[\S3.1]{GHS}.
For a given scattering diagram $\foD$ a
\emph{broken line} is a piecewise linear continuous directed
path 
\begin{equation}
    \label{Eq:  broken line}
    \beta \colon (-\infty,0] \lra B \setminus \Sing(\foD)
\end{equation}
with $\beta(0)\not\in\Supp(\foD)$ and
whose image consists of finitely many line segments $L_1, L_2, \ldots , L_N$, such that 
$\dim L_i\cap\fod =0$ for any wall $\fod\in\foD$,
and each $L_i$ is compact except $L_1$. Further, we require that each $L_i\subseteq\sigma_i$
for some $\sigma_i\in\P^{\max}$. To each such segment we assign a monomial 
\[ m_i :=  \alpha_iz^{(v_i ,q_i)} \in \kk [\Lambda_{L_i} \oplus Q^{\gp}]. \]
Here $\Lambda_{L_i}$, as usual, denotes the group of integral tangent vectors
to $L_i$ and is hence a rank one free abelian group.
Each $v_i$ is non-zero and tangent to $L_i$, with $\beta'(t)=-v_i$
for $t\in (-\infty,0]$ mapping to $L_i$.
We require $\alpha_1=1$ and set $m_1 = z^{(v_1,0)}$. We refer to $v_1$ 
as the \emph{asymptotic direction} of the broken line.
Roughly put, each time $\beta$ crosses a wall of $\foD$ it possibly changes
direction and monomial in a specific way that respects the structure. 
Again, roughly, $\beta$ bends in the direction given by linear combinations of $\bar m$ which appear in \eqref{Eq: canonical fncs}.

In detail, given $L_i$ and its attached monomial $m_i$, we determine
$L_{i+1}$ and $m_{i+1}$ as follows. Let $L_i$ be the image under
$\beta$ of an interval $[t_{i-1},t_i]\subset (-\infty,0]$.
Let $I=[t_i-\epsilon,t_i+\epsilon]$ be an interval with $\epsilon$
chosen sufficiently small so that $\beta([t_i-\epsilon,t_i))$
and $\beta((t_i,t_i+\epsilon])$ are disjoint from $\Supp(\foD)$.
There are two cases:
\begin{itemize}
    \item $\beta(t_i)\in\Int(\sigma_i)$ for $\sigma\in\P^{\max}$.
Then we obtain a wall-crossing map $\theta_{\beta|_I,\foD}:
R_{\sigma_i}\rightarrow R_{\sigma_i}$, and $m_i$ may be viewed as an element
of $R_{\sigma_i}$ via the inclusion $\Lambda_{L_i}\subseteq\Lambda_{\sigma_i}$. 
We expand $\theta_{\beta|_I,\foD}(m_i)$ as
a sum of monomials with distinct exponents, and require that
$m_{i+1}$ be one of the terms in this sum.
\item $\beta(t_i)\in\Int(\rho)$ for $\rho\in\P$ a codimension one
cell. If $y=\beta(t_i-\epsilon)$, $y'=\beta(t_i+\epsilon)$, $x=\beta(t_i)$,
we may view $(v_i,q_i)\in\shP^+_y$. By parallel transport to
$x$ along $\beta$, we may view $(v_i,q_i)\in \shP_x$. In fact,
$(v_i,q_i)\in \shP_x^+$ by the assumption that $\beta'(t_i-\epsilon)=-v_i$
and Proposition \ref{prop:order increasing}. Thus we may view
$m_i\in R_{\rho}$, and then $m_{i+1}$ is required to be a term
in $\theta_{\beta|_I,\foD}(m_i)$. A priori, $m_{i+1}\in R_{\rho}$, but
it may be viewed as a monomial in $R_{\sigma_{i+1}}$ by parallel
transport to $y'$.
\end{itemize}

Having defined broken lines, we proceed to the definition of consistency
for a joint $\foj$ of codimension two. Let $\omega\in
\P^{[n-2]}$ be the smallest cell containing $\foj$. Build a new
affine manifold $(B_\foj,\P_\foj)$ by replacing any $\tau\in\P$
with $\tau\supseteq\foj$ by the tangent wedge of $\omega$ in $\tau$.
Note that the inclusion $\tau\subseteq\tau'$ of faces induces an
inclusion of the respective tangent wedges. So $B_\foj$ is a local
model for $(B,\P)$ near $\foj$ all of whose cells are cones.
The scattering diagram $\foD$ induces a
scattering diagram $\foD_\foj$ by considering only the walls containing
$\foj$ and replacing them with tangent wedges based at $\omega$ for the
underlying polyhedral subsets of codimension one. Since the only
joint of $\foD_\foj$ is the codimension two cell
$\Lambda_{\foj,\RR}$ this scattering diagram is trivially consistent in
codimensions zero and one. Now let $m$ be an asymptotic monomial on
$(B_\foj,\P_\foj)$. For a general point $p\in B_\foj$, say contained
in $\sigma\in\P^{\max}_{\foj}$, define
\begin{equation}
\label{Def: vartheta^foj_m(p)}
\vartheta^\foj_m(p):=\sum_\beta a_\beta z^{m_\beta}\in R_\sigma.
\end{equation}
The sum runs over all broken lines $\beta$ on $(B_\foj,\P_\foj)$ with
asymptotic monomial $m$ and endpoint $p$. Furthermore, 
$a_{\beta}z^{m_{\beta}}$ denotes the monomial associated with the
last domain of linearity of $\beta$.

\begin{definition}
\label{Def: Consistency in codim two}
The scattering diagram $\foD$ is \emph{consistent along the codimension
two joint $\foj$} if the $\vartheta^\foj_m(p)$ satisfy the following
properties:
\begin{enumerate}
\item For $p,p'\in \Int(\sigma)$ with $\sigma\in\P_{\foj}^{\max}$ 
and any path $\gamma$
from $p$ to $p'$ for which $\theta_{\gamma,\foD_{\foj}}$ is
defined, we have 
\[
\theta_{\gamma,\foD_{\foj}}(\vartheta^{\foj}_m(p))=
\vartheta^{\foj}_m(p')
\]
in $R_{\sigma}$.
\item If $\rho\in\scrP_{\foj}$ is codimension one with $\rho\subseteq
\sigma,\sigma'\in\scrP_{\foj}^{\max}$, and $p\in\sigma$, $p'\in\sigma'$,
suppose there is in addition a path
$\gamma$ joining $p$ and $p'$ not crossing any walls of $\foD_{\foj}$
not contained in $\rho$. Then $\theta^{\foj}_m(p)$ and
$\theta^{\foj}_m(p')$ make sense as elements of $R_{\rho}$ and
\[
\theta_{\gamma,\foD_{\foj}}(\vartheta^{\foj}_m(p))=
\vartheta^{\foj}_m(p')
\]
in $R_{\rho}$.
\end{enumerate}
A scattering diagram $\foD$ is \emph{consistent} if it is consistent in
codimensions zero, one and along each
codimension two joint. \end{definition}

\medskip

We now need to explain how consistent scattering diagrams may be
constructred using punctured  Gromov-Witten theory.

\subsection{Punctured Gromov--Witten theory}
\label{Subsec: punctured}
We give here a quick review of punctured Gromov--Witten theory as 
set up in \cite{ACGSII} and as used in \cite{GSCanScat}. We may also
send the reader to \cite{GSUtah} for a brief survey, but the reader
should bear in mind that the definition of punctures evolved since
the latter paper was written. 
\subsubsection{Log schemes and their tropicalizations}
We assume the reader is familiar with the basic language of log
geometry, but we review definitions to establish notation. 
A log structure on a scheme $X$ is a sheaf of (commutative) monoids $\shM_X$ together with a homomorphism of sheaves of multiplicative monoids
$\alpha_X : \shM_X \to \O_X$ inducing an isomorphism $\alpha_X^{-1}(\O_X^{\times})\rightarrow \O_X^{\times}$, allowing us to identify $\O_X^{\times}$ as a subsheaf
of $\shM_X$. The standard notation we use for a log scheme is $X:=(\ul{X},\shM_X,\alpha_X)$, where by $\ul X$ we denote the underlying scheme. Throughout this paper all log schemes are fine and saturated (fs) 
\cite[I,\S1.3]{ogus2018lectures}, except when explicitly noted. 
A morphism of log schemes $f:X \rightarrow Y$ consists of an
ordinary morphism $\ul{f}:\ul{X}\rightarrow \ul{Y}$ of schemes along with
a map $f^{\flat}:f^{-1}\shM_Y\rightarrow \shM_X$ compatible with $f^\#:f^{-1}\O_Y\rightarrow \O_X$ via the structure homomorphisms
$\alpha_X$ and $\alpha_Y$.

The \emph{ghost sheaf} is defined by
$\overline{\shM}_{X}:=\shM_X/\O_X^{\times}$, and captures the key combinatorial
information about the log structure. In particular, it leads
to the \emph{tropicalization} $\Sigma(X)$ of $X$, an abstract polyhedral
cone complex, see \cite{ACGSI}, \S2.1 for details. In brief,
$\Sigma(X)$ is a collection of cones along with face maps between
them. There is one cone $\sigma_{\bar x}:=
\Hom(\overline{\shM}_{X,\bar x},\RR_{\ge 0})$ for every 
geometric point $\bar x\rightarrow X$,
and if $\bar x$ specializes to $\bar y$, there is a generization map
$\overline\shM_{X,\bar y}\rightarrow \overline\shM_{X,\bar x}$
which leads dually to a map
$\sigma_{\bar x}\rightarrow\sigma_{\bar y}$.
The condition of being fine and saturated implies this is an inclusion
of faces. Note that each cone $\sigma_{\bar x}\in \Sigma(X)$ comes with a
tangent space of integral tangent vectors $$\Lambda_{\sigma_{\bar x}}:=\Hom(\overline\shM_{X,\bar x},\ZZ)$$
and a set of integral points
\[
\sigma_{\bar x,\ZZ}:=\Hom(\overline\shM_{X,\bar x},\NN).
\]
Tropicalization is functorial, with $f:X\rightarrow Y$ inducing
$f_{\mathrm{trop}}:\Sigma(X)\rightarrow\Sigma(Y)$, with a map of cones
$\sigma_{\bar x}\rightarrow\sigma_{f(\bar x)}$ induced by
$\bar f^{\flat}:\overline\shM_{Y,f(\bar x)}\rightarrow \overline\shM_{X,x}$. In cases we consider in this article, after identifying $\sigma_{\bar x}$ and
$\sigma_{\bar y}$ whenever $\sigma_{\bar x}\rightarrow\sigma_{\bar y}$
is an isomorphism, we obtain an ordinary polyhedral cone complex. In particular, for us the most important example is the following:
\begin{example}
\label{Ex: divisorial}
Let $(X,D)$ be a pair satisfying the conditions of 
\S\ref{subsubsection:affine manifold}. The divisorial log
structure on $X$ coming from $D$ is given by taking $\shM_X$ to be 
the subsheaf of $\O_X$ consisting of functions invertible on $X\setminus
D$. Then the abstract polyhedral cone complex $\Sigma(X)$ agrees
with $(B,\P)$.
\end{example}

\subsubsection{The Artin fan}
\label{subsubsect:the artin fan}
Given an fs log scheme $X$, we have its associated \emph{Artin fan}
$\shX$, as constructed in \cite[Prop.\ 3.1.1]{ACMW17}. There is
a factorization 
\[
X\rightarrow\shX\rightarrow \Log_{\kk},
\] 
of the tautological morphism $X\rightarrow\Log_{\kk}$, where
$\Log_{\kk}$ is Olsson's stack parameterizing all fs log structures,
\cite{Olsson03}. The morphism $X\rightarrow\shX$ is strict and
the morphism $\shX\rightarrow\Log_{\kk}$ is \'etale and representable
by algebraic spaces. 

In this paper, we will only need the Artin fan
for the log Calabi-Yau pair $(X,D)$ being considered here. In particular,
if $X$ is a toric variety and $D$ is its toric boundary, then 
$\shX=[X/\GG_m^n]$, where $\GG_m^n$ is the big torus acting on $X$.
Further, in this article, all log Calabi-Yau pairs being considered
are obtained by blowing up a locus on the boundary of a toric pair,
and this does not change the Artin fan.

The Artin fan $\shX$ encodes the tropicalization of $X$ as an algebraic
stack. In particular, a useful way of thinking about $\shX$
is given in \cite[Prop.\ 2.10]{ACGSI}, which applies in particular
to the log Calabi-Yau pairs $(X,D)$ being considered here. In this
case, if $T$ is an fs log scheme, then the set of morphisms
of log stacks over $\Spec\kk$, $\Hom(T,\shX)$, coincides with the set
$\Hom_{\mathbf{Cones}}(\Sigma(T),\Sigma(X))$ of morphisms of abstract
cone complexes.

\subsubsection{Stable punctured maps}
\label{subsubsec:stable punctured}
The theory of stable log maps was developed in \cite{GSlog,AbramovichChen}, as a theory of curve counting invariants with target
space a log scheme, as a special case, $X=(X,D)$ as in Example \ref{Ex: divisorial}. A stable log map with
target $X/\Spec\kk$ is a commutative diagram
\[
\xymatrix@C=30pt
{
C\ar[r]^f\ar[d]_{\pi} & X\ar[d]\\
W\ar[r]&\Spec \kk
}
\]
where $\pi$ is a log smooth family of curves, which is required to be an integral morphism all of whose geometric fibres
are reduced curves. Further, $C/W$ comes with
a set of disjoint sections $p_1,\ldots,p_n:\ul{W}\rightarrow
\ul{C}$, referred to as the \emph{marked points}, disjoint from the nodal locus of $C$. Away from the nodal locus, \[\overline{\shM}_C=\pi^*\overline{\shM}_W
\oplus\bigoplus_{i=1}^n p_{i*}\ul{\NN}.\] 
\begin{remark}
\label{Remark: contact orders log map}
Crucially, such a stable log map records \emph{contact orders} at marked and nodal points.
If $p\in C_{\bar w}$ is a marked point of a geometric fibre of
$\pi$, we have
\[
\bar f^{\flat}:P_p:=\overline{\shM}_{X,f(p)}\longrightarrow
\overline{\shM}_{C,p}=\overline{\shM}_{W,\pi(p)}\oplus\NN
\stackrel{\pr_2}{\longrightarrow}\NN,
\]
which can be viewed as an element $u_p\in P_p^{\vee}:=\Hom(P_p,\NN)
\subseteq \sigma_{f(p)}$, called the \emph{contact order at $p$}. Similarly, if $x=q$ is a node,
there exists a homomorphism 
\begin{equation}
\label{eq:node contact order}
u_q:P_q :=\overline{\shM}_{X,f(q)} \lra \ZZ,
\end{equation}
called \emph{contact order at $q$}, see \cite[(1.8)]{GSlog} or \cite[\S2.3.4]{ACGSI}.
In the case the target space is $(X,D)$,
the contact order records tangency information with the irreducible
components of $D$. 
\end{remark}
Punctured invariants, introduced in \cite{ACGSII},
allow negative orders of tangency at particular marked points by enlarging
the monoid $\overline{\shM}_{C,p}$. This is done as follows.
\begin{definition}
Let $(Y,\M_Y)$ be an fs log scheme
with a decomposition $\M_Y = \M \oplus_{\O_Y^\times} \mathcal{P}$. Denote  $\mathcal{E} = \M \oplus_{\O_Y^\times} \shP^{\gp}$. A \emph{puncturing} of $Y$ along $\mathcal{P} \subset \M_Y$ is a \emph{fine} sub-sheaf of monoids $\M_{Y^{\circ}} \subset \mathcal{E}$ containing $\M_Y$ with a structure map $\alpha_{Y^{\circ}}:  \M_{Y^{\circ}} \to \O_Y$ such that
\begin{itemize}
    \item The inclusion $\M_Y \to \M_{Y^\circ} $ is a morphism of fine logarithmic structures
on $Y$.
\item  For any geometric point $\bar{x}$ of $Y$ let $s_{\bar{x}}  \in \M_{Y^\circ}$ be such that $ s_{\bar{x}} \not\in \M_{\bar{x}} \oplus_{\O_Y^{\times}}
\shP_{\bar{x}}$. Representing $ s_{\bar{x}}  = (m_{\bar{x}}, p_{\bar{x}}) \in \M_{\bar{x}} \oplus_{\O_{Y}^{\times}} \shP^{\gp}$, we have $\alpha_{\M_{Y^{\circ}}} (s_{\bar{x}}) =
\alpha_{\shM}(m_{\bar{x}}) = 0$ in $\O_{Y,{\bar{x}}}$.
\end{itemize}
\end{definition}
We will also call the induced morphism of logarithmic schemes $Y^{\circ} \to Y$ a \emph{puncturing of $Y$ along $\mathcal{P}$}, or call $Y^{\circ}$ a \emph{puncturing of $Y$}. Although there may not be a unique choice of puncturing, once given a morphism
to another log scheme, there is a canonical minimal choice:

\begin{definition}
A morphism $f:Y^{\circ}\rightarrow X$ from a puncturing $Y^{\circ}$ of $Y$ 
is \emph{pre-stable} if $\shM_{Y^{\circ}}$ is the fine submonoid
of $\shE$ generated by $\shM_Y$ and $f^{\flat}(f^{-1}(\shM_X))$. 
\end{definition}

\begin{definition}
A punctured curve over an fs log scheme $W$ is given by the data $ C^{\circ} \rightarrow C \to W$ where:
\begin{itemize}
    \item $C \to W$ is a log smooth curve with marked points $\mathbf{p}=p_1, \ldots, p_n$.
In particular, $\shM_C=\shM\oplus_{\O_C^{\times}}\shP$ where 
$\mathcal{P}$ is the divisorial logarithmic
structure on $C$ induced by the divisor $\bigcup_{i=1}^n p_i(W)$.
\item $C^{\circ} \rightarrow C$ is a puncturing of $C$ along $\mathcal{P}$.
\end{itemize}
\end{definition}
We now fix the target space $X=(X,D)$ along with a log smooth
morphism $X\rightarrow S$. In this paper, either $S=\Spec\kk$
with the trivial log structure, or $S=\AA^1$ with the divisorial
log structure coming from $0\in\AA^1$.
\begin{definition}
A punctured map to $X/S$ consists of a punctured curve  
$C^{\circ} \rightarrow C \to W$ and a
morphism $f$ fitting into a commutative diagram
\begin{equation} 
\xymatrix@C=30pt
{
{C}^{\circ} \ar[r]^{f}\ar[d]_{\pi}
&X\ar[d]\\
W \ar[r] & S
}
\nonumber
\end{equation}
Further, we require that
$f$ is pre-stable and defines an ordinary stable map on underlying schemes. We use the notation $(C^{\circ}/W,\mathbf{p},f)$ for a punctured map.
\end{definition}

The main point of increasing the monoid at a punctured point is
that it now allows contact orders $u_p\in P_p^*:=\Hom(P_p,\ZZ)$,
rather than just in $P_p^{\vee}$. Here the contact order $u_p$
is now given as a composition
\begin{equation}
\label{Eq: contact order at punctured points}
    \bar f^{\flat}:P_p\longrightarrow \overline\shM_{C^{\circ},p}
\subset\overline{\shM}_{W,\pi(p)}\oplus\ZZ\stackrel{\pr_2}{\longrightarrow}
\ZZ.
\end{equation}
\medskip

A key point in log Gromov-Witten theory is the tropical interpretation.
Suppose $W=(\Spec\kappa, Q\oplus \kappa^{\times})$ for $\kappa$
an algebraically closed field. Then by functoriality of tropicalization, 
we obtain a diagram
\begin{equation}
\label{eq:tropical diagram}
\xymatrix@C=45pt
{
\Gamma:=\Sigma(C^{\circ})\ar[r]^>>>>>>{h=f_{\mathrm{trop}}}\ar[d]_{\pi_{\mathrm{trop}}} & \Sigma(X)\ar[d]\\
\Sigma(W)\ar[r]& \Sigma(S)
}
\end{equation}
Here $\Sigma(W)=\Hom(Q,\RR_{\ge 0})=Q^{\vee}_{\RR}$ is a rational
polyhedral cone, and for $q\in \Int(Q^{\vee}_{\RR})$, 
$\pi_{\mathrm{trop}}^{-1}(q)$ can be identified with the dual graph
of the curve $C^{\circ}$ obtained as follows. 
\begin{construction}
\label{Const: dual graph}
The dual graph $G$
of the curve $C^{\circ}$ is a graph with sets of
vertices $V(G)$, edges $E(G)$, and legs (or half-edges) $L(G)$, with appropriate incidence relations between vertices and edges, and
between vertices and legs. Each vertex corresponds to an irreducible component of $C^{\circ}$. The edges correspond to nodes of $C$, with vertices of
a given edge indexing
the two branches of $C$ through the node. A leg corresponds to either a marked point or a punctured point. In the marked point case, 
a leg is an unbounded ray.
In the punctured case a leg is a compact interval, with 
one endpoint a vertex corresponding
to the irreducible component containing the punctured point. The
other endpoint of this compact interval is not viewed as a vertex of $G$.
By abuse of notation we denote the topological realization of this graph also by $G$. We denote the topological realization of $G$ obtained by removing the endpoints of compact legs which do not correspond to vertices by $G^{\circ}$.
\end{construction}
As explained in \cite[\S2.5]{ACGSI} and \cite[\S2.2]{ACGSII}, 
$\pi_{\mathrm{trop}}$ can be viewed
as determining a family of tropical curves $(G,{\mathbf g},\ell)$
where ${\mathbf g}:V(G)\rightarrow \NN$ is the genus function,
with ${\mathbf g}(v)$ the genus of the irreducible component
corresponding to $v$, and
$\ell:E(G)\rightarrow Q$ a length function, so that
in the fibre of $\pi_{\mathrm{trop}}$ over $q\in Q^{\vee}_{\RR}$,
the edge $E\in E(G)$ has length $\langle q,\ell(E)\rangle$.
See \cite[Def.\ 2.19, Constr.\ 2.20]{ACGSI}.
As all curves in this paper are genus $0$, we omit the genus
function in the sequel.

Further, $h$ now defines a family of tropical maps to
$\Sigma(X)$ as defined in \cite[Def.\ 2.21]{ACGSI}. For
$s\in \Int(Q^{\vee}_{\RR})$, write
\[
h_s:G\rightarrow\Sigma(X)
\]
for the restriction of $h$ to $G=\pi_{\mathrm{trop}}^{-1}(s)$. We remark that this family of tropical maps is abstract in the sense
that it does not yet necessarily satisfy any reasonable balancing
condition. 

Associated to any family of tropical maps to $\Sigma(X)$ is the type,
recording which cones of $\Sigma(X)$ vertices, edges and legs of
$G$ are mapped to, and tangent vectors to the images of edges and legs:

\begin{definition}
\label{Def: type of the tropical curve}
A \emph{type} of tropical map to $\Sigma(X)$ is data of a triple
$\tau=(G,\boldsymbol{\sigma}, \mathbf{u})$ where $\boldsymbol{\sigma}$
is a map
\[
\boldsymbol{\sigma}:V(G)\cup E(G)\cup L(G)\rightarrow\Sigma(X)
\]
with the property that if $v$ is a vertex of an edge or leg $E$,
then $\boldsymbol{\sigma}(v)\subseteq \boldsymbol{\sigma}(E)$.
Next, $\mathbf{u}$ associates to each oriented edge $E\in E(G)$
a tangent vector $\mathbf{u}(E)\in \Lambda_{\boldsymbol{\sigma}(E)}$
and to each leg $L\in L(G)$ a tangent vector $\mathbf{u}(L)\in
\Lambda_{\boldsymbol{\sigma}(L)}$.
\end{definition}

\begin{definition}
We say a type $\tau$ is \emph{realizable} if there exists a tropical
map $h:G\rightarrow \Sigma(X)$ with $h(v)\in\Int(\bsigma(v))$ for each
vertex $v\in V(G)$, $h(\Int(E))\subseteq \Int(\bsigma(E))$ for each
$E\in E(G)\cup L(G)$, and ${\bf u}(E)$ is tangent to $h(\Int(E))$
again for each $E\in E(G)\cup L(G)$.
\end{definition}

Associated to a realizable type is a moduli space of tropical maps of the
given type, and this dually defines a monoid called the \emph{basic monoid}:

\begin{definition}
\label{def:basic monoid}
Given a realizable type $\tau=(G,\boldsymbol{\sigma}, \mathbf{u})$, 
we define the
\emph{basic monoid} $Q_{\tau}=\Hom(Q_{\tau}^{\vee},\NN)$ of 
$\tau$ by defining its dual:
\begin{equation}
\label{eq:basic dual}
Q^{\vee}_{\tau}:=\big\{\big((p_v)_{v\in V(G)},(\ell_E)_{E\in E(G)}\big)\,|\,
\hbox{$p_{v'}-p_{v}=\ell_E \mathbf{u}(E)$ for all $E\in E(G)$}\big\},
\end{equation}
a submonoid of
\[
\prod_{v\in V(G)} \boldsymbol{\sigma}(v)_\ZZ \times \prod_{E\in E(G)}\NN.
\]
Here $\boldsymbol{\sigma}(v)_{\ZZ}$ denotes the set of integral points
of the cone $\boldsymbol{\sigma}(v)$, and
$v'$, $v$ are taken to be the endpoints of $E$ consistent
with the chosen orientation of the edge.
\end{definition}

Note that the corresponding real cone $Q_{\tau,\RR}^{\vee}$ naturally
parameterizes a universal family of tropical maps of type $\tau$,
defining for $s=\big((p_v),(\ell_E)\big)\in Q^{\vee}_{\tau,\RR}$
a tropical map $h_s:G\rightarrow \Sigma(X)$ given by
$h_x(v)=p_v$, taking an edge $E$ with endpoints $v,v'$ to the
straight line segment joining $p_v$ and $p_{v'}$ inside
$\boldsymbol{\sigma}(E)$, and taking a leg $L$ to the line
segment or ray
\[
h_s(L_p)=(h_s(v)+\RR_{\ge 0}\mathbf{u}(L_p))\cap \boldsymbol{\sigma}(L_p)
\]
inside $\boldsymbol{\sigma}(L_p)^{\gp}$. The notion of type of a tropical map leads to the notion of type
of a punctured map:

\begin{definition}
\label{Def: typr of a punctured curve}
The \emph{type} of a family of punctured curves $C^{\circ} \to W$ for
$W=(\Spec\kappa,\kappa^{\times}\oplus Q)$ a log point with $\kappa$
algebraically closed, is the type $\tau$ of tropical map
defined as follows:
\begin{enumerate}
\item Let $x\in C^{\circ}$ be a generic point, node or punctured
point, with $\omega_x\in \Gamma=\Sigma(C^{\circ})$ the corresponding cone.
Define
$\boldsymbol{\sigma}$
by mapping the vertex, edge or leg of $G$ corresponding to $x$ 
to the minimal cone
$\tau\in\Sigma(X)$ containing $h(\omega_x)$.
\item \emph{Contact orders at edges.} Let $E_q\in E(G)$
be an edge with a chosen order of vertices $v, v'$ (hence an orientation
on $E_q$). Then there is an integral tangent vector $u_q$
to $\boldsymbol{\sigma}$ such that if we consider the
tropical map $h_s:G\rightarrow \Sigma(X)$
for $s\in \Int(Q^{\vee}_{\RR})$, we have
\begin{equation}
\label{eq:punctured legs}
h_s(v')-h_s(v)= \langle s, \ell(E_q)\rangle u_q.
\end{equation}
We note that this
agrees with the $u_q$ of \eqref{eq:node contact order}.
We define $\mathbf{u}(E_q)=u_q$, noting this depends, up
to sign, on the choice of orientation of $E_q$. 
\item \emph{Contact orders at punctures}. For a leg $L_p\in L(G)$
corresponding to a puncture $p$ and vertex $v$, we set $\mathbf{u}(L_p)=u_p$,
as defined in \eqref{Eq: contact order at punctured points}.
\end{enumerate}
\end{definition}

\begin{remark}
A crucial notion of the theory is that of basicness.
See \cite[Def.\ 1.20]{GSlog} and \cite[Def.\ 2.36]{ACGSII} for the precise
definition. But roughly, the punctured map $f:C^{\circ}/W\rightarrow
X$ of the previous remark is basic provided the corresponding
family of tropical maps is universal. In particular, if
$W=(\Spec\kappa,\kappa^{\times}\oplus Q)$ is a log point and the
type of $f$ is $\tau$, then $f$ is basic if $Q=Q_{\tau}$.
A more general
punctured map $f:C^{\circ}/W\rightarrow X$ is basic if 
$f_{\bar w}:C^{\circ}_{\bar w}/\bar w\rightarrow X$ is basic for
each strict geometric point $\bar w$ of $W$.
It is the case that any punctured log map is obtained by
base-change from a basic punctured log map.
\end{remark}

\begin{remark}
We may also consider punctured maps $f:C^{\circ}/W\rightarrow\shX$
with target the Artin fan. We note that here we must remove the
assumption of stability, as stability does not make sense for maps
to the Artin fan. So we only insist on pre-stability. Note by the
discussion of \S\ref{subsubsect:the artin fan}, if $W$ is a log
point, then giving a morphism $C^{\circ}/W\rightarrow\shX$
is the same thing as giving a family of tropical maps $h:\Gamma\rightarrow
\Sigma(X)$. Thus maps to the Artin fan are purely tropical information.
\end{remark}
\subsubsection{Moduli spaces and virtual cycles}
\label{Subsec: moduli spaces}
In \cite{ACGSII}, it is proved that the moduli space
$\scrM(X)$ of basic stable punctured maps defined over $\Spec\kk$ 
is a Deligne-Mumford 
stack. Further, the moduli space $\foM(\shX)$ of basic pre-stable
punctured maps to the Artin fan $\shX$ is an algebraic stack.
There is a natural morphism $\varepsilon:\scrM(X)\rightarrow \foM(\shX)$
taking a punctured map $C^{\circ}\rightarrow X$ to the composition
$C^{\circ}\rightarrow X \rightarrow\shX$, and a relative perfect obstruction
theory for $\varepsilon$ is constructed. At a geometric point of
$\scrM(X)$ represented by a stable punctured map $f:C^{\circ}\rightarrow X$,
the relative virtual dimension of $\varepsilon$ at this point is
$\chi(f^*\Theta_{X/\kk})$, where $\Theta_{X/\kk}$ denotes the logarithmic
tangent bundle of $X$.

In general, when negative orders of tangency are allowed, $\foM(\shX)$
may be quite unpleasant. In the usual stable log map case, $\foM(\shX)$
is in fact log smooth over $\Spec\kk$, hence is toric locally in the 
smooth topology.
However,
once one allows punctures, $\foM(\shX)$ becomes only idealized log
smooth. This means that smooth locally, it looks like
a subscheme of a toric variety defined by a monomial ideal. So 
in particular, it may not be reduced or  equi-dimensional. Hence, to
get useful invariants, it is helpful to restrict to strata, as follows.

One may decompose the reduction of $\foM(\shX)$ into locally closed 
strata on which the corresponding type of map to $\shX$ is constant,
in the sense of Definition \ref{Def: typr of a punctured curve}.
Thus given a type $\tau$, one defines $\foM_{\tau}(\shX)$ as
the closure of the union of strata of type $\tau$, with the reduced
induced stack structure.\footnote{In \cite{GSCanScat}, this is written
as $\foM_{\tau}(\shX,\beta)$, however the type $\beta$ is redundant
information and we omit it here.} 
Note that in general $\foM_{\tau}(\shX)$ may be empty.
Setting
$\scrM_{\tau}(X)=\foM_{\tau}(\shX)\times_{\foM(X)} \scrM(X)$, we
may write a disjoint union
\[
\scrM_{\tau}(X)=\bigcup_{\tilde\tau=(\tau,\ul{\beta})}
\scrM_{\tilde\tau}(X)
\]
where $\scrM_{\tilde\tau}(X)$ is the open and closed substack of
$\scrM_{\tau}(X)$ parametrizing punctured curves which represent the
curve class $\ul{\beta}$. Then \cite{ACGSII} shows that $\scrM_{\tilde\tau}(X)$
is proper over $\Spec\kk$. Thus we have by restriction a perfect
obstruction theory for the map
\[
\varepsilon:\scrM_{\tilde\tau}(X)\rightarrow \foM_{\tau}(\shX),
\]
and we denote by 
\[
\varepsilon^!:A_*(\foM_{\tau}(\shX))\rightarrow A_*(\scrM_{\tilde\tau}(X))
\]
the virtual pull-back defined by Manolache \cite{Manolache}.

\subsection{The canonical scattering diagram}

Given a log Calabi-Yau pair $(X,D)$ as in \S\ref{subsubsection:affine manifold}
with tropicalization $(B,\P)$ with affine structure on $B$
as previously constructed, we review the construction of
the canonical scattering diagram on $B$ as announced in 
\cite{GSUtah} and developed in \cite{GSCanScat}. As we will be constructing a scattering diagram on $B$
using the choice of MVPL function $\varphi$ given by 
Construction \ref{constr:phi},
we need to fix the additional data of the monoid $Q\subseteq N_1(X)$,
and a monomial ideal $I\subseteq Q$ with $\sqrt{I}=\fom$.
Thus we obtain a
$Q^{\gp}_{\RR}$-valued MVPL function $\varphi$.

Fix a curve class $\ul{\beta}\in N_1(X)$, a cone $\sigma\in\P$ which
is of codimension $0$ or $1$, and a non-zero tangent vector 
$u \in \Lambda_{\sigma}$. Further, choose a realizable type 
$\tau=(G,\boldsymbol{\sigma},
\mathbf{u})$ with basic monoid $Q_{\tau}$ (Definition \ref{def:basic monoid}),
and with dual cone
$Q_{\tau,\RR}^{\vee}=\Hom(Q_{\tau},\RR_{\ge 0})$. 
We require that $\tau$ be a \emph{wall type} in the sense of
\cite[Def.~3.6]{GSCanScat}, i.e., 
\begin{itemize}
\item
$G$ is a genus zero graph with $L(G)=\{L_{\out}\}$ with
$u_{\tau}:={\bf u}(L_{\out})\not=0$.
\item $\tau$ is balanced. In other words,
for each vertex $v\in V(G)$ 
with $\bsigma(v)\in\P$ of codimension zero or one, and $x\in\Int(\bsigma(v))$
any point,
suppose we have edges or legs $E_1,\ldots,
E_m$ adjacent to $v$ oriented away from $v$. Then ${\bf u}(E_1),
\ldots,{\bf u}(E_m)$ may be viewed as elements of $\Lambda_x$,
and the equality
\[
\sum_{i=1}^m {\bf u}(E_i)=0
\]
holds in $\Lambda_{x}$.
\item Let $h:\Gamma\rightarrow \Sigma(X)$ be the corresponding
universal family of tropical maps parametrized by $Q_{\tau,\RR}^{\vee}$, 
and $\tau_{\out}\in \Gamma$
the cone corresponding to $L_{\out}$. (See discussion in the following
paragraph.) Then $\dim Q_{\tau,\RR}^{\vee}=n-2$ and
$\dim h(\tau_{\out})=n-1$. Further, $h(\tau_{\out})\not\subseteq
\partial B$.
\end{itemize}


In this case, along with $\tilde\tau=(\tau,\ul{\beta})$ we have a map
$\varepsilon:\scrM_{\tilde\tau}(X)\rightarrow \foM_{\tau}(\shX)$.
It is proved in \cite[Lem.~3.9]{GSCanScat} that 
\[
[\scrM_{\tilde\tau}(X)]^{\virt}:=\varepsilon^![\foM_{\tau}(\shX)]
\]
is a zero-dimensional cycle. Hence we may define
\begin{equation}
    \label{Eq: N-tau}
N_{\tilde\tau}:=\deg [\scrM_{\tilde\tau}(X)]^{\virt}.    
\end{equation}
Now the wall type $\tau$ gives rise to universal family of 
tropical maps of type $\tau$:
\[
\xymatrix@C=30pt
{
\Gamma\ar[r]^>>>>{h}\ar[d]_{\pi_{\mathrm{trop}}}&\Sigma(X)=B\\
Q_{\tau,\RR}^{\vee}&
}
\]
Here we are viewing $\Gamma$ as a polyhedral complex, containing a cone
corresponding to each vertex, edge, or leg of $G$.
In particular, as already mentioned in the definition of wall
type, $\Gamma$ contains an $n-1$-dimensional
cell $\tau_{\out}$, such that for $s\in \Int(Q^{\vee}_{\tau,\RR})$,
$\pi^{-1}_{\trop}(s)\cap \tau_{\out}$ is the leg $L_{\out}$ of $G$. Note also
that the vertex $v\in V(G)$ adjacent to $L_{\out}$ then gives
a cell $\tau_v\in \Gamma$ such that $\pi_{\trop}|_{\tau_v}$ is an isomorphism
of $\tau_v$ with $Q_{\tau,\RR}^{\vee}$, 
and $\tau_v$ is a face of $\tau_{\out}$.
Consider the wall
\begin{equation}
    \label{Eq: walls of canonical}
    (\fod_{\tilde\tau},f_{\tilde\tau}):=
\big( h(\tau_{\out}), \exp(k_{\tau}N_{\tilde\tau}t^{\ul{\beta}} z^{-u})\big).
\end{equation}
Here $k_{\tau}$ is defined as follows.
The map $h$ is affine linear when restricted to $\tau_{\out}$,
inducing a map of integral tangent spaces $h_*:\Lambda_{\tau_{\out}}
\rightarrow \Lambda_{\sigma}$. We then set
\begin{equation}
    \label{Eq: k_tau}
    k_{\tau} := |(\Lambda_{\sigma}/h_*(\Lambda_{\tau_{\out}}))_{\tors}|.
\end{equation}
By \eqref{eq:punctured legs}, we may also write 
\begin{equation}
\label{eq:tauout tauv}
h(\tau_{\out})=(h(\tau_v)+\RR_{\ge 0}u)\cap \sigma.
\end{equation}

Note also that $\exp(k_{\tau}N_{\tilde\tau}t^{\ul{\beta}}z^{-u})$
makes sense as an element of $(\kk[Q]/I)[\Lambda_{\sigma}]
\subseteq \kk[\shP_x^+]/I_x$ for $x\in\Int(\sigma)$.

\begin{definition}
We define the \emph{canonical scattering diagram} associated to $(X,D)$ 
modulo the ideal $I$ to be
\[
\foD_{(X,D)}:=\{(\fod_{\tilde\tau},f_{\tilde\tau})\}
\]
where $\tilde\tau=(\tau,\ul{\beta})$ runs over all wall types and
curve classes $\ul{\beta}\in Q\setminus I$.
\end{definition}

One of the main results of \cite{GSCanScat}, stated as
Theorem 3.12 in that reference, is then:

\begin{theorem} 
$\foD_{(X,D)}$ is a consistent scattering diagram.
\end{theorem}

\begin{remark}
\label{rem:system of consistent diagrams}
The definition of $\foD_{(X,D)}$ depends on $I$, but it is obvious
from the definition that if $I\subseteq I'$, then the scattering diagrams 
defined using $I$ and $I'$ are equivalent modulo $I'$. Thus we tend
to view $\foD_{(X,D)}$ as a compatible system of scattering diagrams,
as in Remark \ref{rem:equivalence}. For the most part, when we discuss
the canonical scattering diagram, generally we take a sufficiently
small ideal $I$ as needed without comment.
\end{remark}

As we will also apply this consistency result in the relative case as described
in \S\ref{subsubsection:affine manifold}, we quote another
result from \cite{GSCanScat}, using the notation in Remark \ref{rem:relative case}. Let $p:X\rightarrow\AA^1$ be in that remark. Here $\AA^1$ now carries
the divisorial log structure induced by $0\in\AA^1$, so that $p$
is a log smooth morphism.
Assume the
fibre $(X_t,D_t)$ over general $t\in \AA^1$ is maximally degenerate,
so that $(B_0,\P_0)$ is the boundary of $(B,\P)$. We thus have 
two scattering diagrams, $\foD_{(X_t,D_t)}$ and $\foD_{(X,D)}$, which we compare in what follows. 

We first note from \cite[Prop.~3.7]{GSCanScat}:

\begin{theorem}
\label{thm:relative diagram}
Let $p:X\rightarrow \AA^1$ be as above, and let 
$p_{\mathrm{trop}}:B\rightarrow \Sigma(\AA^1)=\RR_{\ge 0}$ 
be the induced affine submersion.
Then any punctured log map $f:C^{\circ}/W\rightarrow X$ 
lying in a moduli space contributing to the canonical scattering diagram,
with output leg contact order $u$,
is in fact a punctured log map defined over $S=\AA^1$. In particular,
if $W=(\Spec \kappa,\kappa^{\times}\oplus Q)$ is a log point and $f$
is of type $\tau$,
we have the commutative diagram \eqref{eq:tropical diagram}. 
If $s\in Q_{\RR}^{\vee}$ maps to $1\in \Sigma(S)=\RR_{\ge 0}$,
then the image of $h_s$ lies in $p_{\trop}^{-1}(1)$. Thus we have
$(p_{\mathrm{trop}})_*(u)=0$.
\end{theorem}

We have a natural
map $\iota:N_1(X_t)\rightarrow N_1(X)$ induced by the inclusion 
$X_t\hookrightarrow X$. Choose a monoid $Q_t\subset N_1(X_t)$ as usual
such that $\iota(Q_t)\subseteq Q$. This allows us to define, for
any $\sigma\in \P_0$ of codimension zero or one, a map
\[
\iota_*:\kk[\Lambda_{\sigma}][Q_t] \rightarrow \kk[\Lambda_{\sigma}][Q].
\]
We may then define
\begin{equation}
    \label{Eq: ioata applied to canonical scattering}
    \iota(\foD_{(X_t,D_t)})=\{(\fod,\iota_*(f_{\fod}))\,|\, (\fod,f_{\fod})\in\foD_{(X_t,D_t)}\}.
\end{equation}
\begin{definition}
In the above situation, 
the \emph{asymptotic scattering diagram} of $\foD_{(X,D)}$ is defined as 
\[
\foD_{(X,D)}^{\as} =
\{(\fod\cap p_{\mathrm{trop}}^{-1}(0),f_{\fod})\,|\,
\hbox{$(\fod,f_{\fod})
\in \foD_{(X,D)}$ with $\dim \fod\cap p_{\mathrm{trop}}^{-1}(0)=\dim X_t-1$}\}.
\]
Note that this makes sense as a scattering diagram on $B_0$:
any non-trivial monomial $ct^{\ul{\beta}}z^{-u}$ appearing in a wall function
$f_{\fod}$ must have $u$ tangent to $B_0$ by Theorem \ref{thm:relative diagram}.
\end{definition}

We then also have from \cite[Prop.~3.18]{GSCanScat}:

\begin{theorem}
\label{thm:asymptotic general}
In the above situation, assuming that $p^{-1}(0)$ is reduced,
we have $\iota(\foD_{(X_t,D_t)})=\foD_{(X,D)}^{\as}$.
\end{theorem}

\subsection{Balancing and consistency in higher codimension}

Here we gather a couple of results needed for the behaviour of
tropical curves and consistency along codimension $\ge 2$
cells of $\P$ which are not contained in the discriminant locus.
Hence fix once and for all in this section a log Calabi-Yau pair
$(X,D)$ leading to $(B,\P)$ an affine manifold with singularities.

The following generalizes the balancing statement of 
\cite[Lem.~2.1]{GSCanScat}:

\begin{theorem}
\label{thm:balancing general}
Let $f:C^{\circ}/W\rightarrow X$ be a punctured map, with
$W=(\Spec\kappa, \kk^{\times}\oplus Q)$ a log point. Let
$h_s:G\rightarrow B$ be the induced tropical map for some
$s\in \Int(Q_{\RR}^{\vee})$. Let $\rho\in \P$ satisfy the hypotheses
of Proposition \ref{prop:local affine submersion}. If $v\in V(G)$
with $h_s(v)\in\Int(\rho)$, then $h_s$ satisfies the balancing
condition at $v$. More precisely, if $E_1,\ldots,E_m$ are the
legs and edges adjacent to $v$, oriented away from $v$, then the
contact orders ${\bf u}(E_i)$ may be interpreted as elements
of $\Lambda_{h_s(v)}$, and in this group 
\[
\sum_{i=1}^m {\bf u}(E_i)=0.
\]
\end{theorem}

\begin{proof}
As each $\mathbf{u}(E_i)$ is a tangent vector to 
$\boldsymbol{\sigma}(E_i)$, which is necessarily a cone of $\P$
containing $\rho$, we may view this tangent vector as a tangent
vector to $B$ at $h_s(v)$ via parallel transport in $\Lambda$
via a path inside $\boldsymbol{\sigma}(E_i)$.

Let $C_v\subseteq C$ be the irreducible component corresponding
to the vertex $v$. By splitting $C^{\circ}$ at the nodes contained
in $C_v$ using \cite[Prop.\ 5.2]{ACGSII}, we obtain a punctured
map $f_v:C_v^{\circ}/W\rightarrow X$ which has turned these nodes
into punctures. In particular, the dual graph of $C_v^{\circ}$
consists of one vertex $v$ and $m$ legs $E_1,\ldots, E_m$. Further,
the splitting process does not change the contact orders of
legs or nodes. Thus it is enough to show that the induced
map $(f_v)_{\mathrm{trop}}:\Sigma(C_v^{\circ})\rightarrow\Sigma(X)$ is balanced
at the vertex $v$. 

To do so, first note that $f_v$ factors through the strict inclusion
$D_{\rho}\hookrightarrow X$. The log structure on $D_{\rho}$,
pulled back from that on $X$, has
an explicit description. Let $D'_{\rho}$ be the log structure
on $\ul{D}_{\rho}$ induced by the toric boundary of $D_{\rho}$.
On the other hand, $D_{\rho}\subseteq D_{i_1},\ldots,D_{i_{n-d}}$
as in the proof of Proposition \ref{prop:local affine submersion}.
Let $\shM_{i_j}$ be the restriction to $D_{\rho}$ of the divisorial log
structure on $X$ induced by $D_{i_j}$. This log structure only
depends on the restriction $\O_X(D_{i_j})|_{D_{\rho}}$. Then
\[
\shM_{D_{\rho}}=\shM_{D_{\rho}'}\oplus_{\O_{D_{\rho}}^{\times}}
\bigoplus_{j=1}^{n-d} \shM_{i_j}.
\]
Here, the push-outs are all over $\O_{D_{\rho}}^{\times}$.

Now recall from the proof of Proposition \ref{prop:local affine submersion}
the chart $\psi_{\rho}$ which identifies the closure of the
star of $\rho$ with the support of a fan $\Sigma$ in $\RR^n$.
By Remark \ref{rem:normal bundles}, if we consider
instead $X_{\Sigma}$ with its standard toric log structure,
$D_{\psi_{\rho}(\rho)}\hookrightarrow X_{\Sigma}$ the strict inclusion,
then the log structure on $D_{\psi_{\rho}(\rho)}$ has the same 
description, and the isomorphism $D_{\rho}\cong D_{\psi_{\rho}(\rho)}$ also 
holds at the log level. Thus we may instead view $f_v$ as a morphism
$f_v:C_v^{\circ}\rightarrow D_{\psi_{\rho}(\rho)}\hookrightarrow X_{\Sigma}$.
In other words, $f_v$ can be viewed as a punctured map to a toric
variety. As the chart $\psi_{\rho}$ defined in the
proof of Proposition \ref{prop:local affine submersion} gives
an integral affine identification of $\mathrm{Star}(\rho)$ and
$\mathrm{Star}(\psi_{\rho}(\rho))$, it is enough to check that
the tropicalization of $f_v:C^{\circ}_v\rightarrow X_{\Sigma}$ is
balanced at the vertex $v$. But this follows as in 
\cite[Ex.\ 7.5]{GSlog}. 
\end{proof}

In \cite{GSCanScat}, the canonical scattering diagram $\foD_{(X,D)}$
was only considered as a scattering diagram on $B$ when
$\Delta$ is the union of all codimension two cells. As we often
take $\Delta$ to be smaller here, we need to know that
the wall functions for $\foD_{(X,D)}$ still make sense in higher
codimension, away from $\Delta$.

\begin{theorem}
\label{thm:D defined deeper codim}
Let $\rho\in\P$ satisfy the hypotheses of Proposition 
\ref{prop:local affine submersion}, and let $(\fod,f_{\fod})\in
\foD_{(X,D)}$ with $\fod\cap\Int(\rho)\not=\emptyset$. If 
$x\in \Int(\rho)\cap\fod$, $y\in\Int(\fod)$ and $f_{\fod}=\sum_{m\in\shP^+_y}
 c_m z^m$, then under parallel transport to $x$, $m\in\shP^+_x$ whenever
$c_m\not=0$.
\end{theorem}

\begin{proof}
A wall in $\foD_{(X,D)}$ intersecting $\Int(\rho)$ arises from a choice of 
$\tilde\tau= (\tau,\ul{\beta})$ with the
type $\tau$ having precisely one leg $L_{\out}$, and $\boldsymbol{\sigma}(L_{\out})=\sigma\in\P$
with $\rho\subseteq\sigma$. Then $\fod=h(\tau_{\out})$. 
Further, $f_{\fod}=\exp(\alpha t^{\ul{\beta}}z^{-u})$ for some
$\alpha\in\kk$. Thus we need to show that $(-u,\ul{\beta})\in \shP^+_y$
lies in $\shP^+_x$ under parallel transport. Note that this element
of $\shP_x$ is, using 
\eqref{Eq: Pplus description}, $(-u,(d\varphi_{\rho}|_{\sigma})(-u)
+\ul{\beta})$.

First suppose that $x\not\in h(\tau_v)$. Then by
\eqref{eq:tauout tauv}, necessarily $-u$ lies in the tangent wedge
of $\sigma$ along $\rho$. Then $(-u,(d\varphi_{\rho}|_{\sigma})(-u)+
\ul{\beta})\in \shP_x^+$ by \eqref{Eq: Pplus description}.

Second, suppose $x\in h(\tau_v)$. Let $f:C^{\circ}/W \rightarrow X$
be a curve in $\scrM_{\tilde\tau}(X)$, with $W$ a geometric log point.
Such a curve must exist if the wall is non-trivial, i.e., $N_{\tilde\tau}
\not=0$.
With $v$ the vertex of $G$ adjacent to $L_{\out}$, let
$C^{\circ}_v\subseteq C^{\circ}$ be the union of irreducible components
corresponding to $v$ (as the type of $f$ may not be $\tau$ but have a contraction to $\tau$, see \cite[Def.\ 2.24]{ACGSI}, 
this may not be a single irreducible component). We then
obtain $f_v:C^{\circ}_v\rightarrow X$ by splitting, as in the proof
of Theorem \ref{thm:balancing general}. Further, $f_v$
factors through the
strict embedding $D_{\rho}\rightarrow X$. Let $D_{\rho}'$ be
the log structure on $\ul{D}_{\rho}$ induced by the toric boundary of
$D_{\rho}$. Then there is a morphism of log schemes $D_{\rho}
\rightarrow D_{\rho}'$, whose tropicalization is the map $\upsilon$
of Proposition \ref{prop:local affine submersion}.
We thus may compose to obtain a punctured
map $f_v':C^{\circ}_v\rightarrow D'_{\rho}$. If $E_1,\ldots,E_s$
are the legs and edges of $G$ adjacent to $v$, say with
$E_1=L_{\out}$, then $C^{\circ}_v$
has corresponding punctures $p_1,\ldots,p_s$, and
$m_i=\upsilon_*(\mathbf{u}(E_i))$ is the contact order of
$f_v'$ at $p_i$. By the balancing of Theorem \ref{thm:balancing general}, 
$\sum_i m_i=0$. Let $\Sigma_{\rho}$ be the fan defining $D'_{\rho}$.
Each $m_i$ lies in some minimal cone $\sigma_i\in\Sigma_{\rho}$.
Then the data of the $m_i$, $\sigma_i$ determine a class
$\ul{\beta}'_v\in N_1(D_{\rho})$ by the discussion
preceding Lemma \ref{lem:beta recovery}.
It follows from \cite[Cor.\ 1.14]{GSAssoc} that
the class $\ul\beta_v\in N_1(D_{\rho})$ of the map $f_v'$ in fact
agrees with $\ul{\beta}_v'$.
Indeed, the cited corollary precisely describes the intersection
of $\ul\beta_v$ with boundary divisors of $D'_{\rho}$
in terms of the contact orders $m_i$, and recovers the relation
\eqref{eq:Drho dot beta}.

Let $\iota_*:N_1(D_{\rho})\rightarrow N_1(X)$ be the map induced
by the inclusion. Then $\ul{\beta}=\iota_*(\ul\beta_v)+\ul\beta'$ for some
$\ul\beta'\in Q$. Now note that $m_1=\upsilon_*(u)$. By construction
of $\varphi_{\rho}$ in Proposition \ref{prop:phi tau}, it is the
pull-back of the piecewise linear map $\psi$ given in
Lemma \ref{lem:beta recovery}, (1) under $\upsilon$. Further, by Lemma
\ref{lem:beta recovery}, (3), since $m_i\in\sigma_i$ for each $i$,
\[
\ul\beta_v=\sum_i \psi(m_i).
\]
Thus, bearing in mind that $u$ is in the tangent wedge to 
$\sigma$ along $\rho$ by \eqref{eq:tauout tauv},
\begin{align*}
(d\varphi_{\rho})|_{\sigma} (-u) + \ul{\beta}
{} = & -\psi(m_1)+\ul\beta_v+\ul{\beta}'\\
{} = & \sum_{i=2}^s \psi(m_i)+\ul{\beta}'.
\end{align*}
Thus we may write 
\[
(-u,(d\varphi_{\rho})|_{\sigma}(-u)+\ul{\beta})
= (0,\ul{\beta}')+\sum_{i=2}^s \big(\mathbf{u}(E_i), \psi(\mathbf{u}(E_i))\big).
\]
This is a sum of terms in $\shP^+_x$, hence lies in $\shP^+_x$,
as desired.
\end{proof}

Finally, we check that consistency of $\foD_{(X,D)}$ implies
a consistency in higher codimension, away from $\Delta$, similar
to consistency in codimensions zero and one.

\begin{theorem}
\label{thm:consistency at deeper strata}
Let $\rho\in\P$ satisfy the hypotheses of Proposition 
\ref{prop:local affine submersion}, $\foD_{(X,D)}$ the canonical
scattering diagram on $B$. Let $x\in\Int(\rho)$. Then 
for any path $\gamma$
with endpoints in $\mathrm{Star}(\rho)\setminus\Supp(\foD_{(X,D)})$
and image in $\mathrm{Star}(\rho)\setminus\Sing(\foD_{(X,D)})$ and
which only cross walls of $\foD_{(X,D)}$
which meet $\Int(\rho)$, $\theta_{\gamma,\foD_{(X,D)}}$
is an automorphism of $\kk[\shP^+_x]/I_x$. Further if $\gamma$ is a loop,
this automorphism is the identity.
\end{theorem}

\begin{proof}
Recall that any wall-crossing automorphism associated with
$\gamma$ acts, a priori, on $\kk[\shP^+_y]/I_y$ for $y$ in the
interior of a maximal cell containing the wall. However, by parallel transport
to $x$ inside of $\mathrm{Star}(\rho)$, there is an inclusion
$\shP^+_x\subseteq \shP^+_y$. Further, by the desciption
\eqref{eq:Ix def}, under this inclusion, $I_y\cap \shP^+_x\subseteq
I_x$. Thus $\kk[\shP^+_x]/I_x$ is a subquotient of $\kk[\shP^+_y]/I_y$.
Further, by Theorem \ref{thm:D defined deeper codim}, for each
wall $\fod$ crossed by $\gamma$, $f_{\fod}$ may be viewed as 
an element of $\kk[\shP^+_x]/I_x$. Thus the wall-crossing automorphism
on $\kk[\shP^+_y]/I_y$ restricts to a wall-crossing automorphism
on $\kk[\shP^+_x]/I_x$, proving the first statement.

To prove that $\theta_{\gamma}:=\theta_{\gamma,\foD_{(X,D)}}$ is the identity
for a loop $\gamma$, it
is sufficient to prove that if $\gamma$ is a loop around a joint $\foj$
intersecting $\Int(\rho)$, then $\theta_{\gamma}$ is the identity.
If the joint $\foj$ is codimension zero or codimension one, this
is already implied by the definition of consistency, so it suffices
to show this for a codimension two joint. Thus we may pass to
$(B_{\foj},\P_{\foj})$ and $\foD_{(X,D),\foj}$. Note now that
$B_{\foj}\cong \RR^n$ as an integral affine manifold, and consistency
implies that for each $m\in\ZZ^n$ and
general $p\in B_{\foj}\setminus\Supp(\foD_{(X,D),\foj})$,
there is a function $\vartheta^{\foj}_m(p)$, and these functions are related by
wall-crossing as in Definition \ref{Def: Consistency in codim two}.
But from the construction of broken lines and the definition
of $\vartheta^{\foj}_m(p)$, it follows that $\vartheta^{\foj}_m(p)$
can be viewed as an element of $\kk[\shP^+_x]/I_x$. Thus for
a loop $\gamma$ around $\foj$, $\theta_{\gamma}(\vartheta^{\foj}_m(p))
=\vartheta^{\foj}_m(p)$. Since the $\vartheta^{\foj}_m(p)$
for $m\in\ZZ^n$ generate $\kk[\shP^+_x]/I_x$ as a $\kk[Q]/I$-module,
$\theta_{\gamma}$ hence acts as the identity, proving the result.

Note that if instead $\foj\subseteq\Delta$, this argument would not
work, as there is no common ring in which to compare the
$\vartheta^{\foj}_m(p)$.
\end{proof}

\section{Pulling singularities out and an asymptotic equivalence}
\label{Sec: Pulling singularities out}
In the remaining part of this paper we restrict our attention to log Calabi--Yau pairs $(X, D)$ obtained by blowing up a toric variety along hypersurfaces in its toric boundary. Specifically, let $X_{\Sigma}$ be a projective toric variety associated to a complete toric fan in $M_{\RR}$. Fix a tuple of distinct rays of $\Sigma$,
\begin{equation}
    \label{Eq:p}
    \mathbf{P}=(\rho_1,\ldots,\rho_s).
\end{equation}
Though all the upcoming discussion could be carried out in greater generality,
it is vastly simplifying to make the following two assumptions, analogous 
to the two-dimensional case in \cite[\S4.2]{GPS}:
\begin{enumerate}
\item $X_{\Sigma}$ is non-singular.
\item There is no $\sigma\in\Sigma$ with $\rho_i,\rho_j\subseteq\sigma$
for $i\not=j$.
\end{enumerate}
Note both of these conditions may be achieved by refining the fan $\Sigma$. Corresponding to the ray $\rho_i$ is a divisor $D_{\rho_i}\subseteq X_{\Sigma}$.
We note that by the second assumption above, $D_{\rho_i}\cap D_{\rho_j}=\emptyset$ for
$i\not=j$. Assume also given non-singular general hypersurfaces 
\begin{equation}
    \label{Eq: Hi}
H_i\subseteq
D_{\rho_i}    
\end{equation}
meeting the toric boundary of $D_{\rho_i}$ transversally. Then, $X$ is obtained by the blow-up
\begin{equation}
    \label{Eq: our log CY}
    \Bl_{H}: X \longrightarrow X_{\Sigma}
\end{equation}
with center $H$ given by the union of the hypersurfaces $H_i$. Define $D$ to be the strict transform of $D_{\Sigma}$. 
In this section, we
further describe a degeneration $(\widetilde X, \widetilde D)$ of $(X,D)$ and the associated canonical scattering diagram $\foD_{(\widetilde X,\widetilde D)}$. We show that a slice of $\foD_{(\widetilde X,\widetilde D)}$ is asymptotically equivalent to $\foD_{(X,D)}$.

\subsection{A degeneration $(\widetilde X, \widetilde D)$}
 We construct a degeneration of $(X,D)$ analogous
to that constructed in two dimensions in \cite[\S5.3]{GPS}. 
In higher dimensions, see also \cite[\S5]{mandel2019theta}.
For this, first consider the toric variety 
$X_{\overline{\Sigma}}:= X_{\Sigma} \times \AA^1$,
with associated toric fan in 
\[
\overline{M}_{\RR}:=
M_{\RR}\oplus\RR
\]
given by:
\[
\overline{\Sigma}=\bigcup_{\sigma\in\Sigma} \{\sigma\times \{0\},
\sigma \times \RR_{\ge 0}\}. 
\]
We adopt the notation that if $\sigma\in\Sigma$, then we also
view $\sigma=\sigma\times\{0\}\in\overline{\Sigma}$, and write
$\bar\sigma:=\sigma\times\RR_{\ge 0}\in \overline{\Sigma}$. Now, consider the blow-up $\mathrm{Bl}_{\mathbf{D}} \colon X_{\widetilde{\Sigma}} \to X_{\overline{\Sigma}}$
with center 
\[\mathbf{D}=\bigcup_{i=1}^s D_{\rho_i}\times\{0\}.\]
We note that as the $D_{\rho_i}\times\{0\}$ are toric strata of $X_{\Sigma}
\times \AA^1$, the blow-up with these centers is a toric blow-up. We 
review the description of the toric fan $\widetilde{\Sigma}$  as
a refinement of $\overline{\Sigma}$ 
using standard toric methods in the following construction.

\begin{construction}
\label{Constr: toric fan for blow up}
Note that the codimension two toric stratum $D_{\rho_i}\times \{0\}$ of
$X_{\overline{\Sigma}}=X_{\Sigma}\times\AA^1$ corresponds to the
cone $\bar\rho_i$. By \cite{Oda}, Prop.\ 1.26, the blow-up of 
$X_{\overline{\Sigma}}$
at $\bigcup_{i=1}^m D_{\rho_i}\times\{0\}$ is given by the star 
subdivision of $\overline{\Sigma}$ at the cones $\bar\rho_i$, $1\le i \le s$. 
Recall that this is defined as follows. Note that $\bar\rho_i$ is generated
by $(m_i,0)$ and $(0,1)$, where $m_i$ is a primitive generator of $\rho_i$.
If $\bar\rho_i\subseteq\bar\sigma\in
\overline{\Sigma}$, then we can write the generators of $\bar\sigma$ 
as $(m_i,0), (0,1), u_1,\ldots,u_r \in \overline{M}$.
One defines $\Sigma_{\bar\sigma}^*(\bar\rho_i)$ to be
the set of cones generated by subsets of $\{(m_i,0),(0,1),u_1,\ldots,u_r,
(m_i,1)\}$
not containing $\{(m_i,0), (0,1)\}$. Then we set
\[
\widetilde\Sigma:=\left(\overline{\Sigma}\setminus\bigcup_{i=1}^s 
\{\bar\sigma\in\overline{\Sigma}\,|\,\bar\sigma \supseteq \bar\rho_i\}\right)
\cup \bigcup_{i=1}^s
\bigcup_{\bar\sigma\supseteq\bar\rho_i} \Sigma^*_{\bar\sigma}(\bar\rho_i).
\]
Note that if $\bar\sigma\in \overline{\Sigma}$ does not contain any
$\bar\rho_i$, then also $\bar\sigma\in \widetilde{\Sigma}$, and we
continue to use the notation $\bar\sigma$ for this cone as a cone
in $\widetilde\Sigma$. On the other hand, if $\bar\sigma\supseteq
\bar\rho_i$, then $\bar\sigma$ is subdivided into two cones.
Explicitly, if $\bar\sigma$ is generated by $\bar\rho_i$ and additional
generators $u_1,\ldots,u_r$ as above, then the two subcones are
\[
\hbox{$\tilde\sigma:=\langle (m_i,0),(m_i,1),u_1,\ldots,u_r\rangle$ and 
$\tilde\sigma':=\langle(m_i,1), (0,1),u_1,\ldots,u_r\rangle$.}
\]
\begin{figure}
\center{\input{TildeSigma.pspdftex}}
\caption{Slices of the toric fan $\widetilde{\Sigma}$ at height zero and one for $X_{\overline \Sigma}= \PP^2 \times \AA^1$ and $\mathbf{D}=\PP^1 \times \{ 0 \}$.}
\label{Fig: TildeSigma}
\end{figure}
Note that the divisor in $X_{\widetilde\Sigma}$ corresponding to the 
ray 
\begin{equation}
\label{Eq:nui def}
\nu_i:=\RR_{\ge 0}(m_i,1)
\end{equation}
is an exceptional divisor of the blow-up, which we write as
$\PP_i\cong \PP(\mathcal{O}_{D_{\rho_i}} \oplus \mathrm{Norm}_{D_{\rho_i}/X_{\rho_i}})$, with restriction of $\mathrm{Bl}_{\mathbf D}$ exhibiting the
$\PP^1$-bundle structure:
\begin{equation}
\label{Eq:Pi bundle}
p_i:\PP_i\rightarrow D_{\rho_i}.
\end{equation}
\end{construction}
The composition of $\mathrm{Bl}_{\mathbf{D}}$ with the
projection to $\AA^1$ yields a flat family
\begin{equation}
    \label{Eq: Epsilon}
    \epsilon:X_{\widetilde\Sigma}\rightarrow\AA^1,
\end{equation}
with the property that $\epsilon^{-1}(\AA^1\setminus\{0\})=
X_{\Sigma}\times (\AA^1\setminus \{0\})$, and
$\epsilon^{-1}(0)$ is the union of $s+1$ irreducible components,
one isomorphic to $X_{\Sigma}$ and the remaining $s$ components
the exceptional divisors $\PP_1,\ldots,\PP_s$ of the blowup 
$\mathrm{Bl}_{\mathbf{D}}$.

Recall the hypersurfaces $H_i\subseteq
D_{\rho_i}$ from \eqref{Eq: Hi}, 
meeting the toric boundary of $D_{\rho_i}$ transversally. We now denote by 
$\widetilde H_i\subseteq X_{\widetilde\Sigma}$ the strict
transform of $H_i\times \AA^1 \subseteq X_{\overline{\Sigma}}$.
We note:
\begin{lemma}
\label{lem:Hrhoi}
We have 
\begin{enumerate}
\item $\widetilde H_i\cap \epsilon^{-1}(\AA^1\setminus\{0\})=
H_i\times (\AA^1\setminus\{0\})$.
\item
$\widetilde H_i\cap \epsilon^{-1}(0)$ is contained in the toric
stratum $D_{\widetilde\Sigma,\tilde\rho_i}$ corresponding to $\tilde\rho_i$ of
$X_{\widetilde\Sigma}$. Further, there is an isomorphism of
pairs
\[
(D_{\widetilde\Sigma,\tilde\rho_i},\widetilde H_i\cap \epsilon^{-1}(0))
\cong (D_{\rho_i},H_i).
\]
\end{enumerate}
\end{lemma}

\begin{proof} (1) is obvious as $\mathrm{Bl}_{\mathbf{D}}$ is an
isomorphism away from $\epsilon^{-1}(0)$. For (2), first note
that $H_i\subseteq D_{\rho_i}$, and the divisor
$D_{\rho_i}\times\AA^1\subseteq X_{\overline{\Sigma}}$ is the 
divisor corresponding to the ray $\rho_i\in\overline{\Sigma}$.
The strict transform of this divisor in $X_{\widetilde\Sigma}$ is again
the divisor corresponding to $\rho_i$. Intersecting this divisor
with $\epsilon^{-1}(0)$ gives the toric stratum $D_{\tilde\rho_i}$.

Since $\mathbf{D}\cap (D_{\rho_i}
\times\AA^1)=D_{\rho_i}\times \{0\}$, a Cartier divisor on 
$D_{\rho_i}\times\AA^1$, the strict transform of $D_{\rho_i}\times
\AA^1$ is in fact isomorphic to $D_{\rho_i}\times\AA^1$ under
$\mathrm{Bl}_{\mathbf{D}}$.
Further, as $H_i\times\AA^1\subseteq D_{\rho_i}\times\AA^1$
intersects the centre of the blow-up in the Cartier divisor
$H_i\times\{0\}$, in fact $\widetilde H_i$ is
isomorphic to $H_i\times\AA^1$ under
$\mathrm{Bl}_{\mathbf{D}}$. 
Restricting these isomorphisms to $\epsilon^{-1}(0)$ gives the 
isomorphism of pairs.
\end{proof}

Now consider a further blow-up
\begin{equation}
\label{Eq: Bl_H}
    \mathrm{Bl}_{\widetilde{H}}:\widetilde{X} \longrightarrow X_{\widetilde{\Sigma}},
\end{equation}
with center 
\[ \widetilde{ H}= \bigcup_i \widetilde H_i.\]
By composing \eqref{Eq: Bl_H} with $\epsilon$ defined in \eqref{Eq: Epsilon} we obtain the flat family  
\begin{equation}
    \label{Eq: The degeneration}
    \epsilon_{\mathbf{P}}: \widetilde{X}\longrightarrow  \AA^1.
\end{equation}
For  $t \neq 0$, we have $\epsilon_{\mathbf{P}}^{-1}(t)=X$. The central fiber is of the form
\begin{equation}
    \label{Eq: central fiber of epsilon}
   \epsilon_{\mathbf{P}}^{-1}(0)=X_{\Sigma} \cup \bigcup_{i=1}^s \widetilde\PP_i
\end{equation}
where $\widetilde\PP_i$ is the blow-up of $\PP_i$, (see \eqref{Eq:Pi bundle})
at the limits of the hypersurfaces $H_i\subseteq
D_{\rho_i}$ as $t\to 0$.
\subsection{The tropical space $\widetilde B_1$ associated to $(\widetilde X,\widetilde D)$}
\label{subsec:tropical}
Let $\widetilde X$ be the degeneration in $\eqref{Eq: The degeneration}$ and denote by $\widetilde D$ the strict transform of the toric boundary divisor $D_{\widetilde \Sigma} \subseteq X_{\widetilde \Sigma}$ under the blow-up \eqref{Eq: Bl_H}. We consider $\widetilde X$ as a log space with the divisorial log structure given by $\widetilde D$.

From the construction of \S\ref{subsubsection:affine manifold},
we obtain polyhedral cone
complexes $(\widetilde B, \widetilde\P)$ and $(B_{\widetilde\Sigma},
\P_{\widetilde\Sigma})$ from the log Calabi-Yau pairs $(\widetilde X,
\widetilde D)$ and $(X_{\widetilde\Sigma}, D_{\widetilde\Sigma})$ respectively. 
Note we may also view the fan $\widetilde\Sigma$ as a polyhedral
cone complex $(|\widetilde\Sigma|,\widetilde\Sigma)$. In what follows, for cones $\rho\in \widetilde\P$ (resp.\ $\rho\in\P_{\widetilde
\Sigma}$) we write $\widetilde D_{\rho}$ (resp.\ $D_{\widetilde\Sigma,\rho}$)
for the corresponding stratum of $\widetilde D$ (resp.\ $D_{\widetilde
\Sigma}$) as in \eqref{eq:Drho def}. 

\begin{proposition}
\label{Prop: isomorphic cone complexes}
There are canonical isomorphisms of polyhedral cone complexes:
\begin{equation}
    \label{Eq: tropicalizations}
(\widetilde B,\widetilde\P)\cong (B_{\widetilde\Sigma},\P_{\widetilde{\Sigma}})
\cong (|\widetilde\Sigma|,\widetilde\Sigma).
\end{equation}
\end{proposition}

\begin{proof}
The second isomorphism is obvious.
The first isomorphism in \eqref{Eq: tropicalizations} arises as follows.
The map $\mathrm{Bl}_{\widetilde{H}}$ in \eqref{Eq: Bl_H} induces a one-to-one correspondence
between the irreducible components of $\widetilde D$ and
the irreducible components of $D_{\widetilde\Sigma}$, with an
irreducible component $\widetilde D_i$ of $\widetilde D$ corresponding to
the irreducible component $\mathrm{Bl}_{\widetilde{H}}(\widetilde D_i)$
of $D_{\widetilde\Sigma}$. Because of the transversality
assumption on $H_i$, $\widetilde H_i$ does not contain 
any toric stratum
of $X_{\widetilde\Sigma}$. Thus any stratum of $\widetilde D$ is
the strict transform of any stratum of $D_{\widetilde\Sigma}$, showing
that the bijection between components induces an inclusion
preserving bijection between the cones of $\widetilde\P$
and $\P_{\widetilde\Sigma}$. Hence, we obtain the desired isomorphism
as abstract polyhedral cone complexes. 
\end{proof}

As a consequence of this proposition, we will freely identify cones
in $\widetilde\P$ and cones in $\widetilde\Sigma$.

Recall from the construction of \S\ref{subsubsection:affine manifold} that
$\widetilde B$ and $B_{\widetilde\Sigma}$ both carry the structure of
an affine manifold away from the codimension $\ge 2$ cones. However,
note that if a log Calabi--Yau pair is toric, given by a toric variety with
fan $\Sigma$ and its toric boundary divisor, then the associated pair $(B,\P)$ may be identified with $(|\Sigma|,\Sigma)$. In particular, the obvious
affine structure induced by $|\Sigma|\subseteq M_{\RR}$ agrees, off
of codimension two, with that constructed in \S\ref{subsubsection:affine manifold} by Remark \ref{Rem: Motuvation for toricy definition}. Thus
we may take the discriminant locus
 $\Delta=\emptyset$. So, in the case of the toric log Calabi--Yau pair $(X_{\widetilde{\Sigma}},D_{\widetilde{\Sigma}})$, the integral affine structure on $ B_{\widetilde \Sigma}$, which was constructed away from the codim $\ge 2$ cells in \S\ref{subsubsection:affine manifold}, extends to $B_{\widetilde \Sigma}$. 

As in Proposition \ref{prop:local affine submersion}, 
the affine structure on $\widetilde{B}\setminus \Delta$ 
extends to a larger subset of $\widetilde{B}$ in our situation. 
While we may use
that proposition directly to describe this extension, it is more
convenient now to give several large charts on $\widetilde B$.

For this first note that from Proposition \ref{Prop: isomorphic cone complexes} there exists a piecewise linear homeomorphism
\begin{equation}
    \label{Eq: PL homeomorphism}
    \widetilde B \lra B_{\widetilde \Sigma}.
\end{equation}
Since we have integral affine identifications
$B_{\widetilde\Sigma}=|\widetilde\Sigma|=M_{\RR}\times
\RR_{\ge 0}$, this induces a piecewise linear homeomorphism
\begin{equation}
    \label{Eq: Psi}
\Psi \colon \widetilde B  \lra M_{\RR} \times \RR_{\geq 0}.    
\end{equation}

We will now describe the affine structure on a larger subset of $\widetilde B$ explicitly 
using $\Psi$. To do so, we need one additional piece of information. 
Recall that $H_i$
is a hypersurface in $D_{\rho_i}$, the toric variety defined by the fan $\Sigma(\rho_i)$ in $M_{\RR}/\RR\rho_i$ given by
\begin{equation}
    \label{Eq: sigma rho}
\Sigma(\rho_i)=\{(\sigma+\RR\rho_i)/\RR\rho_i\,|\,\sigma\in\Sigma,
\rho_i\subseteq\sigma\}.
\end{equation}
In particular,
there is a piecewise linear function on the fan $\Sigma(\rho_i)$,
\begin{equation}
    \label{Eq: varphi-i}
    \varphi_i:M_{\RR}/\RR\rho_i \rightarrow \RR
\end{equation}
corresponding to the divisor $H_i$ defined as follows: if $H_i$
is linearly equivalent to a sum $\sum a_\tau D_{\tau}$ of boundary
divisors, where $\tau$ ranges over rays in $\Sigma(\rho_i)$,
then $\varphi_i(m_{\tau})=a_{\tau}$ for $m_{\tau}$ a primitive generator of $\tau$. Note that the $H_i$'s are Cartier divisors, as $X_{\Sigma}$ is smooth, so in particular the boundary toric divisors $D_{\rho_i}$ are also smooth. Therefore, the values $a_{\tau}$ define $ \varphi_i$ uniquely up to a linear function. 

Now, recall that we denoted by $\tilde\rho_i$ the cone of $\widetilde\Sigma$ generated
by $(m_i,0)$ and $(m_i,1)$. Further, recall from \eqref{eq:star def} 
the open star $\mathrm{Star}(\tilde\rho_i)$ of $\tilde\rho_i$.
Let $U,V \subset \widetilde{B}$ be open subsets defined as follows:
\begin{align}
\label{Eq:UV def}
\begin{split}
U:=&
\widetilde{B}\setminus \bigcup_{i=1}^m  
\bigcup_{\rho_i\subseteq \rho\in \Sigma\atop
\dim\rho=n-1} \tilde\rho,\\
V:= & \bigcup_{j=1}^m \mathrm{Star}(\tilde\rho_j).
\end{split}
\end{align}
and let 
\begin{equation}
    \label{Eq: Delta'}
    \widetilde \Delta=\widetilde{B} \setminus \left(U\cup V \right).
\end{equation}
One checks easily that $U\cup V$ contains the interior
of every codimension zero and one cone of $\widetilde\P$,
and hence $\widetilde \Delta\subset\Delta$, where $\Delta$ is the union of
all codimension two cones of $\widetilde\P$.
We will show that the integral affine structure on $\widetilde{B} \setminus \Delta$ extends to $\widetilde{B} \setminus \widetilde \Delta$.
\begin{theorem}
\label{thm: affine structure on the big part}
There is an integral affine structure on $\widetilde{B}\setminus\widetilde \Delta$
extending the integral affine structure on $\widetilde{B}\setminus
\Delta$, with affine coordinate charts given by
\[
\psi_U:U\rightarrow  \overline{M}_\RR
\]
with $\psi_U:=\Psi|_U$ where $\Psi$ is as in \eqref{Eq: Psi}, and
\[
\psi_V:V\rightarrow \overline M_{\RR}
\]
given by, for $(m,r)\in M_{\RR}\times\RR_{\ge 0}$ such that
$\Psi^{-1}(m,r)\in \mathrm{Star}(\tilde\rho_i)\subseteq V$,
\begin{equation}
\label{eq:psii formula}
\psi_V(\Psi^{-1}(m,r))=(m+\varphi_i(\pi_i(m)) m_i,r)
\end{equation}
where $\pi_i:M_{\RR}\rightarrow M_{\RR}/\RR\rho_i$ is the quotient map, and $\varphi_i$ is defined as in \eqref{Eq: varphi-i}.
\end{theorem}
\begin{proof}
While the homeomorphism $\widetilde B \rightarrow B_{\widetilde\Sigma}$ in \eqref{Eq: PL homeomorphism} is only piecewise linear
in general, we first observe that it is in fact affine linear across any
codimension one cone $\rho$ of $\widetilde{\Sigma}$ such that
neighbourhoods of $\widetilde D_{\rho}\subseteq \widetilde{X}$ and 
$\widetilde {D}_{\rho}
\subseteq X_{\widetilde\Sigma}$ are isomorphic. Now the blow-up map
$\mathrm{Bl}_{\widetilde{H}}:\widetilde{X}\rightarrow X_{\widetilde\Sigma}$ 
in \eqref{Eq: Bl_H}
induces an isomorphism in neighbourhoods of those one-dimensional strata
disjoint from any of the $\widetilde H_i$. But by Lemma \ref{lem:Hrhoi},
$\widetilde H_i$ only
intersects those one-dimensional strata corresponding to cones
$\tilde \rho\in\widetilde\Sigma$ with $\rho\in\Sigma$, 
$\dim\rho=n-1$, and $\rho\supseteq\rho_i$. Thus we immediately
see that $\Psi|_U$ is affine linear, and hence $\psi_U$ defines
an affine coordinate chart on $U$ compatible with the affine structure
on $\widetilde B\setminus\Delta$.

Now choose $\rho\in\Sigma$, $\dim\rho=n-1$, $\rho\supseteq \rho_i$.
Let $\sigma_1,\sigma_2$ be the maximal cones of $\Sigma$ containing
$\rho$. Then we wish to verify
\eqref{eq:chart relation general} for the chart \[\psi_V|_{\Int(\tilde\sigma_1
\cup\tilde\sigma_2)}.\]
Write $\tau_1,\ldots,\tau_{n-1}$
for the one-dimensional faces of $\rho$, taking
$\tau_1$ to be $\rho_i$, and let $\tau_n$, $\tau_n'$ be the
additional one-dimensional faces of $\sigma_1$ and $\sigma_2$ respectively.
We will take primitive generators $\mu_1,\ldots,\mu_n,\mu_n'$ for these
rays, noting that $\mu_1=m_i$. Then $\tilde\rho$ has one-dimensional
faces $\tau_1,\ldots,\tau_{n-1},\nu_i$ where $\nu_i=\RR_{\ge 0}(m_i,1)$,
$\tilde\sigma_1$ has one-dimensional
faces $\tau_1,\ldots,\tau_n,\nu_i$, and 
$\tilde\sigma_2$ has one-dimensional faces $\tau_1,\ldots,\tau_{n-1},
\tau_n',\nu_i$. To check compatibility of the chart $\psi_V|_{\Int(\tilde\sigma_1\cup
\tilde\sigma_2)}$ with \eqref{eq:chart relation general}, we need to show that
\[
\psi_V(\mu_n,0)+\psi_V(\mu_n',0)=-\sum_{j=1}^{n-1} (\widetilde D_{\tau_j}\cdot
\widetilde D_{\tilde\rho})\psi_V(\mu_j,0)-
(\widetilde D_{\nu_i}\cdot \widetilde D_{\tilde\rho})\psi_V(m_i,1).
\]
Here we note that $\psi_V$ is defined on all arguments in the above formula
as $\psi_V$ extends continuously in an obvious way to the closure 
of $V$ using \eqref{eq:psii formula}. Now this is equivalent
to the following:
\begin{itemize}
\label{eq:chart relation2}
    \item $\widetilde D_{\nu_i}\cdot \widetilde D_{\tilde\rho}=0$
\item $\mu_n+\mu'_n+
\big(\varphi_i(\pi_i(\mu_n))+\varphi_i(\pi_i(\mu_n'))\big)m_i
=-\sum_{j=1}^{n-1}(\widetilde D_{\tau_j}\cdot
\widetilde D_{\tilde\rho})(\mu_j+\varphi_i(\pi_i(\mu_j))m_i)$.
\end{itemize}
Without loss of generality, we may assume that
$\varphi_i|_{(\sigma_1+\RR\rho_i)/\RR \rho_i}=0$, as $\varphi_i$ is
only defined up to a choice of linear function. In this case, it
is standard toric geometry that 
$\varphi_i(\pi_i(\mu_n'))=
H_i\cdot D_{\rho}$, where the intersection number is calculated in
$D_{\rho_i}$. Further, applying \eqref{eq:chart relation general} to
$X_{\widetilde\Sigma}$, we see that
\[
(\mu_n+\mu'_n,0)= - \sum_{j=1}^{n-1} (D_{\widetilde\Sigma,\tau_j}\cdot
D_{\widetilde\Sigma,\tilde\rho}) (\mu_j,0)-(D_{\widetilde\Sigma,\nu_i}\cdot D_{\widetilde\Sigma,\tilde\rho})
(m_i,1).
\]
Putting this together, it is thus enough to show that
\begin{align}
\label{eq:D1}
\widetilde{D}_{\tau_1}\cdot \widetilde D_{\tilde\rho}
= {} & D_{\widetilde\Sigma,\tau_1}\cdot D_{\widetilde\Sigma,\tilde\rho}-
D_{\rho}\cdot H_i,\\
\label{eq:D2}
\widetilde{D}_{\tau_j}\cdot \widetilde D_{\tilde\rho}
= {} & D_{\widetilde\Sigma,\tau_j}\cdot D_{\widetilde\Sigma,\tilde\rho}, \quad
2\le j\le n-1\\
\label{eq:D3}
\widetilde D_{\nu_i}\cdot \widetilde D_{\tilde\rho}= {} & 
D_{\widetilde\Sigma,\nu_i}\cdot D_{\widetilde\Sigma,\tilde\rho}=0.
\end{align}

{\bf Proof of \eqref{eq:D1}.} 
Recall here that we had selected $\tau_1=\rho_i$. Thus $D_{\widetilde\Sigma,\tau_1},
\widetilde D_{\tau_1}$
are the strict transforms of the divisor 
$D_{\rho_i}\times\AA^1\subseteq
X_{\Sigma}\times\AA^1$ in $X_{\widetilde\Sigma}$ and $\widetilde{X}$
respectively. Since $H_i\times\AA^1\subseteq D_{\rho_i}\times\AA^1$,
we have $\widetilde H_i\subseteq D_{\widetilde\Sigma,\rho_i}$, and hence
\[
\O_{\widetilde{X}}(\widetilde D_{\rho_i})\cong \mathrm{Bl}_{\widetilde{H}}^*\O_{X_{\widetilde
\Sigma}}( D_{\widetilde\Sigma,\rho_i})\otimes \O_{\widetilde{X}}(-\widetilde E_{\rho_i}),
\]
where $\widetilde E_{\rho_i}$
denotes the exceptional divisor over $\widetilde H_i$. Now
\[
D_{\widetilde\Sigma,\rho_i}\cdot  D_{\widetilde\Sigma,\tilde\rho}
=\deg \O_{X_{\widetilde\Sigma}}(D_{\widetilde\Sigma,\rho_i})|_{D_{\widetilde\Sigma,\tilde\rho}}.
\]
Thus, also using that $\mathrm{Bl}_{\widetilde{H}}|_{\widetilde D_{\tilde\rho}}$ induces
an isomorphism between $\widetilde D_{\tilde\rho}$ and $\widetilde 
D_{\widetilde\Sigma,\tilde\rho}$, we obtain
\[
\widetilde{D}_{\rho_i}\cdot \widetilde D_{\tilde\rho}
=\deg \O_{\widetilde{X}}(\widetilde D_{\rho_i})|_{\widetilde D_{\tilde\rho}}
= D_{\widetilde\Sigma,\rho_i}\cdot  D_{\widetilde\Sigma,\tilde\rho}
-\deg \O_{\widetilde{X}}(\widetilde E_{\rho_i})|_{\widetilde D_{\tilde\rho}}.
\]
However, as $ D_{\widetilde\Sigma,\tilde\rho}$ is transverse to 
$\widetilde H_i$, we may in fact calculate
\[
\deg\O_{\widetilde{X}}(\widetilde E_{\rho_i})|_{\widetilde D_{\tilde\rho}}
=\# D_{\rho}\cap H_i=D_{\rho}\cdot H_i,
\]
where the intersection is calculated in $D_{\rho_i}$.
This gives \eqref{eq:D1}.

{\bf Proof of \eqref{eq:D2}.} 
In this case, we note that $\widetilde D_{\tau_j}$ does not contain
$\widetilde H_j$, so that 
\[
\O_{\widetilde{X}}(\widetilde D_{\tau_j})\cong \mathrm{Bl}_{\widetilde{H}}^*\O_{X_{\widetilde
\Sigma}}(D_{\widetilde\Sigma,\tau_j}).
\]
Thus \eqref{eq:D2} immediate as before.

{\bf Proof of \eqref{eq:D3}.} 
Note that the central fibre of the morphism $X_{\widetilde\Sigma}
\rightarrow \AA^1$ is the divisor
\[
 D_{\widetilde\Sigma,\RR_{\ge 0}(0,1)} +\sum_{j=1}^s D_{\widetilde\Sigma,\nu_j}.
\]
As this divisor is linearly equivalent to $0$, we thus have
\[
D_{\widetilde\Sigma,\nu_i}\sim 
- D_{\widetilde\Sigma,\RR_{\ge 0}(0,1)} - \sum_{j\not=i} D_{\widetilde\Sigma,\nu_j}.
\]
Similarly, we have
\[
 \widetilde D_{\nu_i} \sim 
-\widetilde D_{\RR_{\ge 0}(0,1)} - \sum_{j\not=i}^s \widetilde D_{\nu_j}.
\]
Now each divisor $ D_{\widetilde\Sigma,\RR_{\ge 0}(0,1)}$ and $
D_{\widetilde\Sigma,\nu_j}$, $j\not=i$ are in fact disjoint from $ D_{\widetilde\Sigma,\tilde\rho}$,
and the same is true for the $\widetilde D$'s. Thus \eqref{eq:D3} holds.
\end{proof}
\begin{figure}
\center{\input{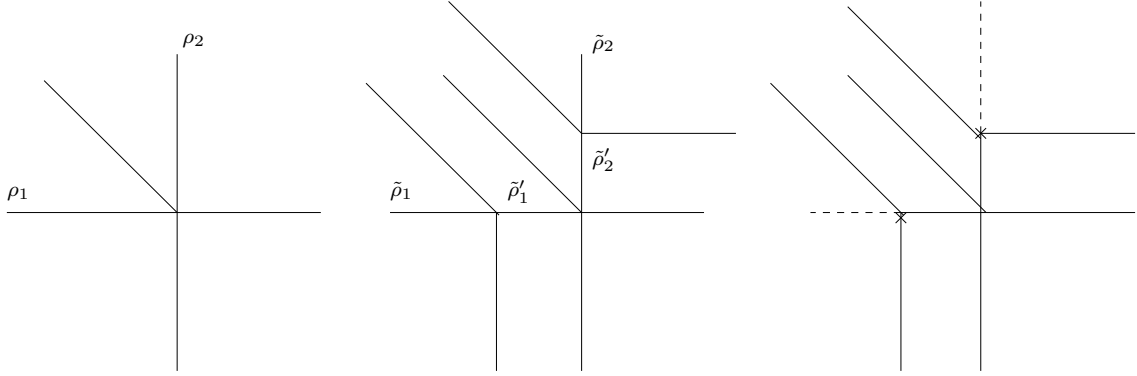}}
\caption{On the left is a two-dimensional fan $\Sigma$. The middle
shows the height $1$ slice of the fan $\widetilde\Sigma$, while
the right-hand picture shows one of the two charts for the affine manifold
$\widetilde B_1$, with the crosses indicating the discriminant locus.}
\label{Fig: Fig1}
\end{figure}

\subsection{The scattering diagram $\foD_{(\widetilde X,\widetilde D)}^1$ on $\widetilde{B}_1$}
\label{sec:The GS-locus}
Let $\widetilde{X}$ be the total space of the degeneration defined in \eqref{Eq: The degeneration}. Let 
\begin{equation}
    \label{Eq: p}
\widetilde p:\widetilde{B}\rightarrow \RR_{\ge 0} 
\end{equation}
be the canonical projection obtained by composition of the map $\Psi$ in \eqref{Eq: Psi} with the projection onto the second factor. We note that
$\widetilde p$ is the tropicalization of $\epsilon_{\mathbf{P}}:\widetilde
X \rightarrow \AA^1$. Recalling
that $\widetilde{B}$ carries the structure of an integral
affine manifold with singularities defined by Theorem \ref{thm: affine structure on the big part}, it is easy to see that
$\widetilde p$ is an affine submersion, either directly from the discription
of the affine structure or from \cite[Prop.~1.14]{GSCanScat}. In particular, the fibres of $\widetilde p$
are affine manifolds with singularities, and the fibres over $\NN\subseteq
\RR_{\ge 0}$ are integral affine manifolds with singularities. Set 
\begin{equation}
    \label{Eq: p-1(i)}
\widetilde{B}_i=\widetilde  p^{-1}(i),~~\mathrm{for~} i=0,1    
\end{equation}
and $\P_i = \{\sigma \cap\widetilde  p^{-1}(i)~| ~ \sigma \in P\}$. Note that $(\widetilde{B}_0,\P_0)$ is the tropicalization of $(X,D)$,
see \cite[Prop.~1.18,(1)]{GSCanScat}. This integral affine
manifold has discriminant locus which is a cone over the origin, and the homeomorphism $\Psi$ of \eqref{Eq: Psi} induces a piecewise linear identification
\begin{equation}
    \label{Eq: mu}
    \mu:=(\Psi|_{\widetilde p^{-1}(0)})^{-1} \colon M_{\RR} \lra \widetilde B_0
\end{equation}
identifying $\Sigma$ with $\P_0$. On the other hand,
$\widetilde{B}_1$ can be viewed as a deformation of $\widetilde
B_0$ in which the discriminant locus is pulled away from the
origin. Fibres $\widetilde p^{-1}(r)$ as $r\rightarrow\infty$ can then be seen
as a ``pulling out'' of the discriminant locus to infinity.
We frequently also use the piecewise linear identification
\[
\Psi_1:=\Psi|_{\widetilde p^{-1}(1)}:\widetilde{B}_1\rightarrow M_{\RR}.
\]

Now let us write $\foD_{(\widetilde X,\widetilde D)}$ for the canonical scattering
diagram for $\widetilde{X}$, living in $
\widetilde{B}$, and $\foD_{(X,D
)}$ for the canonical scattering diagram
for $X$, living in $\widetilde{B}_0$. Note that all vectors $u_L$ appearing in the canonical scattering diagram
satisfy $p_*(u_L)=0$ (i.e., $u_L$ is tangent to fibres of $p$) by
Theorem \ref{thm:relative diagram}. 
Given a wall $(\fod,f_{\fod})$ in $\foD_{(\widetilde X,\widetilde D)}$,
let 
\[
\ind(\fod):=|\coker(\tilde p_*:\Lambda_{\fod}\rightarrow \ZZ)|.
\]
We then obtain a wall 
$(\fod\cap \widetilde{B}_1,f_{\fod}^{\ind(\fod)})$. We set 
\begin{equation}
    \label{Eq: The GS diagram at heght one}
    \foD^{1}_{(\widetilde X,\widetilde D)}=\{(\fod\cap \widetilde{B}_1,
f_{\fod}^{\ind(\fod)})\,|\, (\fod,f_{\fod})\in \foD_{(\widetilde X,\widetilde D)}\}.
\end{equation}
Since $\foD_{(\widetilde X,\widetilde D)}$ is consistent, so is $\foD^{1}_{(\widetilde X,\widetilde D)}$ by \cite[Prop.\ 5.3]{GSCanScat}. 
Note that the exponent $\ind(\fod)$ arises as we are
reversing the procedure of \cite[Def.~4.7]{GHS}.

In the following sections we will analyze in more detail the affine geometry of
$\widetilde{B}_1$ in which $\foD^{1}_{(\widetilde X,\widetilde D)}$ lives, and in particular
describe the affine monodromy around the discriminant locus in $\widetilde{B}_1$. 

\subsubsection{The affine monodromy on $\widetilde{B}_1$}
Recall the description \eqref{Eq: Delta'} of the discriminant locus $\widetilde \Delta \subseteq \widetilde{B}$ as the complement of the union of the open sets $U$ and $V$ described in \eqref{Eq:UV def}.
We may then write
$
\widetilde \Delta  =  
\bigcup_{i=1}^m \widetilde\Delta^i
$
where
\begin{equation}
 \label{Eq: widetilde Delta}    
\widetilde\Delta^i=\left(\bigcup_{\rho_i\subseteq\rho\in\Sigma
\atop\dim\rho=n-1}\tilde\rho\right)\setminus\mathrm{Star}(\tilde\rho_j).
\end{equation}
We set
\begin{equation}
    \label{Eq: Delta-i}
\widetilde \Delta_1 \colon =  \widetilde \Delta \cap \widetilde B_1, \quad
\widetilde \Delta_1^i \colon =  \widetilde \Delta^i \cap \widetilde B_1.
\end{equation}

We may further decompose $\widetilde\Delta_1^i$ as follows.
If $\ul{\rho}\in\Sigma(\rho_i)$ is a codimension 1 cone, then 
$\ul{\rho}=(\rho+\RR\rho_i)/\RR\rho_i$ for some codimension one cone
$\rho\in\Sigma$ containing $\rho_i$. We then obtain a piece
of $\widetilde\Delta^i_1$ defined as
\begin{equation}
\label{Eq:delta rho def}
\widetilde\Delta_{\ul{\rho}}:=\left(\tilde\rho\setminus \mathrm{Star}(\tilde\rho_j)\right)\cap \widetilde{B}_1 = \tilde\rho\cap \tilde\rho'
\cap \widetilde B_1.
\end{equation}
Further,
\[
\widetilde\Delta^i_1=\bigcup_{\ul{\rho}\in\Sigma(\rho_i)} \widetilde
\Delta_{\ul{\rho}},
\]
where the union is over codimension one cones of $\Sigma(\rho_i)$.

If $\ul{\sigma}^{\pm}\in \Sigma(\rho_i)$ are the maximal cones containing
$\ul{\rho}$, with $\ul{\sigma}^{\pm}=(\sigma^{\pm}+\RR\rho_i)/\RR\rho_i$
for some maximal cones $\sigma^{\pm}\in\Sigma$, then we also note that 
the maximal cells of $\scrP_1$ containing $\Delta_{\ul{\rho}}$
are the intersections of $\tilde\sigma^+, \tilde\sigma'^+,
\tilde\sigma^-$ and $\tilde\sigma'^-$ with $\widetilde B_1$ as illustrated in blue in Figure \ref{Fig: B1}.

\begin{figure}
\center{\input{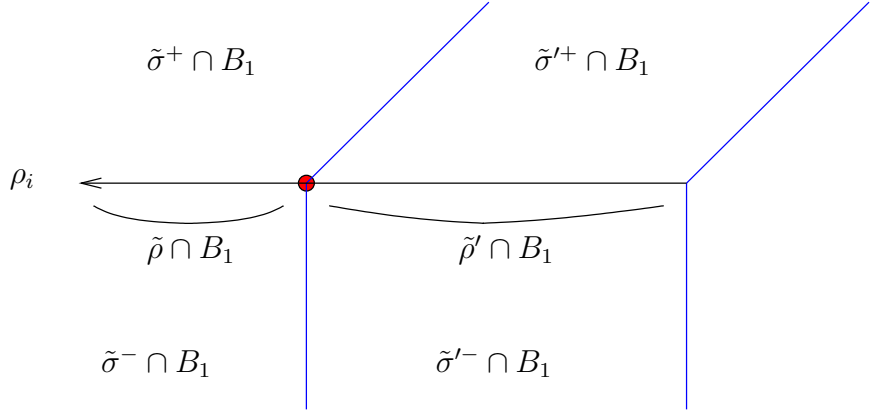}}
\caption{The discriminant locus given by a point in red in $\widetilde B_1$ where $X_{\Sigma}=\PP^2$ and $H\subset \PP^2$ is a point along a divisor.
}
\label{Fig: B1}
\end{figure}

Then the following result is immediate from the description of $\widetilde \Delta_1^i$ in \eqref{Eq: widetilde Delta}.
\begin{proposition}
\label{Eq: Delta1}
Let $\pi_i:M_{\RR}\rightarrow M_{\RR}/\RR\rho_i$ be the quotient map.
Then the restriction of $\pi_i\circ\Psi_1$ to $\widetilde \Delta_1^i  \subseteq \widetilde  B_1$ defined in \eqref{Eq: widetilde Delta} is a piecewise
affine linear isomorphism with its image, the
codimension one skeleton of $\Sigma(\rho_i)$.
\end{proposition}

We turn to the description of the affine monodromy around 
$\widetilde\Delta_{\ul\rho}$ for $\ul{\rho}$ a codimension one cone
of $\Sigma(\rho_i)$. 

Recalling from \eqref{Eq: varphi-i} the piecewise linear function $\varphi_i$ on the fan
$\Sigma(\rho_i)$, let $\kappa^i_{\ul{\rho}}$ denote the kink of 
this function along $\ul{\rho}$, see Definition \ref{def:kinky}.
To fix notation, let $\ul{n}^{\pm}\in (M/\ZZ m_i)^*$
be the slopes of $\varphi_i|_{\ul{\sigma}^{\pm}}$. Let
\begin{equation}
\label{Eq:delta def}
\delta:M/\ZZ m_i\to\ZZ
\end{equation}
be the quotient by
$\Lambda_{\ul\rho}$. Fix signs by requiring that $\delta$
is non-negative on tangent vectors pointing from $\ul{\rho}$ into
$\ul{\sigma}^+$. Then the kink
$\kappa_{\ul\rho}^i\in \ZZ$
satisfies
\begin{equation}
\label{Eqn: kink along rho}
\ul n^+ - \ul n^- =\delta \cdot\kappa_{\ul\rho}^i.
\end{equation}
We now see that the kink determines the affine monodromy around a general point of $\widetilde \Delta_{\ul{\rho}}$.

It is useful to describe a local version of
$\widetilde{B}_1$ in a neighbourhood of $\widetilde\Delta_{\ul\rho}$.
Taking $\foj=\widetilde\Delta_{\ul\rho}$, we obtain an integral affine manifold
$(B_{\foj},\P_{\foj})$ as in \S\ref{Subsubsect: Consistency around codim zero joints}. There is an obvious action by translation of the real
tangent space $\Lambda_{\widetilde \Delta_{\ul{\rho}},\RR}$ on
$B_{\foj}$, and dividing out by this action gives a two-dimensional
integral affine manifold $B'$ with one singularity $0\in B'$. We may then
choose a (non-canonical) integral affine isomorphism
$B_{\foj}\cong B'\times \Lambda_{\widetilde\Delta_{\ul{\rho}},\RR}$. Note
that for a point $x\in\Int(\widetilde\Delta_{\ul{\rho}})$, there is
a small open neighbourhood of $x$ in $\widetilde{B}_1$
isomorphic to a neighbourhood of $(0,0)\in B'\times 
\Lambda_{\widetilde\Delta_{\ul{\rho}},\RR}$.


Using this local description, identify $B'$ with $B'\times \{0\}$, and let
\begin{eqnarray}
\label{Eq: U' and V'}
\nonumber
U' & = & U \cap B', \\
\nonumber
V' & = & V \cap B'.
\end{eqnarray}
Choose $p, p'\in U'\cap V'$ with the property
that under $\pi_i\circ\Psi_1$, $p, p'$ map to $\ul{\sigma}^+$ and $\ul{\sigma}^-$ respectively. Now choose a small loop $\gamma$ in $B'$ starting at $p$, passing to $p'$ through $V'$, and then passing back to $p$ through $U'$ as illustrated in Figure \ref{Fig: monodromy}. We may then view $\gamma$
as a loop in $\widetilde{B}_1$ around the piece of the discriminant
locus $\widetilde\Delta_{\ul{\rho}}$.

\begin{figure}
\center{\input{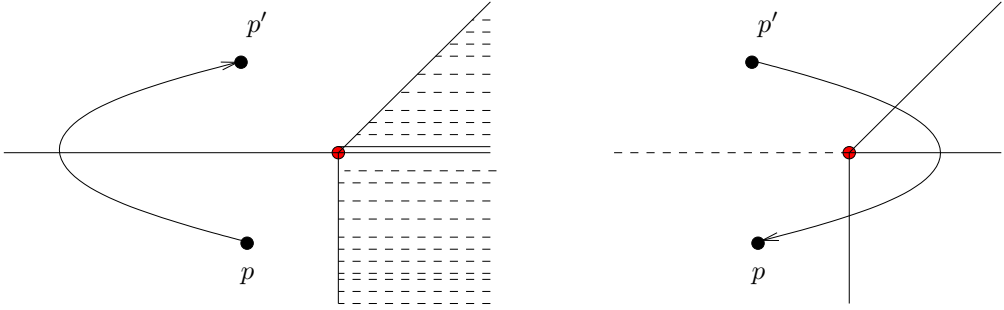}}
\caption{The left side figure illustrates the image of $\gamma$ in $V'$ and on the right hand side its image in $U'$. The dashed lines indicate the complements of the domains $U'$ and $V'$ respectively.
}
\label{Fig: monodromy}
\end{figure}

By definition the affine monodromy around $\gamma$ is obtained by taking the compositions of the linear parts of the change of coordinate functions \cite[Defn 3.1]{AP}. 
Note that the affine chart on $U$ is given by $\psi_U|_{U\cap \widetilde B_1
}= \Psi_1|_{U\cap \widetilde{B}_1}$, which allows
us to identify the tangent spaces $\Lambda_p, \Lambda_{p'}$ with $M$.
By Equation \eqref{eq:psii formula} in Theorem \ref{thm: affine structure on the big part}, 
the coordinate change between the coordinate chart $\psi_U$ and 
$\psi_V$ in a neighbourhood of $p$ or $p'$ is
\begin{equation}
\label{Eq:psiv psiu}
(\psi_V\circ \psi_U^{-1})|_{U\cap V\cap\widetilde{B}_1}(m)
=
m+\varphi_i(\pi_i(m))m_i.
\end{equation}
If $m$ lies close to $p$, then $\varphi_i(\pi_i(m))=\langle \ul{n}^+,
\pi_i(m)\rangle$, as $\pi_i(m)\in \ul{\sigma}^+$. Similarly,
if $m$ lies close to $p'$, then $\varphi_i(\pi_i(m))=
\langle\ul{n}^-, \pi_i(m)\rangle$. Putting this together, the
monodromy around $\gamma$ is then given by
\begin{equation}
    \label{Eq: Tgamma}
    T_{\gamma}(m)= m + (\ul{n}^+ - \ul{n}^-)(\pi_i(m)) \cdot m_i.
\end{equation}
This then gives, by the definition of $\kappa_{\ul \rho}^i$, the
result:

\begin{lemma}
\label{Lem: affine monodromy}
The monodromy around the loop $\gamma$ described above is given by
\begin{equation}
    \label{Eq: afffine monodromy}
T_{\gamma}(m)= m + \kappa_{\ul \rho}^i\cdot \delta(m) \cdot m_i.
\end{equation}
\end{lemma}

\begin{corollary}
\label{Corollary: monodromy invariant}
Around a generic point of the discriminant locus $\widetilde\Delta_{\ul{\rho}}$,
the tangent space to $\pi_i^{-1}(\ul{\rho})$ is monodromy invariant. 
\end{corollary}

\begin{proof}
This is immediate, as $\delta$ vanishes on the tangent space to $\ul{\rho}$
by definition.
\end{proof}

\subsubsection{The MVPL function $\varphi$ on $\widetilde{B}_1$}

We recall from Construction \ref{constr:phi} a choice of
MVPL function $\varphi$ on $\widetilde{B}$ defined via kink
$\kappa_{\rho}=[\widetilde D_{\rho}]$ for $\rho\in\widetilde\P$
of codimension one. This function takes values in $Q^{\gp}_{\RR}=
N_1(\widetilde X)\otimes_{\ZZ}\RR$, with kinks in a chosen
monoid $Q\subseteq N_1(\widetilde X)$ with $Q^{\times}=0$ and $Q$
containing all effective curve classes.

\begin{lemma}
$\varphi$ is a convex MVPL function, i.e., $\kappa_{\rho}\in Q\setminus
\{0\}$ for all codimension one $\rho \in \widetilde\P$ not contained
in the boundary of $\widetilde B$.
\end{lemma}

\begin{proof}
Let $\rho\in \widetilde{\P}$ be codimension one
not lying in the boundary of $\widetilde B$, i.e., $\widetilde B_0$.
Then $\widetilde D_{\rho}$ is contained in the central fibre of
$\epsilon_{\mathbf P}:\widetilde X\rightarrow \AA^1$, and in particular
is a proper curve. Hence $\widetilde D_{\rho}$ defines a class in
$N_1(\widetilde X)$. Note next that there is necessarily a top-dimensional
cone $\sigma\supset \rho$, and hence a unique ray $\tau$ of $\sigma$
not contained in $\rho$. Then $\widetilde D_{\tau}$ is a divisor
intersecting $\widetilde D_{\rho}$ in one point. Hence $\widetilde D_{\rho}$
cannot be numerically trivial.
\end{proof}
 
A priori, $\varphi$ is an
MVPL function on $\widetilde{B}$, but by restriction, it defines an
MVPL functions on $\widetilde B_1$. We continue
to write $\varphi$ for this restricted MVPL function.
This restriction is defined in terms of kinks. Each codimension
one $\rho\in\P_1$ is the intersection $\tilde\rho\cap\widetilde B_1$
of a codimension one $\tilde\rho\in\widetilde{\P}$. Thus we may take
the kink $\kappa_{\rho}$ to be $\kappa_{\tilde\rho}=[\widetilde D_{\tilde\rho}]$.

For $\rho\in\P_1$ with $\rho=\tilde\rho\cap \widetilde{B}_1$
for $\tilde\rho\in\widetilde\P$, we also
write $\widetilde D_{\rho}:=\widetilde D_{\tilde\rho}$,
and note that such a stratum is always contained in the fibre
$\epsilon_{\mathbf P}^{-1}(0)$.

\begin{proposition}
\label{Prop: varphi0}
Let $\overline{V}$ be the closure of $V$ defined in \eqref{Eq:UV def}. There is a single valued representative for $\varphi$ on
$\widetilde{B}_1\setminus \overline{V}$.
We write a choice of such single-valued representative as $\varphi_0$.
\end{proposition}

\begin{proof}
Recalling that $0$ is a vertex of $\scrP_1$, note that $\widetilde B_1
\setminus\overline{V}
=\mathrm{Star}(0)$. Also,
$\widetilde D_0$ is a toric variety, isomorphic to $X_{\Sigma}$.
Further, the stratification of $\widetilde X$ restricts to the toric
stratification of $X_{\Sigma}$. Thus the result follows from
Proposition \ref{prop:phi tau}.
\end{proof}
\begin{figure}
\center{\input{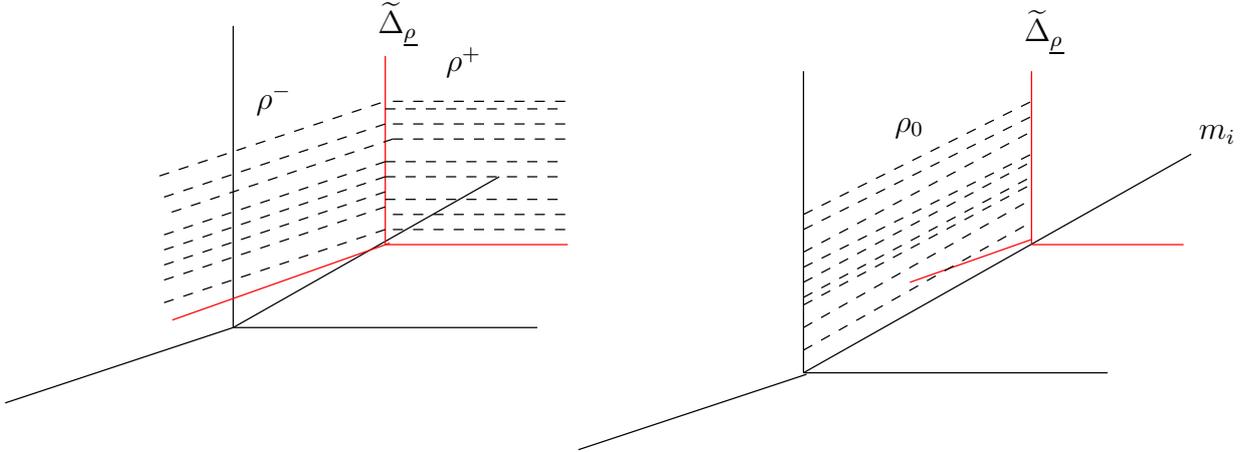}}
\caption{The codim $1$ cells $\rho^+,\rho^-$ and $\rho_0$ for $X=\PP^3$.}
\label{Fig: P3}
\end{figure}
\subsection{Geometry around the discriminant locus at $\widetilde B_1$}
\label{rem:phi disc}
We fix in this subsection $\ul{\rho}\in \Sigma(\rho_i)$ of codimension one, with
$\rho\in\Sigma$ satisfying $\ul{\rho}=(\rho+\RR\rho_i)/\RR\rho_i$.
Further, let $\sigma^{\pm}\in \Sigma$ be the maximal cones containing
$\rho$, as usual.
We have the
part $\widetilde\Delta_{\ul{\rho}}$ of the discriminant locus $\widetilde\Delta_1$ associated to $\ul{\rho}$
as defined in \eqref{Eq:delta rho def}. 
In what follows we will adopt the following notation for the codimension one cells
of $\P_1$ containing $\widetilde\Delta_{\ul{\rho}}$:
\begin{align*}
\rho_{\infty}:= {} &\tilde\rho \cap \widetilde{B}_1\\
\rho_0 := {} & \tilde\rho'\cap \widetilde{B}_1\\
\rho^+ := {} & \tilde\sigma^+\cap\tilde\sigma'^+\cap \widetilde{B}_1\\
\rho^- := {} & \tilde\sigma^-\cap\tilde\sigma'^-\cap \widetilde{B}_1,
\end{align*}
In the notation of \S\ref{sec:The GS-locus}, $\widetilde\Delta_{\ul{\rho}}$
is contained in four top-dimensional cells,
$\tilde\sigma^{\pm}\cap \widetilde{B}_1$ and $\widetilde
\sigma'^{\pm}\cap \widetilde{B}_1$.

\subsubsection{Geometry of the strata}
\label{subsubsec:geometry of strata}
As $\widetilde\Delta_{\ul{\rho}}$ contains the 
point $\nu_i\cap \widetilde{B}_1$, where $\nu_i$ is defined
in \eqref{Eq:nui def}, it follows that 
$\widetilde D_{\widetilde\Delta_{\ul{\rho}}}$
is a surface contained in $\widetilde\PP_i$, and in particular is the
inverse image of $D_{\ul{\rho}}\subseteq D_{\rho_i}$ under the
composition
\begin{equation}
\label{eq:fibre bundle}
\xymatrix@C=30pt
{
\widetilde\PP_i \ar[r]^{\mathrm{Bl}_{\widetilde{H}}}&
\PP_i \ar[r]^{p_i}& D_{\rho_i}.
}
\end{equation}
Note that the toric stratum $p_i^{-1}(D_{\ul{\rho}})$ of
$X_{\widetilde\Sigma}$ is a $\PP^1$-bundle over
$D_{\ul{\rho}}$. Denote the class of the fibre $F_i$. Further, 
this $\PP^1$-bundle has two sections, the strata 
$D_{\widetilde\Sigma,\tilde\rho}$ and 
$D_{\widetilde\Sigma,\tilde\rho'}$. 

These curve classes have a relationship in $N_1(p_i^{-1}(D_{\ul{\rho}}))$, 
which we now explain, along
with some auxiliary notation which will be used in the sequel.

Fix notation as follows. Let $u_1=m_i,u_2,\ldots,u_{n-1}$ generate
$\rho$, and take additional generators of $\sigma^+$, $\sigma^-$
to be $u_n^+$ and $u_n^-$ respectively. Thus we have generators
\begin{align*}
\tilde\rho:\quad&(m_i,1),(u_1,0),\ldots,(u_{n-1},0)\\
\tilde\rho':\quad&(m_i,1),(0,1),(u_2,0),\ldots,(u_{n-1},0)\\
\tilde\sigma^{\pm}
\cap\tilde\sigma'^{\pm}:\quad
&(m_i,1),(u_2,0),\ldots,(u_{n-1},0),(u_n^{\pm},0)
\end{align*}

Choose a primitive normal
vector $\bar n_{\rho^-}\in N_{\RR}\oplus \RR$ to the cone 
$\tilde\sigma^-\cap \tilde\sigma'^-$, positive on $(m_i,0)$.
Explicitly,
we may take $n_{\rho^-}\in N$ to be a primitive normal vector
to the span of $u_2,\ldots,u_{n-1},u_n^-$ positive on $m_i$.  We then take
$\bar n_{\rho^-}=(n_{\rho^-}, -\langle n_{\rho^-},m_i\rangle)$. 
We may similarly define $n_{\rho^+}, \bar n_{\rho^+}$.
Note that since $u_1=m_i,\ldots,u_n^{\pm}$ span a standard cone, in fact 
\begin{equation}
\label{Eq: nrho one}
\langle\bar n_{\rho^{\pm}},(m_i,0)\rangle=\langle n_{\rho^{\pm}},m_i\rangle = 1.
\end{equation}

Denote $\delta_{\ul{\rho}}:=\tilde\rho\cap\tilde\rho'$,
so that $\widetilde\Delta_{\ul{\rho}}=\delta_{\ul{\rho}}\cap \widetilde B_1$.
Using the canonical identification betwen $\widetilde\P$ and
$\widetilde\Sigma$, we also view $\delta_{\ul{\rho}}$ as
a codimension two cone in $\widetilde\Sigma$.
Let $\widetilde\Sigma(\delta_{\ul{\rho}})$ denote the quotient
fan of $\widetilde\Sigma$, i.e., the two-dimensional fan defining
the toric variety $D_{\widetilde\Sigma,\delta_{\ul{\rho}}}=
p_i^{-1}(D_{\ul{\rho}})$.

Now $\bar n_{\rho^-}$ defines a linear function on 
$\widetilde\Sigma(\delta_{\ul{\rho}})$ and hence 
$z^{\bar n_{\rho^-}}$ can be viewed as a rational function
on $D_{\widetilde\Sigma,\delta_{\ul{\rho}}}$, hence defining
a principal divisor supported on the boundary. Explicitly,
the order of vanishing of this rational function on a toric divisor
of $D_{\widetilde\Sigma,\delta_{\ul{\rho}}}$ is obtained by evaluating
$\bar n_{\rho^-}$ on primitive generator of the corresponding ray of 
$\widetilde\Sigma(\delta_{\ul{\rho}})$. The one dimensional
rays of $\Sigma(\delta_{\ul{\rho}})$ are the generated by the images of
$(m_i,0)$, $(0,1)$ and $(u_n^{\pm},0)$, corresponding to the
divisors $D_{\widetilde\Sigma,\tilde\rho}$, $D_{\widetilde\Sigma,
\tilde\rho'}$ and $F_i$ respectively.
Then, using \eqref{Eq: nrho one}, we obtain the linear equivalence relation
\begin{equation}
\label{Eq: lin equiv 1}
D_{\widetilde\Sigma,\tilde\rho}-D_{\widetilde\Sigma,\tilde\rho'}+\langle
n_{\rho^-},u_n^+\rangle F_i\sim 0.
\end{equation}

Next,
\[
\mathrm{Bl}_{\widetilde{H}}:\widetilde D_{\widetilde\Delta_{\ul{\rho}}}
\rightarrow D_{\widetilde\Sigma,\delta_{\ul{\rho}}}
\]
is the blow-up at $\kappa^i_{\ul{\rho}}$ distinct points on
the section $D_{\widetilde\Sigma,\tilde\rho}$.

We continue to write $F_i$ for the fibre class of the induced fibration
\[
p_i\circ\mathrm{Bl}_{\mathbf P}:\widetilde 
D_{\widetilde\Delta_{\ul{\rho}}}\rightarrow
D_{\ul{\rho}},
\]
and write $s_{\rho_{\infty}}$ and $s_{\rho_0}$ for the classes
of the strict transforms of the sections
$D_{\widetilde\Sigma,\tilde\rho}$ and $D_{\widetilde\Sigma,\tilde\rho'}$
respectively under the blow-up. We also write $E^j_{\ul{\rho}}$,
$1\le j\le \kappa^i_{\ul{\rho}}$ for the classes
of the exceptional curves of the blow-up. As classes in
$N_1(\widetilde{X})$, these may or may not be distinct, depending
on whether they lie over different connected components of $H_i$.
However, for bookkeeping purposes, it is convenient to distinguish
all of these classes.
Observe that the relation \eqref{Eq: lin equiv 1} now gives rise
to a linear equivalence relation
\begin{equation}
\label{Eq: lin equiv 2}
s_{\rho_{\infty}}-s_{\rho_0}+\langle n_{\rho^-},m_+\rangle F_i
\sim -\sum_{j=1}^{\kappa^i_{\ul{\rho}}} E^j_{\ul{\rho}}.
\end{equation}

Note in this notation, the kinks of $\varphi$ in $\mathrm{Star}(\widetilde
\Delta_{\ul{\rho}})$ are:
\begin{equation}
\label{eq:local kinks}
\kappa_{\rho_{\infty}}=s_{\rho_{\infty}},\quad
\kappa_{\rho_0}=s_{\rho_0},\quad
\kappa_{\rho^{\pm}}=F_i.
\end{equation}

\subsubsection{Parallel transport}
We next turn to a description of a monodromy invariant
subsheaf of $\shP$ on $\mathrm{Star}
(\widetilde\Delta_{\ul{\rho}})$. We note that 
$\Lambda_{\rho_{\infty}}$ and $\Lambda_{\rho_0}$ are monomdromy
invariant on $\mathrm{Star}(\widetilde\Delta_{\ul{\rho}})$ by
Corollary \ref{Corollary: monodromy invariant}.
So if $y\in \Int(\rho_{\infty})$, $y'\in \Int(\rho_0)$, we may
view $\Lambda_{\rho_{\infty}}\subset \Lambda_y$, 
$\Lambda_{\rho_0}\subseteq \Lambda_{y'}$, and these
two subgroups may then be identified canonically via parallel transport. 

Next, using the splittings $\shP_y=\Lambda_y\oplus Q^{\gp}$,
$\shP_{y'}=\Lambda_{y'}\oplus Q^{\gp}$ of
\eqref{Eq:trivial one} (induced by a choice of maximal cells
containing $y$,$y'$), if $m\in \Lambda_y$ is tangent to
$\rho_{\infty}$, we may parallel transport $(m,0)\in\shP_y$
to $\shP_{y'}$ through either $\rho^+$ or $\rho^-$, using
the description \eqref{eq:parallel transport}. In particular, 
the parallel transport of $(m,0)$ is then 
$(m,\langle n_{\rho^\pm},m\rangle F_i)$, with $n_{\rho^{\pm}}$, $F_i$, 
as in \S\ref{subsubsec:geometry of strata}. However, 
in the notation introduced in this previous subsection, 
$\Lambda_{\rho_{\infty}}$ is generated by $m_i,u_2,\ldots,u_{n-1}$.
Noting that by the explicit descriptions of $n_{\rho^{\pm}}$ given there, 
$\langle n_{\rho^{\pm}},u_j\rangle=0$ for $2\le j\le n-1$, while
$\langle n_{\rho^{\pm}},m_i\rangle =1$ by \eqref{Eq: nrho one}.
Thus $n_{\rho^+}$ and $n_{\rho^-}$ take the same values on
$\Lambda_{\rho_{\infty}}$.
So $(m,q)$ is invariant under monodromy in the local system $\shP$
whenever $m\in \Lambda_{\rho_{\infty}}$, $q\in Q$.

In particular, we have a parallel transport map
\begin{equation}
\label{Eq: wp def}
\wp:\kk[\Lambda_{\rho_{\infty}}][Q] \rightarrow 
\kk[\Lambda_{\rho_0}][Q]
\end{equation}
given by
\begin{equation}
\label{eq:parallel transport P}
t^qz^m\mapsto t^{q+\langle n_{\rho^{\pm}},m\rangle F_i}z^m.
\end{equation}

\subsubsection{Balancing along $\widetilde\Delta$}

While Theorem \ref{thm:balancing general} says that the tropicalization
of a punctured map to $\widetilde X$ satisfies the standard balancing
condition away from the discriminant locus, in our particular case,
we still have a weaker balancing result at the discriminant locus.

\begin{proposition}
\label{prop:delta balancing}
Let $f:C^{\circ}/W\rightarrow \widetilde X$ be a punctured map, with
$W=(\Spec\kappa, \kappa^{\times}\oplus Q)$ a log point. Let
$h_s:G\rightarrow \widetilde B$ be the induced tropical map for some
$s\in \Int(Q_{\RR}^{\vee})$. Let $v\in V(G)$ with $h_s(v)\in\widetilde
\Delta^i$ for some $i$. If $E_1,\ldots,E_m$ are the legs
and edges adjacent to $v$, oriented away from $v$, then the contact
orders $\mathbf{u}(E_j)$ determine well-defined elements 
$u_j\in\Lambda_{x}/\ZZ m_i$ for $x\in\widetilde B\setminus
\widetilde\Delta$ a point close to $h_s(v)$.
Further, we have
\[
\sum_{j=1}^m u_j=0.
\]
\end{proposition}

\begin{proof}
As each $\mathbf{u}(E_i)$ is a tangent vector to $\boldsymbol{\sigma}(E_i)$,
it determines a tangent vector to some maximal cell of $\widetilde\P$
containing $h_s(v)$. The parallel transport of this vector to
$x$ near $h_s(v)$ is then well-defined modulo $\ZZ m_i$ by 
Lemma \ref{Lem: affine monodromy}.

As in the proof of Theorem \ref{thm:balancing general}, we may
split the domain to obtain a punctured map $f_v:C^{\circ}_v
\rightarrow \widetilde X$ with $C^{\circ}_v$ the irreducible component
of $C^{\circ}$ corresponding to $v$. Now $f_v$
factors through the strict inclusion $\widetilde\PP_i\hookrightarrow 
\widetilde X$. Let $\widetilde\PP_i'$ be the log
structure on $\widetilde\PP_i$ induced by the divisor which is
the union of lower dimensional strata of $\widetilde X$ contained in
$\widetilde\PP_i$. Then we have a composition
\[
\xymatrix@C=30pt
{
f'_v:C^{\circ}_v\ar[r]& \widetilde\PP_i\ar[r]& \widetilde\PP_i'
\ar[r]^{p_i\circ \mathrm{Bl}_{\widetilde H}}& D_{\rho_i}
}
\]
as in \eqref{eq:fibre bundle}, noting all morphisms are defined at the
logarithmic level (even though $\mathrm{Bl}_{\widetilde H}$ is not). 
Note that $\Sigma(\widetilde\PP_i)$, the tropicalization
of $\widetilde\PP_i$, is just $\mathrm{Star}(\nu_i)$, where
$\nu_i$ is as in \eqref{Eq:nui def}. The tropicalization of
$\widetilde\PP_i\rightarrow D_{\rho_i}$ is given
by the composition of $\Psi|_{\mathrm{Star}(\nu_i)}:\mathrm{Star}(\nu_i)
\rightarrow M_{\RR}\times\RR_{\ge 0}$ with the map
$M_{\RR}\times\RR_{\ge 0}\rightarrow M_{\RR}/\RR m_i$ given by
$(m,r)\mapsto m\mod \RR m_i$. 
In particular, this map is induced
by the quotient map by the subspace spanned by $(m_i,0)$ and $(0,1)$.
As $D_{\rho_i}$ is toric,
the tropicalization of $f'_v$ is balanced, via \cite[Ex.\ 7.5]{GSlog}. 
As $\Lambda_x/\ZZ m_i$
may be identified with $(M\oplus\ZZ)/\ZZ (m_i,0)$, we see that
$\sum_{j=1}^m u_j=0$ holds in $\Lambda_x/\ZZ m_i$ modulo $\ZZ(0,1)$. On the other hand, the composed
map $C^{\circ}_v\rightarrow \widetilde X\rightarrow \AA^1$
also tropicalizes to a balanced curve, and hence it follows that
in fact $\sum_{j=1}^m u_j=0$ holds in $\Lambda_x/\ZZ m_i$.
\end{proof}

\subsection{$\foD^{1}_{(\widetilde X,\widetilde D)}$ is asymptotically equivalent to $\foD_{(X,D)}$ }
\label{subsec:asymp}
We now note that $\foD_{(X,D)}$ may be reconstructed from
$\foD^1_{(\widetilde X,\widetilde D)}$ via Theorem \ref{thm:asymptotic general}.

First, we must analyze the map
\[
\iota:N_1(X)\rightarrow N_1(\widetilde X)
\]
induced by the inclusion of $X$ into $\widetilde X$ as a general fibre of
$\epsilon_{\mathbf{P}}:\widetilde X\rightarrow\AA^1$. We have:

\begin{lemma}
\label{lem:iota injective}
The map $\iota$ is injective. Further, its image is the subgroup
of $N_1(\widetilde X)$ of curve classes with intersection number
$0$ with all irreducible components of $\epsilon_{\mathbf P}^{-1}(0)$.
\end{lemma}

\begin{proof}
Note that $\epsilon_{\mathbf P}^{-1}(\AA^1\setminus \{0\})=
X\times (\AA^1\setminus \{0\})$, and the inclusion
$X=X\times \{t\} \hookrightarrow X\times (\AA^1\setminus\{0\})$
for any $t\in \AA^1\setminus\{0\}$
induces an isomorphism of Picard groups via pull-back under this inclusion.
On the other hand, there is a surjective restriction map
$\Pic(\widetilde X)\rightarrow 
\Pic(\epsilon_{\mathbf P}^{-1}(\AA^1\setminus \{0\}))$ with kernel
generated by the irreducible components of of $\epsilon_{\mathbf P}^{-1}(0)$.
We obtain the result by dualizing.
\end{proof}

Now if $(\fod,f_{\fod})\in \foD^1_{(\widetilde X,\widetilde D)}$,
let $\mathrm{Cone}(\fod)\subseteq\widetilde B$ denote the cone over
$\fod$. Explicitly, $\fod$ is a polyhedral subset of
some cell $\sigma\cap\widetilde B_1 \in\P_1$ for $\sigma\in\widetilde\P$.
Then we may take $\mathrm{Cone}(\fod)\subseteq\widetilde\rho$ as the closure
of the cone generated by $\fod$ in $\sigma$. Thus we obtain
a wall 
\begin{equation}
\label{eq:cone index}
(\mathrm{Cone}(\fod),f_{\fod}^{1/\ind(\mathrm{Cone}(\fod))})\in 
\foD_{(\widetilde X,\widetilde D)},
\end{equation}
giving a one-to-one correspondence between 
$\foD^1_{(\widetilde X,\widetilde D)}$ and $\foD_{(\widetilde X,\widetilde D)}$.

Thus the following definition is intrinsic to $\foD^1_{(\widetilde X,
\widetilde D)}$:


\begin{definition}
\label{def:asymptotic scattering}
The \emph{asymptotic
scattering diagram} associated to $\foD_{(\widetilde X,\widetilde D)}^1$ is the scattering diagram
\begin{equation}
\label{Eq: height one}
 \foD_{(\widetilde X,\widetilde D)}^{1,\as} :=   \{(\fod\cap \widetilde{B}_0, f_{\fod})\,|\, \hbox{$(\fod,f_{\fod})\in
\foD_{(\widetilde X,\widetilde D)}$ with $\dim\fod\cap\widetilde{B}_0=n-1$}\}.
\end{equation}
\end{definition}

We then may restate Theorem \ref{thm:asymptotic general} in our context,
noting that by Lemma \ref{lem:iota injective},
 we lose no information in applying $\iota$:

\begin{proposition}
\label{Prop: Asymptotic equivalence}
The asymptotic scattering diagram $\foD_{(\widetilde X,\widetilde D)}^{1,\as}$ is equivalent to $\iota(\foD_{(X,D)})$.
\end{proposition}

\section{The structure of $\foD_{(\widetilde{X},\widetilde{D})}^1$ and
radiance.}
\label{sec:radiance}
This section forms the technical heart of the paper. Here we carry
out a deeper analysis of the scattering diagram 
$\foD_{(\widetilde{X},\widetilde{D})}^1$. The key point in our analysis
is a proof that this diagram is \emph{radiant}, as defined below.
Put roughly, recall we have a canonical piecewise linear identification
$\Psi|_{\widetilde B_1}$
of $\widetilde B_1$ with $M_{\RR}$. Radiance of 
$\foD_{(\widetilde{X},\widetilde{D})}^1$ can be interpreted as saying
that each wall of this scattering diagram is contained in the
inverse image of a codimension one linear subspace of $M_{\RR}$.
This will give us good control of how the scattering diagram 
behaves in radial directions from the origin. This eventually allows
us to completely describe $\foD_{(\widetilde{X},\widetilde{D})}^1$
in terms of the structure of the diagram near the origin,
see \S\ref{subsec:equivalence}.

\subsection{$\foD_{(\widetilde{X},\widetilde{D})}^1$ is a radiant scattering diagram}
The following definition of a radiant manifold can be found in \cite{goldman1984radiance}.
\begin{definition}
An affine structure on a topological manifold $B$ is \emph{radiant} if and only if the change of coordinate transformations of the atlas
defining the affine structure on $B$ lie in $GL_n(\RR) \subset \mathrm{Aff}(M_{ \RR})$. We call a radiant affine structure \emph{integral} if the change of coordinate transformations further lie in $GL_n(\ZZ)$. An (integral) \emph{radiant manifold with singularities} is a topological manifold $B$ which admits an (integral) radiant structure on a subset $B\setminus \Delta$, where $\Delta \subset B$ is a union of submanifolds of $B$ of codimension at least $2$. 
\end{definition}
The notion of radiant vector field is defined as in \cite[\S 6.5.1]{Goldman}; we do not review the general definition here. However, if $B$ is 
radiant, $B$ carries a radiant vector field
constructed as follows. Let $y_1,\ldots,y_n$ be local linear
coordinates, and set
\begin{equation}
    \label{Eq: vector field}
    \vec\rho:=\sum_{i=1}^n y_i{\partial\over\partial y_i}.
\end{equation}
This yields a global vector field on $B$ \cite[\S\ 1.5]{Goldman}. Indeed, observe that the formula for $\vec\rho$ remains invariant
under a linear change of coordinates. The following proposition can be found in \cite[\S 6.5.1]{Goldman}.
\begin{proposition}
\label{Prop: equivalent definitions of radiant}
Let $B$ be an affine manifold. Then the following are equivalent:
\begin{itemize}
 \item  M is a radiant manifold.
\item M possesses a radiant vector field.
\end{itemize}
\end{proposition}

\begin{definition} If $B$ is a radiant manifold, $B'$ an affine linear
submanifold of $B$, we say $B'$ is a \emph{radiant submanifold} if $\vec\rho$
is tangent to $B'$ at all points of $B'$. In this case, $B'$ itself
is a radiant manifold.
\end{definition}

\begin{remark} Just as a radiant manifold is modeled on a vector
space $M_{\RR}$, a radiant submanifold is modeled on a \emph{linear}
subspace of $M_{\RR}$; i.e., the only affine linear subspaces of
$M_{\RR}$ which are radiant are the linear subspaces. 
\end{remark}

\begin{lemma}
\label{lem:radiant line}
Let $B$ be a radiant affine manifold with a point $O\in B$ where
the radiant vector field vanishes.
Let $\gamma:[0,1]\rightarrow B$ be an affine linear immersion.
If $O=\gamma(1)$, then $\gamma$ is tangent to $\vec\rho$ defined as in \eqref{Eq: vector field} everywhere.
\end{lemma}

\begin{proof}
We may find $0=t_0<t_1<\cdots<t_n=1$ and charts
$\psi_i:U_i\rightarrow M_{ \RR}$ with linear transition maps,
where $U_i$ is an neighbourhood of
$\gamma(t_i)$ and $\{\gamma^{-1}(U_i)\,|\,0\le i\le n\}$ covers
$[0,1]$ by connected intervals. Now because $\gamma$ is an affine 
linear immersion, 
$(\psi_i\circ\gamma)|_{\gamma^{-1}(U_i)}$ is an affine line, and
the image is either always tangent to $\vec\rho$ or never tangent.
Now if $\gamma(1)=O$, then in the chart $U_n$, it is clear that
this image is tangent to $\vec\rho$. Hence inductively the
same is true on all of $[0,1]$.
\end{proof}

Note that for an arbitrary log Calabi-Yau pair $(X,D)$, the affine structure 
constructed on the associated $(B,\P)$ is always radiant. In our situation,
we would like to say a bit more: while $\widetilde B$ is automatically
radiant, in fact $\widetilde B_1$ also is:
\begin{proposition}
\label{prop: height one radiant}
Let $\widetilde B_1$ be as in \eqref{Eq: p-1(i)}. Then 
$\widetilde B_1$ is radiant.
\end{proposition}
\begin{proof}
The charts on $\widetilde B_1$ are the restrictions of
the charts on $\widetilde B$ described in Theorem 
\ref{thm: affine structure on the big part}, so it suffices
to check these restrictions are linear. But the transition map between $U$ and 
$\mathrm{Star}(\tilde\rho_i)$ is given by
\[
\psi_V\circ \psi_U^{-1}(m,1)=(m+\varphi_i(\pi_i(m)) m_i,1).
\]
This is defined on $\widetilde B_1\cap\tilde\sigma$ for $\tilde\sigma\in\widetilde
\Sigma$ a maximal cone containing $\tilde\rho_i$, and $\varphi_i$
is linear on $(\sigma+\RR\rho_i)/\RR\rho_i$. Thus
$m\mapsto \varphi_i(\pi_i(m)) m_i$ is a linear map on $M_{\RR}$, and hence the
transition maps are indeed linear after restricting to $\widetilde B_1$.
\end{proof}

\begin{definition}
\label{Def: radiant scattering diagram}
We say a scattering diagram $\foD$ on a radiant affine manifold $B$
is \emph{radiant} if for each wall $(\fod,f_{\fod})\in\foD$,
the affine submanifold $\fod$ is a radiant submanifold.
\end{definition}

\begin{remark}
\label{rem:affine sub our case}
In the case of $\widetilde B_1$, note that
the piecewise linear map $\Psi|_{\widetilde B_1}:\widetilde B_1
\rightarrow M_{\RR}$
preserves the radiant vector fields where $\Psi$ is linear. Indeed, the chart
$\psi_U$ on $U$ is the restriction of $\Psi$, and hence identifies
radiant vector fields on $U\cap \widetilde B_1$ and $M_{\RR}$.
As a consequence,
an affine subspace $B'\subseteq \widetilde B_1$ is radiant
if and only $\Psi|_{\widetilde B_1}(B')\subseteq M_{\RR}$ spans
a linear subspace of $M_{\RR}$ of the same dimension as $B'$. This is even the case when $B'$ is contained in the locus
$\widetilde B_1\setminus U$,
where $\Psi$ is not linear, as this set is a union of
cells of the form $\rho_{\infty}$ (in the notation of
\S\ref{rem:phi disc}) for $\rho\in\Sigma$
a codimension one cone containing $\rho_i$ for some $i$. 
However, if $B'\subseteq
\rho_{\infty}$, then $B'$ is radiant if and only if
$\psi_V|_{\widetilde B_1}(B')\subseteq M_{\RR}$ spans
a linear subspace of the same dimension as $B'$,
if and only if
$\Psi|_{\widetilde B_1}(B')\subseteq M_{\RR}$
spans a linear subspace of the same dimension as $B'$, as follows
since $\psi_V\circ\psi_U^{-1}$ is linear.
\end{remark}

Our immediate goal is to prove:

\begin{theorem} 
\label{Thm: radiant scatter}
The scattering diagram $\foD_{(\widetilde{X},\widetilde{D})}^1$ on $\widetilde B_1$ is
equivalent to a radiant scattering diagram.
\end{theorem}

\begin{remark}
Recall that the scattering diagram $\foD_{(\widetilde{X},\widetilde D)}^1$
involves a choice of an ideal $I$. So the above statement should be
interpreted as saying this holds for every ideal $I$.
\end{remark}

\begin{remark}
Morally, there is a simple reason to expect the radiance of
$\foD_{(\widetilde{X},\widetilde{D})}^1$. Indeed, suppose instead
of using the degeneration $\epsilon_{\mathbf P}:\widetilde X\rightarrow
\AA^1$, we used a trivial degeneration
$\epsilon:X\times\AA^1\rightarrow \AA^1$. Then the the
affine manifold associated to
$\big(X\times \AA^1, (D\times \AA^1)\cup (X\times\{0\})\big)$ is just
$B\times \RR_{\ge 0}$. The corresponding 
scattering diagram is simply obtained from $\foD_{(X,D)}$
by replacing each wall $(\fod,f_{\fod})$ with $(\fod\times \RR_{\ge 0},
f_{\fod})$. It is then clear that restriction to $B\times \{1\}$ gives
a radiant scattering diagram. 

Now in fact $\widetilde X$ and $X\times\AA^1$ are equal after
removing their respective boundaries, and this equality extends to
a birational map between $\widetilde X$ and $X\times\AA^1$. 
Experience from \cite{ghkk} suggests that scattering
diagrams associated with such birationally equivalent spaces should
be closely related via moving worms (see \cite[\S3.3]{KS}), 
and in particular have the same support. However,
the necessary theory of birational invariance of punctured invariants
has not been developed in a way which would permit us to prove such
a result directly.
\end{remark}

Before we embark on several lemmas for the proof, we set up notation for
proving this inductively.

\begin{definition}
\label{Def: repulsive}
Let $\foD$ be a scattering diagram and let $(\fod,f_{\fod})\in\foD$ be a wall. We say a facet (i.e., a codimension one face)
$\foj\subseteq \fod$ is \emph{repulsive} if
$f_{\fod}=f_{\fod}(z^{-v})$ and viewing $v\in \Lambda_{\fod}$ as a
tangent vector at a point $y\in\foj$, we have $v$ pointing into $\fod$. 
\end{definition}
\begin{lemma}
\label{lem:canonical repulsive}
Every wall of $\foD_{(\widetilde X,\widetilde D)}^1$
has a repulsive facet.
\end{lemma}
\begin{proof}
This is immediate from the construction of the canonical scattering
diagram, and in particular \eqref{Eq: walls of canonical}
and \eqref{eq:tauout tauv}.
\end{proof}

\begin{proposition}
\label{Prop: repulsive radiance suffices}
If $(\fod,f_{\fod})\in\foD^1_{(\widetilde X,\widetilde D)}$ with one of 
the repulsive facets of $\fod$ radiant, then $\fod$ is radiant.
\end{proposition}

\begin{proof}
Let $\foj\subseteq\fod\subseteq\sigma\in\P_1$ be a repulsive facet. 
Recall that $\Psi|_{\widetilde B_1}$ induces a linear embedding of
$\sigma$ in $M_{\RR}$. Using this embedding, we may describe
$\fod$, locally near a point $y\in\foj$, as $\foj+\RR_{\ge 0}v$. Further,
under this embedding, the radiant vector field at a point $x\in M_{\RR}$
is just $x$. Thus if $y'\in\foj$ near $y$
then the radiant
vector at $y'+\epsilon v$ is $y'+\epsilon v$. 
By assumption, $y'$ is tangent to $\foj$, and thus
$y'+\epsilon v$ is tangent to $\fod$.
\end{proof}

\begin{definition}
Fix once and for all a divisor $A$ on $\widetilde X$
which is relatively ample for the morphism $\epsilon_{\mathbf P}:\widetilde X
\rightarrow \AA^1$ of \eqref{Eq: The degeneration}. We assume that $A^{\perp}\cap Q=\{0\}$; if necessary, we may shrink $Q$. 
Define 
\[
\fom_k:=\{p\in Q\,|\, A\cdot p \ge k\}.
\]
Note this is a monomial ideal in $Q$, and that $\fom_1=Q\setminus \{0\}$
is the maximal monomial ideal of $Q$. We will also denote by $\fom_k$ the corresponding ideal in $\kk[Q]$. If $p\in Q$ or $\alpha=ct^p\in \kk[Q]$, then we say the \emph{order of $p$
or $\alpha$} is $A\cdot p$.
We say a wall $(\fod,f_{\fod})$
is \emph{trivial to order $k$} if
\[
f_{\fod}\equiv 1 \mod \fom_{k+1}.
\]
Otherwise we say the wall is \emph{non-trivial to order $k$}.
\end{definition}

\begin{remark}
Since any ideal $I$
with $\sqrt{I}=\fom$ must contain $\fom_k$ for some sufficiently large
$k$, it is enough to prove Theorem \ref{Thm: radiant scatter}
for $I=\fom_k$, which we do inductively on $k$. Since the statement
is only claimed to be true up to equivalence, we now make a good choice of 
equivalent scattering diagram. To do so, we make use of the following
general observation:
\end{remark}

\begin{lemma}
\label{lem:product uniqueness}
Let $M=\ZZ^n$, $P$ be a monoid equipped with a monoid homomorphism
$r:P\rightarrow M$, and let $\fom_P=P\setminus P^{\times}$.
Write $\widehat{\kk[P]}$ for the completion of $\kk[P]$ with respect
to the monomial ideal $\fom_P$. Let
$f\in\widehat{\kk[P]}$ satisfy $f\equiv 1 \mod \fom_P$. Then there is a unique
convergent infinite product expansion
\begin{equation}
\label{Eq: infinite product}
f=\prod_{m\in M_{\mathrm{prim}}\cup\{0\}}f_m,
\end{equation}
where $M_{\mathrm{prim}}$ denotes the set of primitive elements of
$M$, and $f_m\in\widehat{\kk[P]}$ 
has the properties that (1) every monomial $z^p$ appearing in $f_m$
satisfies $r(p)$ positively proportional to $m$ and (2) $f_m\equiv 1
\mod \fom_P$.
\end{lemma}

\begin{proof}
We show this modulo $\fom_P^k$ for each $k\ge 1$. The base case,
$k=1$, is vacuous. Suppose given a unique product expansion of
$f$ modulo $\fom_P^k$ given as
$\prod f_{m,k}$ as in \eqref{Eq: infinite product}
with $f_{m,k}=1$ for all but a finite number of $m$.
Then modulo $\fom_P^{k+1}$, we have
\[
f-\prod_m f_{m,k} = \sum_i c_iz^{p_i}
\]
for $c_i\in\kk$ and $p_i$ distinct elements of $P$ with $z^{p_i}\in
\fom_P^k\setminus \fom_P^{k+1}$. We then have no choice but to modify
the product expansion by taking
\[
f_{m,k+1} = f_{m,k}+\sum_{i\,:\,r(p_i)\in \ZZ_{>0} m} c_iz^{p_i}.
\]
\end{proof}

\begin{construction}
\label{Constr: GSk}
For any $k$,
let $\foD_{(\widetilde X,\widetilde D),k}^1$ denote the scattering
diagram with respect to the ideal $\fom_k$. Recall that this is a finite set. 
By Lemma \ref{lem:product uniqueness}, we may
replace $\foD_{(\widetilde X,\widetilde D),k}^1$ with a diagram equivalent modulo $\fom_{k+1}$ with
the property that for each $x\in \Supp(\foD_{(\widetilde X,\widetilde D),k}^1)\setminus \Sing(\foD_{(\widetilde X,\widetilde D),k}^1)$,
and for each integral and primitive $v$ tangent to $\Supp(\foD_{(\widetilde X,\widetilde D),k}^1)$ at $x$,
there is at most one wall whose attached function involves monomials
of the form $t^{\beta}z^{-nv}$ for $n>0$. Further, for any such
$x$, the set of wall functions of walls containing $x$ are then
uniquely determined.
\end{construction}

We first consider the consequences of consistency along the discriminant
locus. This is slightly delicate as the definition of consistency there
makes use of broken lines.

\begin{lemma}
\label{Lem: walls lie on radiant directions}
Fix a general point $x$ on the discriminant locus $\widetilde \Delta_1^i
\subseteq\widetilde B_1$. Suppose that $m_i$ is tangent to any wall of $\foD_{(\widetilde X,\widetilde D),k}^1$ 
containing $x$. Then if
$\fod\in \foD_{(\widetilde X,\widetilde D),k+1}^1$ with $x\in\fod$ contained in a repulsive
facet of $\fod$, we have $m_i$ tangent to $\fod$. In particular,
$\fod$ is radiant.
\end{lemma}

\begin{proof}
In the notation of \eqref{Eq:delta rho def},
let $\widetilde\Delta_{\ul{\rho}}\in \scrP_1$ be the codimension two
cell containing the point $x$, for some codimension one cone 
$\ul{\rho}\in\Sigma(\rho_i)$.
Necessarily there is a joint $\foj$ of $\foD_{(\widetilde X,\widetilde D),k+1}^1$
with $x\in\Int(\foj)\subseteq \Delta_{\ul{\rho}}$, by the genericity 
assumption on $x$.

We need to use the definition of consistency at codimension two joints
as reviewed in \S\ref{Subsubsec: Consistency around codim two joints}.
In particular, the affine manifold $B_{\foj}$ has already been 
discussed in \S\ref{sec:The GS-locus}, with a natural quotient
$B'$, which is an affine surface with singularity.
We also make use of the notation introduced in \S\ref{rem:phi disc},
and we write $\rho_{\pm}, \rho_{\infty},\rho_0$ also
for the corresponding codimension one cells of $\P_{\foj}$.
The affine manifold $(B_{\foj},\scrP_{\foj})$ carries
a scattering diagram $\foD_{(\widetilde X,\widetilde D),k+1}^{1,\foj}$. 

Write
\[
\nu:B_{\foj}\rightarrow B'
\]
for the affine submersion obtained by dividing out by the
$\Lambda_{\widetilde\Delta_{\ul{\rho}},\RR}$ action on $B_{\foj}$.
This map will be useful for visualizing broken lines in $B_{\foj}$.
It is immediate from Lemma \ref{Lem: affine monodromy}
that the monodromy around the unique singularity of $B'$ takes
the form $\begin{pmatrix} 1& a \\0&1\end{pmatrix}$ for some
integer $a$, and that $\overline m_i:=\nu_* m_i$ is a monodromy invariant
tangent vector. The polyhedral decomposition $\scrP_{\foj}$
projects to a polyhedral decomposition $\scrP'$ of $B'$. 
Note that $\nu(\rho_{\infty})$ is a ray generated by $\overline m_i$
and $\nu(\rho_0)$ is a ray generated by $-\overline m_i$.

For each wall $\fod\in \foD_{(\widetilde X,\widetilde D),k+1}^{1,\foj}$, 
$\nu(\fod)$ is a ray in $B'$.
Let $L=\rho_{\infty}\cup \rho_0\subseteq B_{\foj}$,
so that $B_{\foj}\setminus L$ consists
of two connected components, which we write as $U$ and $U'$, say
with $U$ intersecting $\rho^+$.
\begin{figure}
\center{\input{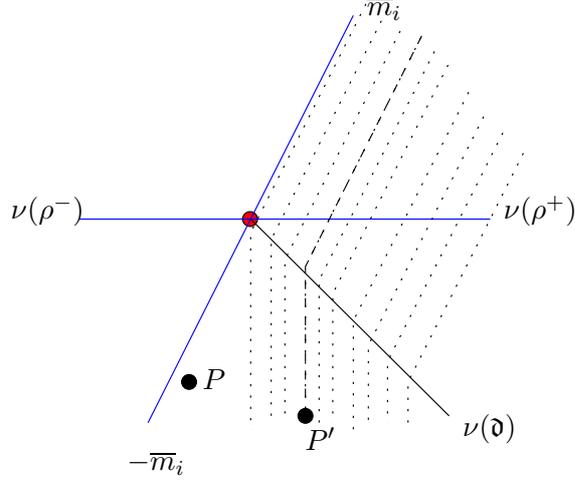}}
\caption{The polyhedral decomposition $\P'$ on $B'$, the projection of a wall $\fod \in B_j$ and a broken line $\nu\circ\gamma$ ending at $P'$ along with its possible translations.}
\label{Fig: walls}
\end{figure}
In Figure \ref{Fig: walls}, we depict $B'$, with the blue rays
being one-dimensional cones in $\scrP'$ and the black rays being
additional images of walls in $\foD_{(\widetilde X,\widetilde D),k+1}^{1,\foj}$. Note that a wall may
coincide with a codimension one cell of $\scrP_{\foj}$, and the image
of such a wall coincides with a one-dimensional cell in $\scrP'$.
Now let us assume there is a wall $\fod\in \foD_{(\widetilde X,\widetilde D),k+1}^{1,\foj}$
for which the unique face of $\fod$ (the tangent space to
$\Delta_{\ul{\rho}}$) is repulsive. Suppose further
that $m_i$ is not tangent to $\fod$. We wish to arrive at a contradiction
by showing that $\foD_{(\widetilde{X},\widetilde{D}),\ell+1}^1$ 
is not consistent for some sufficiently large $\ell$. Note
that the hypothesis of the lemma implies that all walls of
any such $\foD_{(\widetilde{X},\widetilde{D}),\ell+1}^1$ non-trivial
modulo $\fom_k$ are contained in $L$.

Without loss of generality, assume that $\fod$ is contained in the
closure of $U$. The assumption that $m_i$ is not tangent to $\fod$ tells
us that $\fod\not\subseteq L$. We now construct a specific broken line 
$\gamma:(-\infty,0]\rightarrow B_{\foj}$ for
the scattering diagram $\foD_{(\widetilde X,\widetilde D),k+1}^{1,\foj}$. This broken line 
will bend only
at $\fod$, and will have asymptotic direction either $m_i$
or $-m_i$. If $\fod\subseteq\rho^+\in\scrP_{\foj}$, this sign may be chosen 
arbitrarily.  Otherwise we choose the sign so that the broken line crosses
$\rho^+$ before bending. 
In Figure \ref{Fig: walls},
we depict the image of $\nu\circ\gamma$ of such a broken line as a
dashed piecewise linear path. 

Again without loss of generality, assume the asymptotic direction
is $m_i$; the argument is identical if the asymptotic direction
is $-m_i$.
First, the initial monomial attached to the broken line is then
$z^{m_i}\in\kk[\shP^+_x]$ for $x$ any point in the interior of the maximal
cell of $\scrP_{\foj}$ corresponding to $\tilde\sigma_+$. 
Here we use the description $\shP^+_x=\Lambda_x\oplus Q$ of \eqref{eq:P+x max}
and write for $(m,q)\in\shP^+_x$ the corresponding monomial as usual
as $t^q z^m$. Further, using the chart $\psi_U$ of Theorem 
\ref{thm: affine structure on the big part}, we may identify
$\Lambda_x$ with $M$. When we cross $\rho^+$,
we need to use parallel transport of the
monomial $z^{m_i}$ to the new cell using 
\eqref{eq:parallel transport P}.
Assuming $\fod\not=\rho^+$, then as we are not bending
at $\rho^+$, we may use \eqref{eq:parallel transport P} to accomplish
this transport, and the monomial $z^{m_i}$ becomes
$t^{\langle n_{\rho^+},m_i\rangle F_i }z^{m_i}$.
By \eqref{Eq: nrho one}, this is $t^{F_i}z^{m_i}$.

We now continue our broken line until we reach $\fod$.
Write
\[
f_{\fod}(z^{-v})=1+\alpha z^{-v}+\cdots
\]
with
$\alpha \in\fom_{k+1}\setminus \fom_{k+2}$, i.e., $\alpha$ of order
$k+1$. Then 
the monomial attached to this broken line after the bend is of the form
$\langle n_{\fod},m_i\rangle \alpha t^{F_i}z^{m_i-v}$, with $n_{\fod}$
a primitive normal vector to $\fod$ positive on $m_i$. 
In particular, as the tangent vector to the broken line
after the bend is $-(m_i-v)=-m_i+v$, the fact that $\fod$ is
repulsive shows the broken line bends away from $L$, as shown in 
Figure \ref{Fig: walls}. 

If $\fod$ coincides with $\rho_+$, we obtain similar behaviour, but
the change from the kink and the bend occur at the same time.

Note that because $F_i\in \fom_1$, in fact $\alpha t^{F_i}\in
\fom_{k+2}$. Thus we need to pass to higher order to detect the
malicious effect of this broken line. Let $\ell=k+1+A\cdot F_i$ be the order of 
$\alpha t^{F_i}$, and consider
now the scattering diagram $\foD_{(\widetilde X,\widetilde D),\ell+1}^{1,\foj}$. 

We next will prove:

\medskip

\emph{Claim.} For any point $P'$ lying between $\fod$
and $\rho_0$, any broken line for the scattering diagram
$\foD^1_{\widetilde X,\ell+1}$ with asymptotic direction $m_i$ ending at $P'$
with final monomial $\beta z^m\not\in \fom_{\ell+1}$ bends at most once
and is entirely contained in $U$.
Further the only possible broken line with final monomial
$\alpha'' z^{m_i-v} \not\equiv 0 \mod \fom_{\ell+1}$ is the one just
described.

\begin{proof}[Proof of claim.]
First consider a broken line entirely contained
in $U$.  It may only bend at most once
and still yield a coefficient non-trivial modulo
$\fom_{\ell+1}$. Indeed, if it bends twice, the coefficient
will be divisible by $t^{F_i} \alpha_1\alpha_2$, where the $t^{F_i}$
term arises from crossing $\rho^+$ and $\alpha_1,\alpha_2\in\fom_{k+1}$.
Thus the order of $t^{F_i}\alpha_1\alpha_2$ is at least 
$\ell+k+1\ge \ell+1$, so this term is trivial modulo $\fom_{\ell+1}$.

Next suppose that the broken line crosses between $U$ and $U'$. 
Such a line may start in $U'$ and cross into $U$, or worse start
in $U$, cross into $U'$, and then cross back into $U$. Regardless,
for this to happen, it must bend at least once at a wall
not contained in $L$, it must cross
either $\rho^+$ or $\rho^-$, and it must cross $L$. As a result,
the final coefficient must be divisible by $t^{\kappa_{\rho_{\pm}}}
t^{\kappa'}\alpha$, where as before $\kappa_{\rho_{\pm}}$ is the
kink of $\varphi$ along either $\rho^+$ or $\rho^-$, both of these
being the class $F_i$, see \eqref{eq:local kinks}. Further, $\kappa'$ 
comes from the kink along either $\rho_0$ or $\rho_{\infty}$, 
while $\alpha\in
\fom_{k+1}$ arises from the bend. However, as the order of $\kappa'$
is at least one, the whole term has order at least $\ell+1$, and
hence the attached final monomial lies in $\fom_{\ell+1}$.

Now observe that if a broken line is wholly contained in $U$
and bends only once, with an attached final monomial
$\alpha'' z^{m_i-v}\not\in \fom_{\ell+1}$, it must
bend at a wall whose attached function is $f(z^{-v})$. However,
$\fod$ is the unique such wall, given the assumption of Construction
\ref{Constr: GSk} that there is at most one wall with such an attached
function.
\end{proof}
Now consider a
point $P\in U\setminus (\fod+\RR_{\ge 0}(v-m_i))$ which lies between $\fod$ and $\rho_0$. Such
a point exists as $v$ is not tangent to $L$ as $L\cap\fod$ is repulsive a
repulsive facet of $\fod$. 
Note that there is a broken line of the type constructed bending
at $\fod$ and ending at $P'$ for any $P'\in \fod+\RR_{\ge 0}(v-m_i)$. However, there is no such broken line
ending at $P$. From this and the claim, we can then deduce that
\[
\vartheta^{\foj}_{m_i}(P') = t^{F_i} z^{m_i}+\alpha t^{F_i} z^{m_i-v} +\sum_j \alpha_j
z^{v_j}
\]
with $\alpha_j\in \fom_\ell$, $v_j\in M$. The first term comes from 
a straight line, the second from the constructed broken line,
and the sum comes from any other
broken lines which bend once and end at $P'$. Further, $v_j\not=m_i-v$
for any $j$.
On the other hand,
\[
\vartheta^{\foj}_{m_i}(P) = t^{F_i} z^{m_i} + \sum_j \alpha'_j z^{v_j'},
\]
again with $\alpha_j'\in \fom_\ell$, $v_j'\in M$, with $v_j'\not=
m_i-v$ for any $j$.

But we also have, for a path $\delta$ connecting $P$ to $P'$,
\[
\theta_{\delta,\foD_{(\widetilde X,\widetilde D),\ell+1}^{1,\foj}}(\vartheta^{\foj}_{m_i}(P'))
=\vartheta^{\foj}_{m_i}(P)
\]
by consistency, Definition \ref{Def: Consistency in codim two}.
However, as $\delta$ only crosses walls of
the form $(\fod', f_{\fod'}(z^{v'}))$ with $f_{\fod'}\equiv 1\mod
\fom_k$, we see easily that $\theta_{\delta,\foD_{(\widetilde X,\widetilde D),\ell+1}^{1,\foj}}$
leaves invariant modulo $\fom_{\ell+1}$ all terms of $\vartheta^{\foj}_{m_i}(P)$
except for the term $t^{F_i} z^{m_i}$. However, in order for
$\theta_{\delta,\foD_{(\widetilde X,\widetilde D),\ell+1}^{1,\foj}}(t^{F_i} z^{m_i})$
to produce a term $-\alpha t^{F_i} z^{m_i-v}$ necessary to cancel
this term in $\vartheta^{\foj}_{m_i}(P)$, $\delta$ would have to cross $\fod$,
which it does not. Thus we have obtained a contradiction to consistency.

We thus conclude that $m_i$ is tangent to $\fod$. This implies that
$\fod$ is contained in either $\rho_{\infty}$ or $\rho_0$.
Both of these are contained in $\bar\rho\cap \widetilde B_1$,
which is identified with $\rho$ under $\Psi|_{\widetilde B_1}$.
Thus by Remark \ref{rem:affine sub our case}, $\bar\rho\cap\widetilde B_1$
is radiant, and hence $\fod$ is radiant.

\end{proof}

\begin{proof}[Proof of Theorem \ref{Thm: radiant scatter}]
We will show the result by induction by showing that $\foD_{(\widetilde X,\widetilde D),k}^1$
is radiant for each $k$. The base case $k=1$ is trivial as
$\foD^1_{(\widetilde X,\widetilde D),1}$ is empty. So assume the result is true for
$\foD_{(\widetilde X,\widetilde D),k}^1$.

Suppose given a wall $(\fod,f_{\fod}(z^{-v}))\in\foD_{(\widetilde X,\widetilde D),k+1}^1$. 
Recall that by Proposition 
\ref{Prop: repulsive radiance suffices}, to show $\fod\in\foD_{(\widetilde X,\widetilde D),k+1}^1$ 
is radiant, it suffices to show that one of the repulsive facets of $\fod$ is 
radiant. 

If $f_{\fod}\not\equiv 1\mod \fom_k$, then $\fod$ is already contained
in a wall of $\foD_{(\widetilde X,\widetilde D),k}^1$, and hence is radiant. Thus we
may assume that $f_{\fod}\equiv 1\mod \fom_k$.

Let $\sigma_{\fod}$ be the smallest cell of $\scrP_1$ containing
$\fod$.  Choose a general point
$x\in\Int(\fod)$. This can be done sufficiently generally so that
$(x-\RR_{\ge 0}v)\cap \sigma_{\fod}$ (which makes sense inside 
$\sigma_{\fod}$)
does not intersect any $(n-3)$-dimensional face of any $\fod'\in
\foD_{(\widetilde X,\widetilde D),k+1}^1$.
Let $\foj\subseteq\fod$ be the unique repulsive facet of $\fod$  
intersecting $(x-\RR_{\ge 0}v)\cap\sigma$, and let $y$ be the
only point of $(x-\RR_{\ge 0}v)\cap\foj$. Let $\sigma_{\foj}\in\scrP$
denote the smallest cell containing $\foj$.

We may assume that $\foj$ is not contained
in $\widetilde\Delta_1$. Indeed, if $\foj\subseteq
\widetilde\Delta_1$, then $\fod$
is radiant by the induction hypothesis and Lemma \ref{Lem: walls lie on radiant directions}. Thus we may assume $y\not\in\widetilde\Delta_1$.


Set 
\[
\foT_y:=\{\Lambda_{\fod'}\,|\, y\in \fod'\in \foD_{(\widetilde X,\widetilde D),k+1}^1\},
\]
where $\Lambda_{\fod'}$ denotes the integral tangent space to
$\fod'$ as a sublattice of $\Lambda_y$. Note 
different walls may give rise to the same element of $\foT_y$.

We analyze two different possibilities:

\emph{Case I. $\#\foT_y\le 2$.}
If there is any wall $(\fod',f_{\fod'})\in\foD_{(\widetilde X,\widetilde D),k+1}^1$ 
containing $y$ with 
$\Lambda_{\fod}=\Lambda_{\fod'}$, and $f_{\fod'}\not\equiv 1
\mod\fom_{k}$, then by the induction hypothesis $\fod'$ is
radiant, and hence so is $\fod$. Thus we may assume that all
walls containing $y$ and with the same tangent space as $\fod$
are trivial modulo $\fom_k$.
In particular, the wall-crossing automorphisms associated to
crossing these walls commute with all other wall-crossing automorphisms
modulo $\fom_{k+1}$.

If $\#\foT_y=2$, 
let $\Lambda\in \foT_y$ with $\Lambda\not=\Lambda_{\fod}$, and choose
a $m\in \Lambda\setminus\Lambda_{\fod}$. If $\#\foT_y=1$, take any
$m\in \Lambda_y\setminus \Lambda_{\fod}$.
Then the wall-crossing automorphism $\theta$ associated to 
crossing a wall $\fod'$ with $\Lambda_{\fod'}=\Lambda$ satisfies
$\theta(z^m)=z^m$. Thus if $\gamma$ is a suitably oriented loop
around $\foj$ close to $y$, we find
\begin{equation}
\label{Eq:theta action1}
\theta_{\gamma,\foD_{(\widetilde{X},\widetilde{D})}^1}(z^m)=z^mf_1^{\langle m, n_{\fod}\rangle}
f_2^{-\langle m,n_{\fod}\rangle},
\end{equation}
where $n_{\fod}$ is a primitive normal vector to $\Lambda_{\fod}$, and
\begin{align*}
\foD_1:= {} & \{(\fod',f_{\fod'})\in\foD_{(\widetilde X,\widetilde D),k+1}^1\,|\,
y\in\fod', \Lambda_{\fod}=\Lambda_{\fod'}, \dim\fod\cap\fod'=n-1\},\\
\foD_2:= {} & \{(\fod',f_{\fod'})\in\foD_{(\widetilde X,\widetilde D),k+1}^1\,|\,
y\in\fod', \Lambda_{\fod}=\Lambda_{\fod'}, \dim\fod\cap\fod'=n-2\},\\
f_i:= {} & \prod_{\fod'\in \foD_i} f_{\fod'},\quad i=1,2.
\end{align*}
Note that by the definition of consistency in codimensions $0$ and
$1$ (Definitions \ref{Def: consistency in codim zero} and
\ref{Defn: consistency along codim one}), the identity
\eqref{Eq:theta action1} holds in the ring
$\kk[\shP^+_y]/\fom_{k,y}$.

\begin{figure}
\center{\input{Repulsive.pspdftex}}
\caption{The local structure near $\foj$ when $\#\foT_y=2$.}
\label{Fig: repulsive}
\end{figure}

We now distinguish further between three cases. If $\dim\sigma_{\foj}=n$,
then $\kk[\shP^+_y]/\fom_{k,y}=\kk[\Lambda_y][Q]/\fom_k$.
Then \eqref{Eq:theta action1} implies that $f_1=f_2\mod \fom_{k+1}$. 
In particular, there must
be a wall $(\fod',f_{\fod'}(z^{-v}))\in \foD_2$, with an $\epsilon>0$
such that $y-\epsilon v\in \fod'$. We may now replace $\fod$ with $\fod'$
and $x$ with $x':=y-\epsilon v$ and continue the process.

If $\dim\sigma_{\foj}=n-1$, then 
$\fom_{k+1,y}$ is the monomial ideal generated
by the inverse image of the ideals $\fom_{k+1}+\shP^+_x$ and
$\fom_{k+1}+\shP^+_{x'}$ under the maps
$\shP^+_y\rightarrow \shP^+_x, \shP^+_{x'}$ of \eqref{eq:parallel transport}.
\eqref{Eq:theta action1} implies that $f_1=f_2 \mod \fom_{k+1,y}$. 
It then follows
from Proposition \ref{prop:order increasing} that $f_2$ must contain
a term $\alpha z^{-v}$ of order $<k$, so that $f_2\not\equiv 1\mod
\fom_k$. Thus by the induction hypothesis, one of the walls in
$\foD_2$ is non-trivial to order $k$, and hence by the induction
hypothesis is radiant. Thus $\fod$ is radiant.

If $\dim\sigma_{\foj}=n-2$, then 
$\foj$ is contained in a codimension two
cell of $\scrP_1$ which is not contained in the discriminant locus.
However, all codimension two cells of $\scrP_1$ not contained in the
discriminant locus are radiant, by Remark \ref{rem:affine sub our case}.
Thus $\foj$ is radiant. So by Proposition 
\ref{Prop: repulsive radiance suffices}, $\fod$ is radiant.

\emph{Case II. $\#\foT_y \ge 3$.} We will show that either
(a) we have similar behaviour as in the previous case, or (b)
there exists walls $\fod_1,\fod_2$ containing $y$ with
$f_{\fod_i}\not\equiv 1\mod \fom_{k}$ and $\Lambda_{\fod_1}\not=
\Lambda_{\fod_2}$. Thus $\fod_1,\fod_2$ are radiant, so their
intersection is also radiant. So $\foj$ is radiant, and hence
so is $\fod$.

Case (a) occurs when there is at most one $\Lambda\in \foT_y$
with the property that there exists a wall $(\fod',f_{\fod'})\in
\foD_{(\widetilde X,\widetilde D),k+1}^1$ with $y\in\fod'$, $\Lambda_{\fod'}=\Lambda$,
and $f_{\fod'}\not\equiv 1\mod \fom_k$. If there is no such
$\Lambda$, then pick some $\Lambda\in\foT_y$ distinct from $\Lambda_{\fod}$
for what follows.

Similarly to before, choose some 
$m\in \Lambda\setminus\bigcup_{\Lambda'\in\foT_y\setminus\{\Lambda\}}\Lambda'$.
Let $\gamma$ again be a small loop around $\foj$ near $y$.
All wall-crossing automorphisms involved in $\theta_{\gamma,\foD^1_{\widetilde X,
k+1}}$ commute except for those arising from crossing the
walls with tangent space $\Lambda$. However, the latter wall-crossing
automorphisms leave $z^m$ invariant as $m\in\Lambda$. Thus
we can write, the first equality by consistency,
\[
z^m=\theta_{\gamma,\foD_{(\widetilde X,\widetilde D),k+1}^1}(z^m)
=z^m \prod_{i=1}^n f_i^{a_i},
\]
where this equality holds in $\kk[\shP^+_y]/\fom_{k+1,y}$ as in Case I.
Here $f_1,\ldots,f_n$ are wall functions associated to walls
crossed by $\gamma$ whose tangent space is not $\Lambda$,
and all $f_i\equiv 1\mod \fom_{k}$. Further, the $a_i$'s are non-zero
integers. 

Now some $f_i=f_{\fod}(z^{-v})$. Thus in order to get the above
equation, there must be another $f_j$ which is a function of $z^{-v}$.
Necessarily this wall $(\fod_j,f_j)$ lies on the other side of
$\foj$ as $\fod$. Thus as in Case I, we have three possibilities.
If $\dim\sigma_{\foj}=n$, 
we can replace $\fod$ with $\fod_j$ and $x$ by $y-\epsilon v$.
If $\dim\sigma_{\foj}=n-1$, then $f_j\not\equiv 1\mod \fom_k$, and hence
$\fod_j$, and so $\fod$, is radiant by the induction hypothesis.
Finally, if $\dim\sigma_{\foj}=n-2$, we argue as in Case I.

This gives case (a). Otherwise (b) holds, and as observed, $\fod$
is radiant.

\medskip

Now observe that the process described above, replacing a wall $\fod$
with another wall and replacing $x$ with $y-\epsilon v$, eventually
stops. Indeed, as $\foD_{(\widetilde X,\widetilde D),k+1}^1$ is finite, and this process
stays within a cell of $\scrP_1$, we must eventually end up 
in one of the cases where we can conclude $\fod$ is radiant.
\end{proof}

\subsection{Consequences of radiance}

We now use radiance and consistency to analyze more carefully the structure,
up to equivalence, of 
$\foD^1_{(\widetilde X,\widetilde D)}$. We work
with the specific equivalent scattering diagrams
$\foD^1_{(\widetilde X,\widetilde D),k}$
of Construction \ref{Constr: GSk}, working with
as large a $k$ as is necessary for any given statement. 

We first will give two lemmas concerning the structure of the diagram
near the discriminant locus. In some sense, these two lemmas
are the main idea of the paper. Lemma \ref{Lem:functional equation}
gives a relationship between the walls on either side of the
discriminant locus which is implied entirely by the monodromy
of the sheaf $\shP$. Using this relationship, we may then give
in Lemma \ref{lem: multiple covers}
a ``pure thought'' calculation of contributions of
multiple covers of the exceptional curves of the blow-up
$\widetilde X\rightarrow X_{\widetilde\Sigma}$ to the canonical
scattering diagram. This replaces a localization argument which
was used in \cite{GPS}. See also \cite{parker2014tropical} for a derivation of the multiple cover formula in the two dimensional case, via tropical geometric arguments.

For this analysis,
choose a general point $x\in \widetilde\Delta_{\ul{\rho}}$, for $\ul{\rho}$
a codimension one cone in $\Sigma(\rho_i)$. In the notation of
\S\ref{rem:phi disc}, let $y,y'$ be points
sufficiently close to $x$ with $y\in \rho_{\infty}$ and $y'\in\rho_0$.
Since $\foD^1_{(\widetilde X,\widetilde D),k}$ is finite, this choice may be made so that
any wall containing 
$x$ contains either $y$ and $y'$. Here we use the fact that by radiance
of $\foD^1_{(\widetilde X,\widetilde D),k}$, (Theorem \ref{Thm: radiant scatter}),
any wall of this scattering diagram containing
$x$ is necessarily contained in either $\rho_{\infty}$ or $\rho_0$.

Recall for $z\in \widetilde B_1\setminus 
\Sing(\foD_{(\widetilde X,\widetilde D),k}^1)$ the notation
$f_z$ of \eqref{eq: fx def}. In particular, this gives elements
$f_y\in \kk[\Lambda_{\tilde\rho}][Q]/\fom_k$ and
$f_{y'}\in \kk[\Lambda_{\tilde\rho'}][Q]/\fom_k$.
Recall these may be compared via the map $\wp$ of \eqref{Eq: wp def}.
We then have:

\begin{lemma}
\label{Lem:functional equation}
In the notation of \S\ref{rem:phi disc}, we have
\[
f_{y'} = \wp(f_y)\prod_{j=1}^{\kappa^i_{\ul{\rho}}}
\left(t^{F_i-E_{\ul{\rho}}^j}z^{m_i}\right).
\]
\end{lemma}

\begin{proof}
Let $\foj$ be the joint of $\foD^1_{(\widetilde X,\widetilde D),k}$ containing
the point $x$. Necessarily $\foj\subseteq \widetilde\Delta_{\ul{\rho}}$
is a codimension two joint. Thus we make use of the consistency of
\S\ref{Subsubsec: Consistency around codim two joints} for
the scattering diagram $\foD^{1,\foj}_{(\widetilde X,\widetilde D),k}$, similarly as
in Lemma \ref{Lem: walls lie on radiant directions}.

We need to be quite careful about the charts we use to describe
$B_{\foj}$. The charts $\psi_U,\psi_V$ of Theorem 
\ref{thm: affine structure on the big part} induce charts on
$B_{\foj}$, with $U_{\foj}:=B_{\foj}\setminus\rho_{\infty}$
and $V_{\foj}=\mathrm{Star}(\rho_{\infty})$. Then we have
charts $\psi_{U_{\foj}}:U_{\foj}\rightarrow M_{\RR}$, $\psi_{V_{\foj}}:
V_{\foj}\rightarrow M_{\RR}$, with 
\[
\psi_{V_{\foj}}\circ \psi_{U_{\foj}}^{-1}(m)=m+\varphi_i(\pi_i(m)) m_i,
\]
on $U_{\foj}\cap V_{\foj}$, see \eqref{Eq:psiv psiu}.

In what follows, we use the chart $\psi_{U_{\foj}}$ to identify
tangent or cotangent vectors on $B_{\foj}$ with elements of $M$
or $N$, taking care when we need to compare such vectors via parallel
transport across $\rho_{\infty}$. As such, we have normal vectors
$n_{\rho_{\pm}}\in N$ to $\rho^+$ and $\rho^-$, as described in
\S\ref{subsubsec:geometry of strata}, always taken to be positive on
$\rho_{\infty}$. In addition, take
$n_{\rho_{\infty}}=n_{\rho_0}$ to be primitive normal vectors
to $\rho_{\infty}$ and $\rho_0$, positive on $\rho^+$. By \eqref{Eq: afffine monodromy}, this covector is monodromy invariant as $m_i$ is tangent to
$\rho_{\infty}$ and $\rho_0$.

Pick a tangent vector $m_+$ to $\rho^+$ pointing
away from $\widetilde\Delta_{\ul{\rho}}$ with the property that
$\langle n_{\rho_{\infty}}, m_+\rangle =1$; this can be chosen
to be the vector $u_n^+$ in the notation of 
\S\ref{subsubsec:geometry of strata}.

Let $p_{\pm}, p'_{\pm}$ be general points in the interior of
the cells corresponding to $\tilde\sigma^{\pm}$ and $\tilde\sigma'^{\pm}$.
We will consider the theta functions
$\vartheta^{\foj}_{m_+}(p_{\pm})$ and $\vartheta^{\foj}_{m_+}(p'_{\pm})$.
In particular, $\vartheta^{\foj}_{m_+}(p_+)=\vartheta^{\foj}_{m_+}(p_+')
=z^{m_+}$, as any broken line ending at $p_+$ or $p_+'$ with asymptotic
monomial $m_+$ never crosses a wall or a codimension one cell
of $\P_{\foj}$, and in particular does not bend.

As $\foD^{1}_{(\widetilde X,\widetilde D)}$ is consistent,
it then follows from condition (2) of Definition 
\ref{Def: Consistency in codim two} that
\begin{align*}
\vartheta^{\foj}_{m_+}(p_-) = {} & t^{s_{\rho_{\infty}}}z^{m_++\kappa^i_{\ul{\rho}}m_i}f_y,\\
\vartheta^{\foj}_{m_+}(p'_-) = {} & t^{s_{\rho_0}}z^{m_+}f_{y'}.
\end{align*}
Here the factors $t^{s_{\rho_0}}$, $t^{s_{\rho_{\infty}}}$ arise 
from the kinks of $\varphi$ along $\tilde\rho$, $\tilde\rho'$, see
\eqref{eq:local kinks}, and
the fact that $\langle n_{\rho_{\infty}},m_+\rangle = 1$.
Recalling that we are identifying tangent vectors with elements of
$M$ via the chart $\psi_{U_{\foj}}$, we see that the factor 
$z^{\kappa^i_{\ul{\rho}}m_i}$ arises from
Lemma \ref{Lem: affine monodromy} as we must use parallel transport
in the chart $V_{\foj}$.

Because there is no wall of $\foD_{(\widetilde X,\widetilde D)}^{1,\foj}$ between $p_-$ and $p_-'$, 
$\vartheta^{\foj}_{m_+}(p_-)$ and $\vartheta^{\foj}_{m_+}(p'_-)$
must agree under parallel transport across $\rho^-$, again
by Definition \ref{Def: Consistency in codim two}, (2).
Thus we have
\[
t^{s_{\rho_{\infty}}}t^{(\langle n_{\rho^-}, m_+\rangle + \kappa^i_{\ul \rho})F_i} 
z^{m_++\kappa^i_{\ul\rho} m_i}\wp(f_y)
=t^{s_{\rho_0}}z^{m_+} f_{y'}.
\]
Here, using \eqref{Eq: nrho one}, the exponent $(\langle n_{\rho^-}, 
m_+\rangle + \kappa^i_{\ul{\rho}}) F_i$ accounts for the change in the 
monomial $z^{m_++\kappa^i_{\ul{\rho}} m_i}$
caused by the kink of $\varphi$ along $\rho^-$. 

We now note that if we work with sufficiently high $k$, we may
cancel the $z^{m_+}$ on both sides. Indeed, the monomial $z^{m_+}$
lies in $\fom_1$ and is a zero-divisor, but if we are are interested
in proving the equality of the lemma modulo $\fom_k$ for any 
given $k$, we may increase $k$ and compare monomials, thereby getting,
after using \eqref{Eq: lin equiv 2},
\[
f_{y'}=t^{\kappa^i_{\ul{\rho}}F_i-\sum_j E^j_{\ul{\rho}}}z^{\kappa^i_{\ul{\rho}} m_i}\wp(f_y),
\]
which gives the desired formula.
\end{proof}

\begin{lemma}
\label{lem: multiple covers}
We have
\[
f_{y}= g \prod_{j=1}^{\kappa_{\ul{\rho}}^i}(1+t^{E^j_{\ul{\rho}}}z^{-m_i}),\quad
f_{y'}=\wp(g) \prod_{j=1}^{\kappa_{\ul{\rho}}^i}
(1+t^{F_i-E_{\ul{\rho}}^j}z^{m_i})
\]
for some $g\in \kk[\Lambda_{\tilde\rho}][Q]/\fom_k$.
\end{lemma}

\begin{proof}
{\bf Step I. Reduction to calculation of local terms.}
Suppose we show that 
\begin{equation}
\label{Eq: fy fyprime}
f_y= g\prod_j (1+t^{E^j_{\ul{\rho}}}z^{-m_i}),\quad
f_{y'}= g' \prod_j (1+t^{F_i-E^j_{\ul\rho}}z^{m_i}).
\end{equation}
Then applying Lemma \ref{Lem:functional equation}, we see that
\begin{align*}
f_{y'}=\wp(f_y)\prod_j (t^{F_i-E^j_{\ul{\rho}}}z^{m_i})
= {} & \wp(g)\prod_j\left[ t^{F_i-E^j_{\ul{\rho}}}z^{m_i}
(1+t^{E^j_{\ul{\rho}}}t^{-F_i}z^{-m_i})\right]\\
= {} & \wp(g)\prod_j (z^{F_i-E^j_{\ul{\rho}}}z^{m_i}+1).
\end{align*}
Here in the first line we use the fact that parallel transport
of $z^{-m_i}$ across $\rho^+$ or $\rho^-$ is given by 
\eqref{eq:parallel transport P}. Thus we see that
$g'=\wp(g)$. 

To show \eqref{Eq: fy fyprime}, we will calculate contributions
to the canonical scattering diagram arising from
curves entirely contained in $\widetilde\PP_i$, see
\eqref{Eq: central fiber of epsilon}. As we shall
see, these contributions arise only 
from multiple covers of the curve
classes $E^j_{\ul{\rho}}$ and $F_i-E^j_{\ul{\rho}}$.
\smallskip

{\bf Step II. Reduction to a simpler target space.}
Since the contributions from curves mapping into $\widetilde\PP_i$
only depend on an open neighbourhood of
$\widetilde\PP_i$, it is sufficient to carry out the calculation
for the target space where
we only blow up $\widetilde H_i$, and not $\widetilde H_j$ for
$i\not=j$. Specifically, let $\mathbf P_i=(\rho_i)$, the sub-$1$-tuple
of the $s$-tuple $\mathbf P$. Thus we obtain also a family
$\widetilde X_i\rightarrow \AA^1$, and
a corresponding scattering diagram $\foD_{(\widetilde X_i,\widetilde D_i)}^1$ on 
$\widetilde B_{1,i}$. We note that in a neighbourhood
of $\widetilde\PP_i$, $\widetilde X_i$ and $\widetilde X$
are isomorphic. Thus the desired contributions may be calculated in 
either space, and we now consider the canonical scattering diagram
for $\widetilde X_i$.
\smallskip

{\bf Step III. The tropical analysis.}
Let us now consider a type $\tau=(G,\boldsymbol{\sigma},
\mathbf{u})$ and curve class $\ul{\beta}$ contributing a non-trivial wall 
$(\fod,f_{\fod})\in\foD^1_{(\widetilde X_i,\widetilde D_i)}$.
Recall from Theorem \ref{thm:relative diagram} that any curve in
this moduli space is defined over $\AA^1$. In particular, we obtain
an induced map $Q_{\tau,\RR}^{\vee}\rightarrow \Sigma(\AA^1)=\RR_{\ge 0}$,
and a general point $s\in\tau$ mapping to $1\in \RR_{\ge 0}$
then yields a tropical map $h_s:G\rightarrow 
\widetilde B_{1,i}$.

By Theorem \ref{thm:balancing general}, the map $h_s$ satisfies the usual
tropical balancing condition at each vertex $v\in V(G)$
with $h_s(v)$ not lying in the discriminant locus $\widetilde\Delta_1^i$ of 
$\widetilde B_{1,i}$. On the
other hand, by Proposition \ref{prop:delta balancing}, if $h_s(v)$ lies in the 
discriminant locus, then the
balancing condition still holds modulo $m_i$.

The graph $G$, being genus $0$, is a tree. Further, it has precisely one leg,
$L_{\out}$,
corresponding to the unique punctured point. The graph then has a number
of univalent vertices, i.e., leaves of the tree.
There are two possibilities for a univalent vertex $v$
with adjacent edge or leg $E$.

First, if $h_s(v)\not\in \widetilde\Delta_1^i$, then balancing implies that
$\mathbf{u}(E)=0$. In particular, $E\not=L_{\out}$, as
$\mathbf{u}(L_{\out})\not=0$ for types contributing to the scattering
diagram.
If $v$ and $E$ are removed from the graph $G$
to obtain a new type $\tau'$, we note that $\dim Q^{\vee}_{\tau',\RR}
=\dim Q^{\vee}_{\tau,\RR}-1$, as the length of the edge $E$ is a
free parameter. On the other hand, if 
$h:\Gamma/Q^{\vee}_{\tau,\RR}\rightarrow\Sigma(\widetilde X)$
and $h':\Gamma'/Q^{\vee}_{\tau',\RR}\rightarrow\Sigma(\widetilde X)$
are the universal families of tropical maps of type $\tau$ and $\tau'$
respectively, it is clear that $h(\tau_{\out})=h'(\tau'_{\out})$.
Since the type $\tau$ only contributes to $\foD_{(\widetilde X,\widetilde D)}$
if $\dim Q^{\vee}_{\tau,\RR}=n-1$ and $\dim h(\tau_{\out})=n$, 
we obtain a contradiction. Thus no such univalent vertex occurs.

Second, if $h_s(v)\in \widetilde\Delta_1^i$, we then have
$\mathbf{u}(E)$ proportional (positively or negatively) to $m_i$.
The same argument as above shows $\mathbf{u}(E)$ can't be $0$.

Combining these two observations with balancing at vertices mapping to
non-singular points of $\widetilde B_1$, we see that the only possible 
tropical maps of the type being
considered must have image in $\widetilde B_1$ a ray or line segment
emanating from $\widetilde\Delta_i^1$ in the direction $m_i$ or $-m_i$.

Note that curves in a moduli space $\scrM_{(\tau,\ul{\beta})}(\widetilde X_i)$
have image entirely contained in $\widetilde\PP_i$ if and only if
$h_s(G^{\circ})$ is entirely contained in the open star of
$\nu_i$ defined in \eqref{Eq:nui def}, where $G^{\circ}$ is as 
given in Construction \ref{Const: dual graph}.
If this is the case, the restrictions on $\dim Q_{\tau,\RR}^{\vee}$
and $\dim h(\tau_{\out})$ imply that $G$ has one vertex $v$,
one leg $L_{\out}$, and no edges. Further,
$\mathbf{u}(L_{\out})$ is proportional to $m_i$.

If the constant of proportionality is positive, we say we are in
the \emph{positive case}, and then the corresponding wall
has support $h(\tau_{\out})\cap\widetilde B_1
=\rho_{\infty}$ for some codimension one $\rho\in\Sigma$ containing
$\rho_i$. If the constant of proportionality is negative,
we say we are in the \emph{negative case}, and then the corresponding
wall has support
$h(\tau_{\out})\cap\widetilde B_1=\rho_0$ again for some such $\rho$.
\smallskip

{\bf Step IV. Analysis of curve classes.}
Now let $f:C\rightarrow \widetilde{X}_i$ be a punctured
log map in the moduli space $\scrM_{(\tau,\ul{\beta})}(\widetilde{X}_i)$.
Then the image of $f$
must lie inside the surface $\widetilde D_{\widetilde\Delta_{\ul{\rho}}}$,
with $\ul{\rho}=(\rho+\RR\rho_i)/\RR\rho_i$ and $\rho$ as in the
previous step. We wish to determine
the curve class $\ul{\beta}$ expressed as a curve class in 
this surface. 

To do so, note that in the positive and negative cases described above, 
the vector $u=\mathbf{u}(L_{\out})$
is a tangent vector to the cones $\tilde\rho$ and $\tilde\rho'$
respectively. One may then write $u$ as a linear combination of the
generators of this cone. We may in particular use the description of these
generators in \S\ref{subsubsec:geometry of strata}. In the positive case,
we may write $u=k (u_1,0)$ for some $k>0$. In the negative case,
we may write $u=k(0,1)-k(m_i,1)$ for some $k>0$. Note that in the
language of \eqref{Eq: polyhedral decomposition}, we may
equate the generator $(u_1,0)$ of $\tilde\rho$ with $\widetilde D_{\rho_i}^*$,
and the generators $(0,1)$ and $(m_i,1)$ of $\tilde\rho'$ with
$\widetilde D_{\RR_{\ge 0}(0,1)}^*$ and $\widetilde D_{\nu_i}^*$ respectively.

Now \cite[Cor.\ 1.14]{GSAssoc} implies that
if $\widetilde D_j$ is any irreducible component of the boundary of 
$\widetilde
X_i$, we have $\widetilde D_j\cdot \ul{\beta}=\langle \widetilde D_j,u\rangle$ 
where $u$ is viewed as an element of 
$\Div_{\widetilde D}(\widetilde X_i)^*_{\RR}$ as described above.

Explicitly in our situation, 
in the positive case, the only boundary divisor which intersects
$\ul{\beta}$ non-trivially is $\widetilde D_{\rho_i}$. In the negative
case, $\widetilde D_{\RR_{\ge 0}(0,1)}$ intersects $\ul{\beta}$
positively and $\widetilde D_{\nu_i}$ intersects $\ul{\beta}$ negatively,
and all other intersections with boundary divisors are $0$.

We now view $\ul{\beta}$ as a curve class in the
surface $\widetilde D_{\Delta_{\ul{\rho}}}$: this makes sense, as the
image of $f$ is contained in $\widetilde D_{\Delta_{\ul{\rho}}}$. Note that
the curve classes $s_{\rho_{\infty}}$ and $s_{\rho_0}$ on 
$\widetilde D_{\widetilde\Delta_{\ul{\rho}}}$
are the classes of the intersection of 
$\widetilde D_{\widetilde\Delta_{\ul{\rho}}}$ with
$\widetilde D_{\rho_i}$ and $\widetilde D_{\RR_{\ge 0}
(1,0)}$ respectively. The curve class $F_i$ may be viewed
as the intersection of $\widetilde D_{\widetilde\Delta_{\ul{\rho}}}$
with another divisor which meets $\ul{\beta}$ trivially.
Thus, on the surface $\widetilde D_{\widetilde\Delta_{\ul{\rho}}}$,
we see that we have $\ul{\beta}\cdot F_i =0$ in both cases. In the positive 
case, $\ul{\beta}\cdot s_{\rho_{\infty}}>0$ and $\ul{\beta}\cdot s_{\rho_0}=0$,
while in the negative case $\ul{\beta}\cdot s_{\rho_{\infty}}=0$
and $\ul{\beta}\cdot s_{\rho_0}>0$. 

From this, we may conclude that in the positive case $\ul{\beta}=
\sum_j w_j E^j_{\ul{\rho}}$, while in the negative case $\ul{\beta}=\sum_j
w_j (F_i-E^j_{\ul{\rho}})$ for some collection of integers $w_j$.
As $\ul{\beta}$ is effective, all $w_j$ must be non-negative.
Note the $E^j_{\ul{\rho}}$ are a disjoint set of $-1$ curves.
Thus, in the positive case, if $w_j>0$, then $\ul{\beta}\cdot E^j_{\ul{\rho}}
<0$, and hence the image of $f$ must contain the exceptional
curve $E^j_{\ul{\rho}}$. It then follows that the image of $f$
is precisely the union of the $E^j_{\ul{\rho}}$ with $w_j\not=0$.
Since the image is connected, at most
one $w_j$ may be non-zero, so $\ul{\beta} = w E^j_{\ul{\rho}}$ for
some $j$. 

Similarly, the $F_i-E^j_{\ul{\rho}}$ form a disjoint set of
$-1$ curves, so the same argument implies that in the negative case, 
$\ul{\beta} = w (F_i-E^j_{\ul{\rho}})$
for some $j$. 

In particular, in the two cases, $f$ must be a degree $w$ cover
of $E^j_{\ul{\rho}}$ or $F_i-E^j_{\ul{\rho}}$.
\smallskip

{\bf Step V. Application of Lemma \ref{Lem:functional equation}
to finish the calculation.}
Now let $f_y$, $f_{y'}$ be the wall-crossing functions at $y,y'$
for the canonical scattering diagram for $\widetilde{X}_i$.
By the above analysis, we may write 
\[
f_y=\prod_{j=1}^{\kappa^j_{\ul{\rho}}}f_{y,j}(t^{E^j_{\ul{\rho}}}z^{-m_i}),
\quad
f_{y'}=\prod_{j=1}^{\kappa^j_{\ul{\rho}}}f_{y',j}(t^{f-E^j_{\ul{\rho}}}z^{m_i}).
\]
Here $f_{y,j}$, $f_{y',j}$ should be viewed as polynomials in one
variable. Now certainly $f_{y,j}$ is independent of the choice of
$j$ by symmetry, and the same is true of $f_{y',j}$. Further, since
$f_y$ is a polynomial in $z^{-m_i}$ and $f_{y'}$ is a polynomial
in $z^{m_i}$, Lemma \ref{Lem:functional equation} implies that
$f_y$ and $f_{y'}$ are polynomials of degree $\kappa^i_{\ul{\rho}}$
in $z^{-m_i}$ and $z^{m_i}$ respectively. 
This now implies that each $f_{y,j}$,
$f_{y',j}$ is a linear polynomial. Again, Lemma \ref{Lem:functional equation}
and symmetry tells us that
\[
f_{y',j}(t^{F_i-E^j_{\ul{\rho}}}z^{m_i}) = 
\wp\left(f_{y,j}(t^{E^j_{\ul{\rho}}}z^{-m_i})\right)
t^{F_i-E^j_{\ul{\rho}}}z^{m_i}.
\]
We may use in addition the fact that by construction
of the canonical scattering diagram, the constant term of each 
$f_{y,j}$ must be $1$, and the same for $f_{y',j}$. 
This immediately gives only one choice for $f_{y,j}$ and $f_{y',j}$,
namely
\[
f_{y,j}(t^{E^j_{\ul{\rho}}}z^{-m_i})=1+t^{E^j_{\ul{\rho}}}z^{-m_i},
\quad
f_{y',j}(t^{F_i-E^j_{\ul{\rho}}}z^{m_i})=1+t^{F_i-E^j_{\ul{\rho}}}z^{m_i}.
\]
This allows us to conclude the lemma.
\end{proof}

In the remainder of this section, we observe that
$\foD_{(\widetilde X,\widetilde D)}^1$ is now determined completely
by its behaviour near the origin.

To set up the notation, let $S=(M_{\RR}\setminus\{0\})/\RR_{> 0}$ be the
sphere parameterizing rays from the origin. Given $s\in S$, we
write $\rho_s$ for the corresponding ray. For a fixed $k$,
there is a subset $S_0\subset S$ whose complement is codimension $\ge 2$,
with the property that for $s\in S_0$, $\dim \rho_s\cap 
\Sing(\foD_{(\widetilde X,\widetilde D),k}^1)=0$.

\begin{theorem}
\label{thm:radial directions}
Let $s\in S_0$. There are three possible behaviours:
\begin{enumerate}
\item There exists $\sigma\in\P_1$ with $\rho_s\subseteq\sigma$.
In this case, $f_x$ is independent of $x\in \rho_s$.
\item There exists a maximal $\ul{\sigma}\in\Sigma(\rho_i)$ with 
$\rho_s\subseteq (\tilde\sigma'\cup\tilde\sigma)\cap \widetilde B_1$, 
but $\rho_s\cap \widetilde\Delta_1=\emptyset$. In this case,
$f_x$ for $x\in\rho_s$ only depends on whether 
$x\in \tilde\sigma'\cap \widetilde B_1$
or $x\in \tilde\sigma'\cap\widetilde B_1$. Taking 
\[
y\in \Int(\rho_s\cap\tilde\sigma),\quad y'\in \Int(\rho_s\cap\tilde\sigma'),
\]
$f_y$ and $f_{y'}$ are related by parallel transport in $\shP^+$ along
$\rho_s$.
\item There exists a codimension one cell $\rho\in\Sigma$
with $\rho_s\subseteq (\tilde\rho'\cup\tilde\rho)\cap 
\widetilde B_1$. In this case, $\rho_s$ intersects 
$\widetilde\Delta_{\ul{\rho}}$ at a point.
Again, $f_x$ for $x\in\rho_s$ only depends on whether $x\in \tilde\rho'\cap
\widetilde B_1$ or $x\in \tilde\rho\cap \widetilde B_1$.
Taking
\[
y\in \Int(\rho_s\cap\tilde\rho),\quad y'\in \Int(\rho_s\cap\tilde\rho'),
\]
then $f_y$, $f_{y'}$ can be written as in Lemma \ref{lem: multiple covers}
\end{enumerate}
\end{theorem}

\begin{proof}
The three possibilities for the interaction of $\rho_s$ with
$\P_1$ are immediate from the description of the polyhedral
cone complex $\widetilde\P$ in \S\ref{subsec:tropical}. Thus we
only need to check the behaviour of $f_x$ along $\rho_s$. 

We consider case (1) first. Suppose that $\rho_s$ intersects
a joint $\foj$ of $\foD_{(\widetilde X,\widetilde D),k}^1$. 
By the definition of $S_0$, $\rho_s$ only intersects $\foj$
at one point, say $x$, so $\foj$ cannot be radiant. Thus by radiance
of $\foD_{(\widetilde X,\widetilde D),k}^1$, any wall containing
$\foj$ must intersect $\rho_s$ in a one-dimensional set,
as $\rho_s$ is radiant. 
We may then argue similarly as in the proof of Theorem
\ref{Thm: radiant scatter} that consistency of the scattering
diagram implies $f_y=f_{y'}$ for $y, y'\in\rho_s$ near $x$,
with $y,y'$ lying on the two sides of $\foj$. This shows the claimed
independence.

The argument is the same in cases (2) and (3), except that in
passing from $\widetilde\sigma'$ to $\widetilde\sigma$ in case (2)
we have to bear in mind parallel transport in $\shP^+$, and in
passing from $\widetilde\rho'$ to $\widetilde\rho$ in case (3),
we apply Lemma \ref{lem: multiple covers}.
\end{proof}

\section{Pulling singularities to infinity: scattering in $M_{\RR}$}
\label{Sec: The tropical vertex}
Let $T_0\widetilde B_1$ denote the tangent space to the origin $0\in \widetilde B_1$, thought of
as a (radiant) affine manifold. Note that $T_0\widetilde B_1$ has
a canonical identification with $M_{\RR}$. Define a scattering diagram in $T_0\widetilde B_1$ obtained by localizing $\foD^1_{(\widetilde X,\widetilde D)}$ to the origin by
\begin{equation}
    \label{Eq: asymptotic canonical }
T_0\foD^{1}_{(\widetilde X,\widetilde D)}:=\{(T_0\fod, f_{\fod})\,|\,
(\fod,f_{\fod})\in \foD^{1}_{(\widetilde X,\widetilde D)}, \quad 0\in \fod\}.    
\end{equation}
Here $T_0\fod$ denotes the \emph{tangent wedge} to $\fod$ at $0$, in this case, just the cone generated by $\fod$. Note that $T_0\foD^{1}_{(\widetilde X,\widetilde D)}$ is a scattering diagram
on $M_{\RR}$. In this section we will describe another viewpoint on
scattering diagrams in $M_{\RR}$, and then construct another diagram which we associate to the toric variety $X_{\Sigma}$ and the hypersurfaces $H_i$.
This diagram is constructed purely algorithmically, using the method of
\cite{GS2}, generalizing a two-dimensional construction of
\cite{KS}.
In the final section we will compare these two scattering diagrams in $M_{\RR}$. 
\subsection{The higher dimensional tropical vertex}
\subsubsection{The general scattering setup}
\label{Sect: the general setup}
Fix a lattice $M$ of finite rank, and as usual let $N=\Hom_{\ZZ}(M,\ZZ)$,
$M_{\RR}=M\otimes_{\ZZ}\RR$. In what follows, assume we are given
a toric monoid $P$ along with a map $r:P\rightarrow M$; e.g., we might
take $P=M\oplus\NN^p$ for some positive integer $p$ with $r$ the projection. Let $P^{\times}$
be the group of units of $P$, and let $\fom_P=P\setminus P^{\times}$.
This induces an ideal $\fom_P$ in the monoid ring $\kk[P]$.
We write for any monomial ideal $I\subseteq P$ the ring
\begin{equation}
\label{Eq: RI}
    R_I:=\kk[P]/I,
\end{equation}
and we denote by $\widehat{\kk[P]}$ the completion of $\kk[P]$ with respect to
$\fom_P$.

Fix a monomial ideal $I$ with $\sqrt{I}=\fom_P$.
We define the module of log derivations as usual as
\[
\Theta(R_I):= R_I \otimes_{\ZZ} N.
\]
Here we write $z^m \partial_n:= z^m\otimes n$ for $m\in P$, $n\in N$,
and $z^m\partial_n$ acts on $R_I$ via 
\[
z^m\partial_n(z^{m'})=\langle n,r(m')\rangle z^{m+m'}.
\]
Thus in particular if $\xi\in \fom_P\Theta(R_I)$, then
\[
\exp(\xi)\in \Aut(R_I).
\]
We note the commutator relation
\[
[z^m\partial_n,z^{m'}\partial_{n'}]=z^{m+m'}
\partial_{\langle r(m'),n\rangle n'-\langle r(m),n'\rangle n}
\]
and obtain a nilpotent Lie subalgebra of $\fom_P\Theta(R_I)$
defined by
\[
\fov_I:=\bigoplus_{m\in P\setminus I\atop r(m)\not=0} z^m(\kk\otimes r(m)^{\perp})
\]
which is closed under Lie bracket and hence defines a group
\[
\VV_I:=\exp(\fov_I).
\]
This is the group whose set of elements is $\fov_I$ and with multiplication
given by the Baker-Campbell-Hausdorff formula.

We may then define the pro-nilpotent group
\[
\widehat\VV:= \lim_{\longleftarrow} \VV_{\fom_P^k}.
\]
This is the higher dimensional tropical vertex group. Note that
it acts by automorphisms on $\widehat{\kk[P]}$.

Let $\foj\subseteq M_{\RR}$ be a codimension two affine subspace
with rational slope and let $\Lambda_{\foj}\subseteq M$ be
the set of integral tangent vectors to $\foj$. This is a saturated sublattice
of $M$. Then we have the following Lie subalgebras 
of $\fov_I$:
\begin{align*}
\fov_{I,\foj} := {} & \bigoplus_{m\in P\setminus I\atop r(m)\not=0} z^m
(\kk\otimes (r(m)^{\perp}\cap \Lambda_\foj^{\perp}))\\
{}^{\perp}\fov_{I,\foj} := {} & \bigoplus_{m\in P\setminus I\atop r(m)\not\in\Lambda_\foj} z^m
(\kk\otimes (r(m)^{\perp}\cap \Lambda_\foj^{\perp}))\\
{}^{\|}\fov_{I,\foj} := {} & \bigoplus_{m\in P\setminus I\atop r(m)\in 
\Lambda_\foj\setminus
\{0\}} z^m
(\kk\otimes \Lambda_\foj^{\perp})
\end{align*}
One notes easily 
that $[\fov_{I,\foj},{}^{\perp}\fov_{I,\foj}]\subseteq {}^{\perp}\fov_{I,\foj}$,
and that ${}^{\|}\fov_{I,\foj}$ is abelian. In particular, if we
denote by $\VV_{I,\foj}$, ${}^{\perp}\VV_{I,\foj}$ and ${}^{\|}\VV_{I,\foj}$
the corresponding groups, we see that
\[
{}^{\|}\VV_{I,\foj}\cong \VV_{I,\foj}/{}^{\perp}\VV_{I,\foj}.
\]
Similarly, taking inverse limits, we have subgroups 
$\widehat\VV_{\foj}$,
${}^{\|}\widehat\VV_{\foj}$ and ${}^{\perp}\widehat\VV_{\foj}$.

We next consider scattering diagrams in $M_{\RR}$. Recall in \S\ref{Sec: wall structures} we defined the notion of a scattering diagram associated to an affine manifold together with a polyhedral decomposition $(B,\P)$, a choice of a monoid $Q$ and an MVPL-function $\varphi$. Here as a particular case, we consider the situation where $B=M_{\RR}$ and $\P$ is the trivial polyhedral decomposition whose only cell is $M_{\RR}$. For the monoid $Q$ we take $\NN^p$ for some positive integer $p$, and we consider the trivial MVPL-function $\varphi=0$. Under these choices the sheaf $\shP^+$ on $B$ is the constant sheaf with stalk $P= M \oplus \NN^p$. In this case we work with the algebra $R_I$ of \eqref{Eq: RI}. For a wall $\fod$, the attached function $f_{\fod}$ is an element of $R_I$ and is a sum 
$\sum c_m z^m$ with $r(m)$ negatively proportional to a primitive vector $m_0\in M\setminus\{0\}$ tangent to $\fod$, called the \emph{direction} of the wall. This gives us a scattering diagram $\foD$ over $R_I$ as in \S \ref{Sec: wall structures}, as a finite set of walls. By taking the inverse limit over all possible $I$ we obtain the notion of a scattering diagram over $\widehat{\kk[P]}$, which can have infinitely many walls. We may also generalize away from
using the sheaf $\shP^+$, and replace $P=M\oplus\NN^p$ with a more
general choice of monoid equipped with the map $r:P\rightarrow M$.
All definitions still apply without any difficulty.


Let $\foD$ be a scattering diagram over $\widehat{\kk[P]}$. For each $k>0$, let $\foD_k\subset\foD$ be the subset of
walls $(\fod,f_{\fod})\in\foD$ with $f_{\fod}\not\equiv 1\mod \fom_P^k$. A wall defines an element of $\VV_{\fom^k_P}$ in a standard way
(see e.g., \cite[Rem.\ 2.16]{GS2}), which in turn induces wall-crossing
automorphisms of $\kk[P]/\fom_P^k$ which agree with \eqref{Eqn: theta_fop}.
This allows us to define the wall-crossing automorphisms 
\[
\theta_{\gamma,\foD_k}:\kk[P]/\fom_P^k\rightarrow\kk[P]/\fom_P^k
\]
as in \eqref{Eq: path ordered product}. However, we may also view
$\theta_{\gamma,\foD_k}\in\VV_I$.  We then obtain $\theta_{\gamma,\foD}$, which is an element of the tropical vertex group  $\theta_{\gamma,\foD}\in\widehat{\VV}$, by taking the limit as $k\rightarrow
\infty$.

As we consider the trivial polyhedral decomposition $\P$ on $M_{\RR}$, all joints are of codimension zero in the sense of Definition \ref{Def: joints}. Therefore, the consistency of a scattering diagram is defined as in \S \ref{Subsubsect: Consistency around codim zero joints}.

\begin{example}
\label{exam: T0}
The scattering diagram $T_0\foD^1_{(\widetilde X,\widetilde D)}$ as defined
in \eqref{Eq: asymptotic canonical } may be viewed as
a scattering diagram as described above. If $(T_0\fod,f_{\fod})\in
T_0\foD^1_{(\widetilde X,\widetilde D)}$, then a priori,
in the setup of Definition \ref{Def: canonical scattering walls}, 
we have $f_{\fod}\in \kk[\shP^+_x]/I_x$ for some any $x\in\Int(\fod)$.
However, by Theorem \ref{thm:D defined deeper codim},
in fact $f_{\fod}\in \kk[\shP^+_0]/I_0$. Thus,
$T_0\foD^1_{(\widetilde X,\widetilde D)}$ may be viewed in this
revised setup, taking $P=\shP^+_0$ and $r:P\rightarrow M$ given by
$m\mapsto \bar m$.

Further, Theorem \ref{thm:consistency at deeper strata} then implies
that $T_0\foD^1_{(\widetilde X,\widetilde D)}$ is consistent in the above
sense.
\end{example}

\subsubsection{Widgets}
We now need to state the higher dimensional analogue of
the Kontsevich-Soibelman lemma, see \cite[Thm.\ 6]{KS} and 
\cite[Thm. 1.4]{GPS}. We cannot, however, start
with an arbitrary scattering diagram. The basic objects
we start with are as follows.
\begin{definition}
\label{Def: tropical hypersurface}
A \emph{tropical hypersurface} in $M_{\RR}$
is a fan $\T$ in $M_{\RR}$ whose support $|\T|$ is pure dimension $\dim M_{\RR}-1$, along
with a positive integer weight attached to each cone of $\T$ of
dimension $\dim M_{\RR}-1$, which satisfies the following
balancing condition.
For every $\omega\in\T$ of dimension $\dim M_{\RR}-2$, let $\gamma$
be a loop in $M_{\RR}\setminus \omega$ around an interior point
of $\omega$, intersecting top-dimensional 
cones $\sigma_1,\ldots,\sigma_p$ of $\T$ of weights $w_1,\ldots,w_p$.
Let $n_i\in N$ be the primitive element associated with the crossing
of $\sigma_i$ by $\gamma$ in the usual convention. Then
\begin{equation}
\label{Eq:balancing}
\sum_{i=1}^p w_in_i=0.
\end{equation}
\end{definition}
\begin{definition}
\label{def:widget}
Suppose we are given
\begin{enumerate}
\item A complete toric fan $\Sigma$ in $M_{\RR}$.
\item A ray $\rho_0\in\Sigma$ with primitive generator $m_0\in M\setminus\{0\}$.
\item A tropical hypersurface $\T$ in $(M/\ZZ m_0)\otimes_{\ZZ} \RR$ with
support contained in the union of codimension one cones of the
quotient fan $\Sigma(\rho_0)$.
\item An element $f_0\in\widehat{\kk[P]}$ such that $f_0=\sum c_mz^m$
with $r(m)$ positively proportional to $m_0$.
\end{enumerate}
Let $\pi:M_{\RR}\rightarrow M_{\RR}/\RR m_0$ be the projection.
Then the \emph{widget} associated to this data is the scattering diagram
\begin{equation}
\label{Eq:widget}
\foD_{m_0}:=\{(\fod_{\sigma}, f_0^{w_{\sigma}})\,|\,\sigma\in\T_{\max}\},
\end{equation}
where $\fod_{\sigma}$ is the unique codimension one cell of $\Sigma$ containing $\RR_{\geq 0}m_0$ and with $\pi(\fod_{\sigma})=\sigma$. 
\end{definition}
We illustrate two widgets $\foD_{e_1}$ and $\foD_{e_2}$ in Figure 
\ref{Fig: InitialP3}. 
\begin{definition}
\label{def:incoming}
We say a wall of a scattering diagram $(\fod,f_{\fod})$ with direction $m_0$ is \emph{incoming} if
\[
\fod=\fod-\RR_{\ge 0} m_0.
\]
\end{definition}

Note that by definition a widget is a union of incoming walls. In what follows we denote the relative boundary of a widget $\foD_{m_0}$ in $M_{\RR}$ by $\partial (\foD_{m_0})$.

\begin{lemma}
\label{widgetlemma}
Let $\foD_{m_0}$ be a widget. Then 
\[
\Sing(\foD_{m_0}) \setminus \partial (\foD_{m_0}) =\left(\bigcup_{\rho\in\T} \pi^{-1}(\rho)\right)
\cap \Supp(\foD_{m_0}),
\]
where the union is over all $\rho$ of dimension $\dim M_{\RR}-3$.
If $\gamma$ is a loop in $\mathrm{Star}(\rho_0)\setminus \Sing(\foD_{m_0})$, where $\rho_0$ is as in Definition \ref{def:widget},
then $\theta_{\gamma,\foD_{m_0}}=\id$.
\end{lemma}

\begin{proof}
The first statement is obvious. For the second statement, 
it is enough to consider a loop $\gamma$ around $\pi^{-1}(\rho)$
for $\rho$ of dimension $\dim M_{\RR}-3$. Let $\sigma_1,\ldots,\sigma_p\in\T_{\max}$
contain $\rho$ with weights $w_1,\ldots, w_p$, 
$\fod_i=\pi^{-1}(\sigma_i)$, and suppose
that $\theta_{\gamma,\fod_i}$ is defined using $n_i\in N$, so that
\begin{align*}
\theta_{\gamma,\foD}(z^m)= {} &
z^m\prod_{i=1}^n f_0^{w_i\langle n_i,r(m)\rangle}\\
= {} & z^m f_0^{\langle \sum_i w_in_i,r(m)\rangle}\\
= {} & z^m
\end{align*}
by the balancing condition \eqref{Eq:balancing}.
\end{proof}
Let $\foD_{m_i}$ be a widget associated to a primitive vector $m_i \in M\setminus \{ 0 \}$, for
$i \in {1,\ldots, n}$. Denote their union by
\[  \foD= \bigcup_{i=1}^n \foD_{m_i} \]
The higher-dimensional analogue of the Kontsevich-Soibelman Lemma is then:

\begin{theorem}
\label{thm:higher dim KS}
There is a consistent scattering diagram
$\Scatter(\foD)$ containing $\foD$ such that \[\Scatter(\foD)\setminus
\foD\] consists only of non-incoming walls.
Further, this scattering diagram is unique up to equivalence.
\end{theorem}

\begin{proof}
{\bf Step I. Order by order construction of $\Scatter(\foD)$.}
We construct $\Scatter(\foD)$ order by order, assuming we have
constructed $\foD_{k-1}\supseteq \foD$ which is consistent modulo
$\fom_P^k$. For the base case, $k=1$, we take $\foD_0=\foD$.
This works as modulo $\fom_P$, all wall-crossing group elements are
the identity.

We now construct $\foD_k$ from $\foD_{k-1}$. Let $\foD_{k-1}'$
consist of those walls in $\foD_{k-1}$ such that $f_{\fod}\not\equiv 1
\mod \fom_P^{k+1}$. By definition of scattering diagram, this is a finite
set. We recall the notion of joint from Definition 
\ref{Def: canonical scattering walls}. We also define an
\emph{interstice} to be a facet of a joint, hence a dimension $n-3$
polyhedron. We denote $\Joints(\foD'_{k-1})$ 
and $\Interstices(\foD'_{k-1})$ to be the set of joints and interstices
respectively of $\Sing(\foD'_{k-1})$. These cells form the
top-dimensional and codimension one cells of a polyhedral cell complex
structure on $\Sing(\foD'_{k-1})$.

We first carry out a standard procedure for each
joint. Let $\foj\in\Joints(\foD'_{k-1})$, and let $\Lambda_{\foj}\subseteq
M$ be the set of integral tangent vectors to $\foj$. If $\gamma_{\foj}$
is a simple loop around $\foj$ small enough so that it only intersects
walls containing $\foj$, we note that every group element
$\theta_{\gamma_{\foj},\fod}$ contributing to 
$\theta_{\gamma_{\foj},\foD'_{k-1}}$ lies in $\VV_{\fom_P^{k+1},\foj}$. 
Thus modulo
$\fom_P^{k+1}$ we can write
\[
\theta_{\gamma_p,\foD'_{k-1}}=\exp\left(\sum_{i=1}^s c_iz^{m_i}
\partial_{n_i}\right)
\]
with $c_i\in\kk$, $m_i\in \fom_P^k\setminus \fom_P^{k+1}$ 
and $n_i\in r(m_i)^{\perp}
\cap \Lambda_{\foj}^{\perp}$ primitive. Let
\[
\foD[\foj]:=\{(\foj-\RR_{\ge 0}m_i,1\pm c_iz^{m_i})\,|\,\hbox{$i=1,\ldots,s$
and $r(m_i)\not\in\Lambda_{\foj}$}\}.
\]
Here the sign is chosen in each wall so that its contribution to
$\theta_{\gamma_{\foj},\foD[\foj]}$ is $\exp(-c_iz^{m_i}\partial_{n_i})$
modulo $\fom_P^{k+1}$.

We now take 
\[
\foD_k:=\foD_{k-1}\cup\bigcup_{\foj} \foD[\foj].
\]

\medskip

{\bf Step II. Analysis of joints of $\foD_{k}$.}
Consider a joint
$\foj\in\Joints(\foD_k)$. There are three types of walls $\fod$ in $\foD_k$
containing $\foj$:
\begin{enumerate}
\item $\fod\in\foD_{k-1}\cup\foD[\foj]$.
\item $\fod\in\foD_{k}\setminus(\foD_{k-1}\cup\foD[\foj])$, 
but $\foj\not\subseteq
\partial\fod$. This type of wall does not contribute to 
$\theta_{\gamma_{\foj},\foD_k}$, as the associated group elements
lie in the center of $\VV_{\fom_P^{k+1},\foj}$,
and in addition this wall contributes twice to $\theta_{\gamma_{\foj},\foD_k}$,
with the two contributions inverse to each other.
\item $\fod\in\foD_k\setminus(\foD_{k-1}\cup \foD[\foj])$ and
$\foj\subseteq\partial\fod$. Necessarily
$f_{\fod}=1+c_iz^{m_i}$ for some $m_i$ with $r(m_i)\in\Lambda_{\foj}$.
So the associated group element lies in ${}^{\|}\VV_{\fom_P^{k+1},\foj}$, and
again lies in the center of $\VV_{\fom_P^{k+1}}$.
\end{enumerate} 
From this we see that by construction of $\foD[\foj]$,
\[
\theta_{\gamma_{\foj},\foD_k}\in{}^{\|}\VV_{\fom_P^{k+1},\foj}.
\]
A priori, this is not the identity, and we have, modulo $\fom_P^{k+1}$,
\[
\theta_{\gamma_{\foj},\foD_k}=\exp\left(\sum_{m\in \fom_P^k\setminus
\fom_P^{k+1}} z^m \partial_{n(m,\foj)}\right)
\]
where the constants in $\kk$ have been absorbed into $n(m,\foj)
\in N\otimes_{\ZZ}\kk$, and $n(m,\foj)=0$ if $r(m)$ is not tangent
to $\foj$. Note that if we fix an orientation on $M_{\RR}$, then
a choice of orientation on a joint $\foj$ determines a choice of
orientation for a loop $\gamma_{\foj}$, and changing the orientation
on the loop changes the sign of $n(m,\foj)$. Thus we can view, for
a fixed $m\in \fom^k_P\setminus\fom^{k+1}_P$, the map $\foj\mapsto
n(m,\foj)$ as an $N\otimes_{\ZZ}\kk$-valued $(n-2)$-chain for the polyhedral
complex $\Sing(\foD_k)$, the orientation
on $\foj$ being implicit.

We wish to show that in fact $\theta_{\gamma_{\foj},\foD_k}$ is the
identity, i.e., show that $n(m,\foj)=0$ for any joint $\foj$ and
any $m\in \fom^k_P\setminus\fom^{k+1}_P$.

\medskip

{\bf Step III. 
$\foj\mapsto n(m,\foj)$ is a cycle.}
To show this, we consider an interstice $\foc\in
\Interstices(\foD_k)$, and let $B=S^2\subseteq M_{\RR}$ be a small sphere
contained in a three-dimensional affine subspace of $M_{\RR}$
which intersects the interior of
$\foc$ transversally at one point. The sphere $B$
should be centered at this point. Let $\foj_1,\ldots,\foj_n$
be the joints of $\foD_k$ containing $\foc$; then $\foj_i$ intersects
$B$ at a single point, $y_i$. 

Fix a base-point $x\in B$, $x\not= y_i$ for any $i$. After choosing an
orientation on $B$, let $\gamma_i$ be a positively oriented loop
around $y_i$ based at a point $y_i'$ near $y_i$; we can choose
these along with paths $\beta_i$ joining $x$ to $y_i'$ so that
the loop
\[
\gamma:=\beta_1\gamma_1\beta_1^{-1}\cdots\beta_n\gamma_n\beta_n^{-1}
\]
is contractible in $B\setminus \{y_1,\ldots,y_n\}$, hence
$\theta_{\gamma,\foD_k}=\id$.
On the other hand, by the inductive assumption, 
$\theta_{\gamma_i,\foD_k}=\id\mod \fom_P^k$, and hence
$\theta_{\beta_i}$ commutes with $\theta_{\gamma_i}\mod\fom_P^{k+1}$. Thus
\[
\theta_{\gamma,\foD_k}=\theta_{\gamma_n,\foD_k}\circ
\cdots\circ\theta_{\gamma_1,\foD_k}.
\]
Since $\theta_{\gamma_i,\foD_k}
=\exp\left(\sum_m z^m\partial_{n(m,\foj_i)}\right)$,
we see that
\[
\sum_{i=1}^n n(m,\foj_i)=0,
\]
for each $m$.
This is precisely the cycle condition, completing this step.

{\bf Step IV. $\foD_{k}$ is consistent modulo $\fom_P^{k+1}$.}
It is sufficient to show that given $m$, $n(m,\foj)=0$ for
all joints $\foj$. First note that if $r(m)=0$, then $n(m,\foj)=0$
anyway, as terms $z^m\partial_n$ with $r(m)=0$ don't appear in $\fov$. 
Otherwise, for a joint $\foj$, consider the line $L_x=x+\RR r(m)$
for $x\in\Int(\foj)$. Note that $n(m,\foj)=0$ anyway unless $r(m)$
is tangent to $\foj$, so as $x$ moves $L_x$ varies in an $(n-3)$-dimensional
family. Since the boundary of an interstice is dimension $n-4$, 
we can choose $x$ so that for any interstice $\foc$ that $L_x$
intersects, $L_x$ intersects only the interior of $\foc$. Thus
there exists real numbers $0=\lambda_0
<\lambda_1<\cdots<\lambda_s\le\infty$ and joints
$\foj_1=\foj$, $\foj_2,\ldots,\foj_s$ with $x+\lambda r(m)
\in\Int\foj_i$ for $\lambda\in(\lambda_{i-1},\lambda_i)$, and
$s$, $\lambda_s$ maximal for this property. Suppose that $\foj_i$ and
$\foj_{i+1}$ meet at an interstice $\foc_i$. Since $r(m)$ is not tangent
to $\foc_i$ by the general choice of $x$, we see that $r(m)$ can only
be tangent to at most two joints containing $\foc_i$, and these must
be $\foj_i$ and $\foj_{i+1}$. By the cycle condition, it then follows
that $n(m,\foj_i)=n(m,\foj_{i+1})$. Thus inductively $n(m,\foj)=
n(m,\foj_s)$. Furthermore, if $\lambda_s\not=\infty$, then we can
conclude that $n(m,\foj_s)=0$.

Otherwise, since by construction, $\theta_{\gamma_{\foj_s},\foD_k}\in
{}^{\|}\VV_{\foj_s}$, we can compute $\theta_{\gamma_{\foj_s},\foD_k}$
in $\VV_{\foj_s}/{}^{\perp}\VV_{\foj_s}$. 
If $\fod$ contains $\foj_s$, then $\theta_{\gamma_{\foj_s},\fod}$ is
non-trivial in $\VV_{\foj_s}/{}^{\perp}\VV_{\foj_s}$ only if the
direction of $\fod$ is tangent to $\foj_s$. There are two sorts of walls
in $\foD_k$ with this property: incoming walls contained in a widget 
containing $\foj_s$, and non-incoming walls. From Lemma \ref{widgetlemma}
and the fact that ${}^{\|}\VV_{\foj_s}$ is abelian, one sees in fact
only the non-incoming walls contribute to $\theta_{\gamma_{\foj_s},\foD_k}$.
However, a non-incoming wall with $z^m$ appearing in $f_{\fod}$
must have direction negatively proportional to $r(m)$, and cannot
be unbounded in the direction $r(m)$. This contradicts the
fact that $\foj_s$ is by choice unbounded in this direction. Hence
$n(\foj_s,m)=0$.

\medskip

{\bf Step V. Uniquess of $\foD_k$ given uniqueness of $\foD_{k-1}$.}
This is similar to the argument in Step VI of the proof of
Theorem 1.28 in \cite[App.\ C]{ghkk}. We omit the details.
\end{proof}

\subsubsection{The scattering diagram $\foD_{(X_\Sigma,H)}$}
\label{Subsec: scattering in MR}
We will now give details of the key scattering diagram arising from the
setup in the previous subsection, which will be the higher dimensional analogue of the one constructed in \cite{GPS}. Let $\Sigma$ be a complete toric fan in $M_{\RR}$ and denote by $X_{\Sigma}$ the associated complete toric variety, as in \S\ref{Sec: Pulling singularities out}, along with the data of hypersurfaces $H=(H_1,\ldots,H_s)$ with $H_i \subset D_{\rho_i}$ in its toric boundary divisor corresponding to rays $\mathbf{P}=(\rho_1,\ldots,\rho_s)$.  We now use this data to determine a scattering diagram $\foD_{(X_{\Sigma},H)}$. First, decompose $H_i=\bigcup_{j=1}^{s_i} H_{ij}$
into connected components. Take
\begin{equation}
\label{Eq: P}
    P:=M\oplus \bigoplus_{i=1}^s\NN^{s_i}.
\end{equation}
We write $e_{i1},\ldots,e_{is_i}$ for the generators of $\NN^{s_i}$, and
write $t_{ij}:=z^{e_{ij}}\in \kk[P]$.
We define $r:P\rightarrow M$ to be the projection.

For each $1\le i \le s$, $1\le j\le s_i$, we give the data determining a widget
via Definition \ref{def:widget}: First, take $m_0=m_i$, the
primitive generator of the ray $\rho_i$. Second,
consider a tropical hypersurface
\[\scrH_{ij}\subseteq (M/m_i\ZZ)\otimes_{\ZZ}\RR\] determined by
$H_{ij}\subseteq D_{\rho_i}$. This is the tropical hypersurface
supported on the codimension $1$ skeleton of the toric fan $\Sigma(\rho_i)$ of the divisor $D_{\rho_i}$ defined as in \eqref{Eq: sigma rho},
with the weight on a cone $(\sigma+\RR\rho_i)/\RR\rho_i$ being
$w_{\sigma}=D_{\sigma}\cdot H_{ij}$, where the intersection number
is calculated on $D_{\rho_i}$. 
Third, we take for $f_0$ the function
\[
f_{\rho_i}:=1+t_{ij}z^{m_i}.
\]
This data now determines a widget $\foD_{ij}$ using \eqref{Eq:widget}.
We then set
\begin{equation}
    \label{Eq: widgets}
\foD_{(X_\Sigma,H),\mathrm{in}}=\bigcup_{i=1}^s\bigcup_{j=1}^{s_i} \foD_{ij}.
\end{equation}
We will denote the consistent scattering diagram obtained from \eqref{Eq: widgets} by
\[ 
\foD_{(X_\Sigma,H)} :=
\Scatter(\foD_{(X_\Sigma,H),\mathrm{in}}).
\]

\begin{remark}
\label{rem:smaller monoid}
It will be convenient later to use a submonoid $P'\subseteq P$, the free
submonoid with generators $e'_{ij}=(m_i,e_{ij})$. Then
$\foD_{(X_{\Sigma},H),\inc}$ can be viewed as a scattering diagram
defined using this smaller monoid, and hence the same is true of
$\foD_{(X_{\Sigma},H)}$. 
\end{remark}

\begin{example}
\label{Ex: P3 with two lines}
This example was calculated in collaboration with Tom Coates. Let $X_{\Sigma} = \mathbb{P}^3$, corresponding to the fan with ray generators $\langle e_1,e_2,e_3,e_4 =-e_1-e_2-e_3  \rangle$ and let $D=\bigcup_i D_i$, for $1\leq i \leq 4$, 
be the toric boundary divisor. Choose
$H=\ell_1 \cup \ell_2$ where $\ell_1\subset D_1$ and $\ell_2 \subset D_2$ are two general disjoint lines. The associated widgets are illustrated in Figure \ref{Fig: InitialP3}.
\begin{figure}
\input{TwoWidgets.pspdftex}
\caption{Walls of $\foD_{\PP^3,\mathrm{in}}$.}
\label{Fig: InitialP3}
\end{figure}
where we denote by $x=z^{e_1},y=z^{e_2},z=z^{e_3}$ the standard coordinates in $\RR^3$. We display the walls of the minimal scattering diagram equivalent to the consistent scattering diagram $\foD_{(X_\Sigma,H)}$ in Table \ref{Table: walls of final HDTV P3}, where the first two rows correspond to the walls of the initial scattering diagram $\foD_{\PP^3,\mathrm{in}}$. We denote by $\langle e_i,e_j \rangle$ the support of the codimension one cone spanned by $e_i$ and $e_j$, and we use analogous notation for cones generated by linear combinations of $e_i$'s and $e_j$'s in what follows.
\begin{table}[]
    \centering
    \begin{tabular}{ll} \hline
 $\fod$ & $f_{\fod}$ \\ \hline
  $\langle e_1,e_2 \rangle,\langle e_1,e_3 \rangle,\langle e_1,e_4 \rangle$ & ~~ $1+t_1x$  \\ 
$\langle e_2,e_1 \rangle,\langle e_2,e_3 \rangle,\langle e_2,e_4 \rangle$ & ~~$1+t_2y$ \\ 
  $\langle e_3,-e_1 \rangle,\langle e_4,-e_1 \rangle$ & ~~ $1+t_1x$  \\  
  $\langle e_3,-e_2 \rangle,\langle e_4,-e_2 \rangle$ & ~~ $1+t_2y$  \\ 
 $\langle -e_2,-e_1-e_2 \rangle,\langle -e_1,-e_1-e_2 \rangle,\langle e_3,-e_1-e_2 \rangle,\langle e_4,-e_1-e_2 \rangle$ & ~~$1+t_1t_2xy$ \\
 $\langle e_1,-e_2 \rangle$ & ~~ $1+t_2y+t_1t_2xy$  \\ 
  $\langle e_2,-e_1 \rangle$ & ~~ $1+t_1x+t_1t_2xy$  \\ 
\hline
  \end{tabular}
    \caption{Walls of $\foD_{(\PP^3,\ell_1\cup \ell_2)} $}
    \label{Table: walls of final HDTV P3}
\end{table}
\begin{figure}
\input{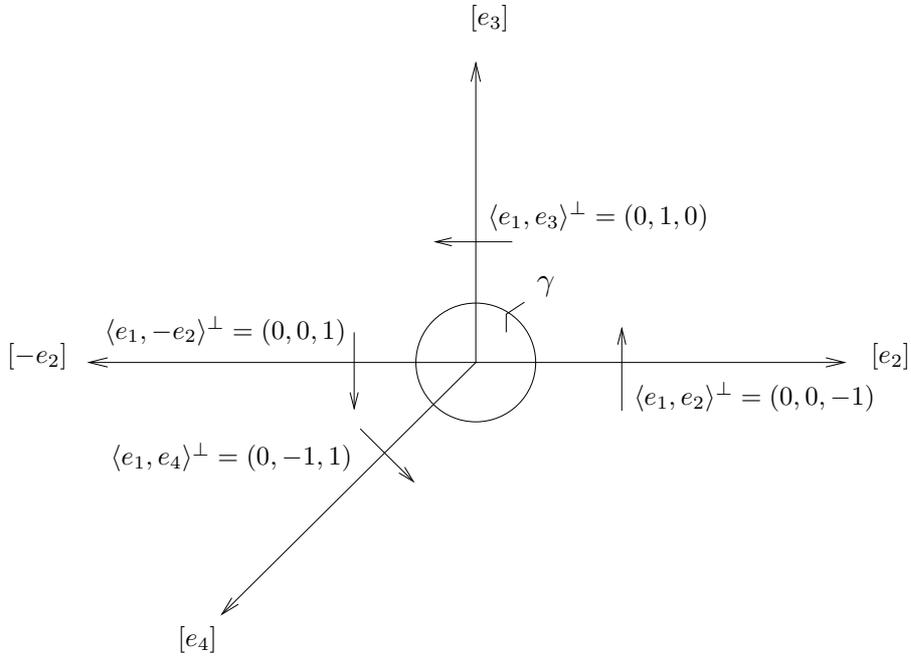}
\caption{The projection along $e_1$ of the walls of $\foD_{(\PP^3,\ell_1\cup \ell_2)}$ whose support contain $e_1$.}
\label{Fig: AroundX}
\end{figure}
The attached functions on the walls in the last two rows are obtained from an infinite product
\begin{eqnarray}
\label{Eq: 1+y+xy}
1+t_2y+t_1t_2xy & = & (1+t_2y)(1+t_1t_2xy)(1-t_1t_2^2xy^2)(1+t_1t_2^3xy^3)\ldots \\
\label{Eq: 1+x+xy}
1+t_1x+t_1t_2xy & = & (1+t_1x)(1+t_1t_2xy)(1-t_1^2t_2x^2y)(1+t_1^3t_2x^3y)\ldots 
\end{eqnarray}
Each wall $1 + (-1)^{b-1} t_1^at_2^bx^ay^b$ in the expansion \eqref{Eq: 1+y+xy} has support in $\langle e_1, (-a,-b,0) \rangle$.
However, we also obtain walls $1 + (-1)^{b} t_1^at_2^bx^ay^b$ with support $\langle -e_2,  (-a,-b,0) \rangle$ to achieve consistency around the joint generated by $-e_2$. Hence, after all cancellations there remains a single wall with the attached function $1+t_2y+t_1t_2xy$, whose support is $\langle e_1, - e_2 \rangle$. The wall with the attached function $1+t_1x+t_1t_2xy$, whose support is $\langle -e_1, e_2 \rangle$, is obtained analogously.

We illustrate the projections of the walls of $\foD_{(\PP^3,\ell_1\cup \ell_2)}$  around the joint generated by $e_1=\langle 1,0,0 \rangle$ in Figure \ref{Fig: AroundX}. To check for consistency around this joint, we need to check that the wall crossing automorphisms defined in \eqref{Eqn: theta_fop} along a loop $\gamma$ as illustrated Figure \ref{Fig: AroundX} leaves each of the monomials $x,y,z$ invariant. This follows immediately for $x$. For $y$ we have;
\begin{align*}
   y  & \xrightarrow{\langle e_1,e_3 \rangle} ~~ y(1+t_1x)  \xrightarrow{\langle e_1,-e_2 \rangle} ~~y(1+t_1x)
\\ 
  & \xrightarrow{\langle e_1,e_4 \rangle}  ~~ y(1+t_1x)^{-1}(1+t_1x)=y   \xrightarrow{\langle e_1,e_2 \rangle}  ~~y \\
\end{align*}
where the superscripts on the arrows indicate the support of the walls $\gamma$ crosses. Similarly, for $z$ we compute;
\begin{align*}
 z  & \xrightarrow{\langle e_1,e_3 \rangle} ~~ z   \xrightarrow{\langle e_1,-e_2 \rangle} ~~z(1+t_2y+t_1t_2xy)  \xrightarrow{\langle e_1,e_4 \rangle} ~~   ~~z(1+t_1x)(1+t_2y) \\
 & \xrightarrow{\langle e_1,e_2 \rangle} ~~ z(1+t_1x)^{-1}(1+t_2y)^{-1}(1+t_1x)(1+t_2y) = z
\end{align*}
Therefore, consistency of $\foD_{(\PP^3,\ell_1\cup\ell_2)}$ around the joint $e_1$ follows. The computation around the other joints is analogous, and is left to the reader. Hence, $\foD_{(\PP^3,\ell_1\cup\ell_2)}$ is consistent. 
\end{example}

\section{The main theorem}
\label{sec: at infinity}
We now state a precise version of the main theorem of the paper, which shows 
that the canonical scattering diagram $\foD_{(X,D)}$ defined using punctured invariants of $(X,D)$ can be obtained from the scattering diagram $\foD_{(X_{\Sigma},H)}$
defined in \S\ref{Subsec: scattering in MR}.
To do so, we introduce some additional notation.
Recall the degeneration $\widetilde X$ of $(X,D)$ defined in \eqref{Eq: The degeneration} has associated tropical space $\widetilde B$, which comes together with the map
\[ \widetilde p : \widetilde B \to \RR_{\geq 0}   \]
defined in \eqref{Eq: p}. We use the identification of $\widetilde B_0 :=\widetilde p^{-1}(0)$ with the tropical space $B$ associated to $(X,D)$ as explained in Remark
\ref{rem:relative case}. Thus the piecewise linear identification $\mu: M_{\RR} \to  \widetilde B_0 $ of \eqref{Eq: mu} induces a piecewise linear identification $\Upsilon:M_{\RR}\rightarrow B$.

Let $\mathbf{A}=(a_{ij})$ denote a tuple of integers, $1\le i\le s$,
$1\le j\le s_i$. Consider the collection of vectors
\[
\{a_{ij} m_i\,|\,1\le i\le s, 1\le j\le s_i\}\cup \{m_{\mathbf A}\}
\]
with $m_{\mathbf{A}}=-\sum_{i,j} a_{ij}m_i$. To each vector in this collection we associate a cone in $\Sigma$ as follows: to $a_{ij}m_i$ we associate the cone
$\rho_i$, and to $m_{\mathbf A}$ we associate some arbitrary $\sigma\in\Sigma$
such that $m_{\mathbf A}$ is tangent to $\sigma$. This data then
induces, as in Lemma \ref{lem:beta recovery},
a class $\bar\beta_{\mathbf{A},\sigma}\in N_1(X_{\Sigma})$. 
Recalling the blow-up map $\mathrm{Bl}_H:X\rightarrow X_{\Sigma}$,
there is then a class
\begin{equation}
\label{eq:beta A def}
\beta_{\mathbf{A},\sigma}:=
\mathrm{Bl}_{H}^*(\bar\beta_{\mathbf{A},\sigma})-\sum_{i,j} a_{ij}E^j_i.
\end{equation}
Here $\mathrm{Bl}_H^*(\bar\beta_{\mathbf{A},\sigma})$ is the unique curve class in
$N_1(X)$
whose push-forward to $X_{\Sigma}$ under $\mathrm{Bl}_H$ is
$\bar\beta_{\mathbf{A},\sigma}$ and whose intersection with all 
exceptional divisors is $0$. The curve class $E^j_i$ is the class of 
an exceptional curve of $\mathrm{Bl}_H:X\rightarrow X_{\Sigma}$
over $H_{ij}$.

We may now define a scattering diagram $\Upsilon(\foD_{(X_{\Sigma},H)})$ as follows.
First, we may refine $\foD_{(X_{\Sigma},H)}$ by replacing each
wall with a union of walls with the same wall function. We do
this so that any $\fod\in\foD_{(X_{\Sigma},H)}$ satisfies $\fod\subseteq
\sigma$ for some $\sigma\in\Sigma$. Note we do not need to refine
any walls in $\foD_{(X_{\Sigma},H),\inc}$.
Define $\Upsilon(\foD_{(X_{\Sigma},H),\inc})$ to be 
the collection of walls of the form 
\begin{equation}
\label{eq:Upsilon in}
\left(\Upsilon(\rho), \prod_{j=1}^{\kappa^i_{\ul{\rho}}} (1+t^{E^j_{\ul\rho}}
z^{-m_i})\right)
\end{equation}
with $1\le i\le s$ and
$\rho\in\Sigma$ running over codimension one cones containing $\rho_i$.
The curve class $E^j_{\ul{\rho}}$ is now viewed as lying in 
$N_1(X)$, being one of the $\kappa^i_{\ul{\rho}}$ exceptional curves 
of $\mathrm{Bl}_H:X\rightarrow X_{\Sigma}$ mapping to the one-dimensional
stratum $D_{\rho}$ of $X_{\Sigma}$. Note each
$E^j_{\ul{\rho}}$ coincides with $E^{j'}_i$ for some $i,j'$.

Next, if $(\fod,f_{\fod})\in \foD_{(X_{\Sigma},H)}\setminus
\foD_{(X_{\Sigma},H),\inc}$ with $\fod\subseteq\sigma\in\Sigma$, 
necessarily a monomial of $f_{\fod}$
is of the form  
\[
\prod_{i,j} (t_{ij}z^{m_i})^{a_{ij}}.
\]
for some tuple of integers $\mathbf{A}=(a_{ij})$, $1\leq i\leq s$,
$1\leq j \leq s_i$. We define $\Upsilon_*$ applied to this monomial to be
$t^{\beta_{\mathbf{A},\sigma}}z^{-\Upsilon_*(m_{\mathbf{A}})}$.
Note that $\Upsilon$ is piecewise linear, hence defines a well-defined
push-forward of tangent vectors to cones of $\Sigma$. This allows us to define
$\Upsilon_*(f_{\fod})$, and then set
\[
\Upsilon( \foD_{(X_{\Sigma},H)}\setminus \foD_{(X_{\Sigma},H),\inc}):=
\{(\Upsilon(\fod), \Upsilon_*(f_{\fod}))\,|\,
(\fod,f_{\fod})\in \foD_{(X_{\Sigma},H)}\setminus
\foD_{(X_{\Sigma},H),\inc}\}.
\]
Finally, we define
\[
\Upsilon(\foD_{(X_{\Sigma},H)}):=\Upsilon(\foD_{(X_{\Sigma},H),\inc})\cup
\Upsilon( \foD_{(X_{\Sigma},H)}\setminus
\foD_{(X_{\Sigma},H),\inc}).
\]

\begin{theorem}
\label{Thm: HDTV}
There is an equivalence of scattering diagrams between
$\Upsilon(\foD_{(X_{\Sigma},H)})$ and $\foD_{(X,D)}$.
\end{theorem}

This theorem will be proved in \S\ref{subsec:T0 to asymptotic}.

\subsection{Equivalence of $\foD_{(X_{\Sigma},H)}$ and $T_0\foD_{(\widetilde X,\widetilde D)}^1$}
\label{subsec:equivalence}
We will compare the consistent scattering diagrams
$\foD_{(X_{\Sigma},H)}$ and
$T_0\foD^{1}_{(\widetilde X,\widetilde D)}$, the latter defined in 
\eqref{Eq: asymptotic canonical }
and discussed further in Example \ref{exam: T0}.
To make this comparison, recall these scattering diagrams involve a choice
of $r:P\rightarrow M$. 

For the scattering diagram $\foD_{(X_{\Sigma},H)}$, we take
the monoid to be $P'\subseteq M\oplus\bigoplus_{i=1}^s \NN^{s_i}$
with generators $e'_{ij}=(m_i,e_{ij})$ as in Remark \ref{rem:smaller monoid}.

For $T_0\foD^{1}_{(\widetilde X,\widetilde D)}$, we take
$P=\shP^+_0$, which may be described explicitly as follows.
After making a choice
of representative $\varphi_0$ for the MVPL function
$\varphi$ in the star of $0$, 
as promised by Proposition \ref{Prop: varphi0},
we may extend $\varphi_0$ linearly on each cone of $\Sigma$
to obtain a PL function $\varphi_0:M_{\RR}\rightarrow Q^{\gp}_{\RR}$.
Using the description \eqref{Eq: Pplus description}, we may then write
\begin{equation}
\label{eq:P0 plus again}
\shP^+_0=\{(m, \varphi_0(m)+q)\,|\, m\in M, q\in Q\}\subseteq M\oplus Q^{\gp}.
\end{equation}

We may now define a monoid homomorphism
\[
\nu: P' \rightarrow \shP^+_0
\]
by
\[  \nu(e'_{ij})=\left(m_i,\varphi_0(m_i)+F_i-E^j_i\right). \]
Here we view $E^j_i$ as a curve class on $\widetilde X$
under the inclusion $X\hookrightarrow \widetilde X$. Note that this curve class is the class of an exceptional curve of $\mathrm{Bl}_{\widetilde H}:
\widetilde X\rightarrow X_{\widetilde\Sigma}$.

Next, we define $\nu_*:\kk[P'] \rightarrow \kk[\shP^+_0]$
by $\nu_*(z^p)=z^{\nu(p)}$. This allows us to define, for a scattering
diagram $\foD$ with respect to the monoid $P'$,
\begin{equation}
    \label{Eq: nu of scattering diagram}
    \nu(\foD):=\{(\fod,\nu_*(f_{\fod}))\,|\, (\fod,f_{\fod})\in
\foD\},
\end{equation}
a scattering diagram with respect to the monoid $\shP^+_0$. 
It is immediate from definitions that if $\foD$ is consistent,
so is $\nu(\foD)$.

The key comparison result is the following:
\begin{theorem}
\label{thm:key theorem}
The scattering diagram $T_0\foD_{(\widetilde X,\widetilde D)}^1$ is equivalent to 
$\nu(\foD_{(X_{\Sigma},H)})$.
\end{theorem}

Before embarking on the proof, we will make some more
generally useful observations about $T_0\foD_{(\widetilde X,\widetilde D)}^1$.
These statements all follow immediately 
from Theorem \ref{thm:radial directions} and
Lemma \ref{lem: multiple covers}.
\begin{remarks}
\label{rem:T0 observations}
(1) $T_0\foD_{(\widetilde X,\widetilde D)}^1$ is equivalent
to a scattering diagram $\foD_1 \cup \foD_2$
where
\[
\foD_1:=\{(\rho,  
\prod_{j=1}^{\kappa_{\ul{\rho}}^i}(1+
z^{(m_i,\varphi_0(m_i)+F_i-
E^j_{\ul{\rho}})})
\,|\,
\hbox{$1\le i \le s, \ul{\rho}\in\Sigma(\rho_i)$ of codim $1$}\}.
\]
Here, the factor $1+z^{(m_i,\varphi_0(m_i)+F_i-E^j_{\ul\rho})}$
arises simply by rewriting the factor $1+t^{F_i-E^j_{\ul{\rho}}}z^{m_i}$
appearing in Lemma \ref{lem: multiple covers}
as an element of $\kk[\shP^+_0]$, using parallel transport in
$\shP^+$ via \eqref{eq:psi tau sigma}.

Further, using the factorization of Lemma \ref{lem:product uniqueness},
we may assume that each wall $(\fod,f_{\fod})$ of $\foD_2$ has a direction
$m_{\fod}\in M\setminus\{0\}$, i.e., 
$f_{\fod}=\sum_p c_pz^p$ with $r(p)$ negatively proportional to
$m_{\fod}$ for each $p$ occuring in the sum. We also note that for each $\fod\in\foD_2$,
$\fod\subseteq\sigma$ for some $\sigma\in\Sigma$.

(2) $\foD^1_{(\widetilde X,\widetilde D)}$ may be recovered, up to
equivalence, from $\foD_1,\foD_2$, as follows.
\begin{itemize}
\item[(a)] 
$\foD_1$ gives rise to the following walls in $\foD^1_{(\widetilde X,
\widetilde D)}$.
For each $i$ and $\rho_i\subseteq\rho\in\Sigma$ with $\rho$
of codimension one, we have two walls in $\foD^{1}_{(\widetilde X,\widetilde D)}$:
\begin{equation}
\label{eq:slab walls in}
\left(\rho_0,\prod_{j=1}^{\kappa^i_{\ul{\rho}}} (1+t^{F_i-E^j_{\ul\rho}}
z^{m_i})\right)
\end{equation}
and
\begin{equation}
\label{eq:slab walls out}
\left(\rho_{\infty},\prod_{j=1}^{\kappa^i_{\ul{\rho}}} (1+t^{E^j_{\ul\rho}}
z^{-m_i})\right).
\end{equation}
\item[(b)] 
Each wall $(\fod,f_{\fod})\in \foD_2$,
with $f_{\fod}$ written with exponents
in $\shP^+_0$, gives precisely one or two corresponding walls
in $\foD^{1}_{(\widetilde X,\widetilde D)}$ as follows. Let $\sigma\in\Sigma$
be the smallest cone containing $\fod$. Suppose first that
there is no $i$ with $\rho_i\subseteq\sigma$.
Then using the piecewise linear identification
$\Psi|_{\widetilde B_1}:\widetilde B_1\rightarrow M_{\RR}$, $\fod$ can be
considered as a wall in $\widetilde B_1$. 
We use parallel transport in $\shP^+$ to rewrite $f_{\fod}$ with exponents
in $\shP^+_x$ for $x\in\Int(\bar\sigma\cap\fod)$.

Next suppose that $\rho_i\subseteq\sigma$. Then we get two walls with
support $\fod\cap \tilde\sigma'$ and $\fod\cap\tilde\sigma$ respectively.
The wall functions of these two walls come from $f_{\fod}$,
after parallel transport in $\shP^+$ from $\shP_0^+$. Note that
in passing from $\fod\cap\tilde\sigma'$ to $\fod\cap\tilde\sigma$,
one must account for an additional kink, which we shall do explicitly
later.
\end{itemize}
\end{remarks}

\begin{proof}[Proof of Theorem \ref{thm:key theorem}.]
Let us use the decomposition of $T_0\foD_{(\widetilde X,\widetilde D)}^1
=\foD_1\cup\foD_2$, up to equivalence, of Remarks \ref{rem:T0 observations},(1).

Recall from Definition \ref{def:incoming} the notion of an incoming wall.
Note that all walls of $\foD_1$ are incoming.

We first show that no wall of $\foD_2$ is incoming. 
Suppose $(\fod,f_{\fod})\in\foD_2$ was incoming, with $\fod\subseteq
\sigma\in\Sigma$. Let $m_\fod$ be the direction of $\fod$, so that 
$-m_\fod\in\fod$.
Suppose first that $\sigma$ does not contain
$\rho_i$ for any $i$. Then by Remark
\ref{rem:T0 observations}, (2), $\fod'=\Psi|_{\widetilde B_1}^{-1}(\fod)$
is a wall of $\foD^1_{(\widetilde X,\widetilde D)}$, and
necessarily $\fod'$ does not have a repulsive face as $-m_\fod\in\fod$.
However, by Lemma \ref{lem:canonical repulsive},
every wall in the canonical scattering diagram has
a repulsive face. We note, though, that the wall just described is only a wall
in a scattering diagram \emph{equivalent} to the canonical scattering diagram.
Nevertheless, if such a wall were to exist, the ray 
$\Psi|_{\widetilde B_1}^{-1}(\RR_{\le 0}m_\fod)$ would then have to pass through
an infinite sequence of walls, all with direction $m_\fod$, in the original
canonical scattering diagram.
But, as we work over an ideal $I$, the canonical scattering
diagram only contains a finite number of walls, so this can't happen.
Thus we obtain a contradiction.

If instead $\rho_i\subseteq\sigma$, then as in Remark
\ref{rem:T0 observations} (2), we obtain a wall with support
$\Psi_{\widetilde B_1}^{-1}(\fod)\cap (\tilde\sigma \cap\widetilde B_1)$.
The argument is then the same.

Having shown no wall in $\foD_2$ is incoming, it follows from
the uniqueness statement of Theorem \ref{thm:higher dim KS} that
$T_0\foD^1_{(\widetilde X,\widetilde D)}$ is equivalent to $\Scatter(\foD_1)$.
It is thus enough to show that
\begin{equation}
\label{eq:still to be shown}
\foD_1=\nu(\foD_{(X_{\Sigma},H),\mathrm{in}}).
\end{equation}

Let $\rho\in\Sigma$ be codimension one with $\rho_i\subseteq\rho$.
If we take the product of the wall functions
of those walls of $\foD_{(X_{\Sigma},H),\inc}$ with support $\rho$,
we obtain
\[
\prod_{j=1}^{s_i} (1+t_{ij}z^{m_i})^{D_{\rho}\cdot H_{ij}}.
\]
Applying $\nu$ gives
\begin{equation}
\label{eq:ffod2}
\prod_{j=1}^{s_i} (1+z^{(m_i,\varphi_0(m_i)+F_i-E^j_i)})^{D_{\rho}\cdot H_{ij}}.
\end{equation}
Noting that $E^j_{\ul{\rho}}$ represents the same class in
$N_1(\widetilde X)$ as $E^{j'}_i$ for some $j'$, and for a given
$j'$ there are precisely $D_{\rho}\cdot H_{ij'}$ of the $E^j_{\ul{\rho}}$
equivalent to $E^{j'}_i$, we see that the wall function of the
wall of $\foD_1$ with support $\rho$ agrees with \eqref{eq:ffod2}.
This shows \eqref{eq:still to be shown}.
\end{proof}

\subsection{From $T_0\foD^{1}_{(\widetilde X,\widetilde D)}$ to $\foD_{(\widetilde X,\widetilde D)}^{1,\mathrm{as}}$ and the proof of the main theorem}
\label{subsec:T0 to asymptotic}
\begin{proof}[Proof of Theorem \ref{Thm: HDTV}.]
First, Theorem \ref{thm:key theorem} shows that $\foD_{(X_{\Sigma},H)}$
determines $T_0\foD^1_{(\widetilde X, \widetilde D)}$. 
From Remark \ref{rem:T0 observations} (2), $T_0\foD^1_{(\widetilde X,
\widetilde D)}$ determines $\foD^1_{(\widetilde X,\widetilde D)}$,
 which in turn determines $\foD_{(\widetilde X,\widetilde D)}$ and
$\foD^{1,\mathrm{as}}_{(\widetilde X,\widetilde D)}$ as in Definition~\ref{def:asymptotic scattering}.
By Lemma \ref{lem:iota injective} and
Proposition \ref{Prop: Asymptotic equivalence}, 
$\foD_{(\widetilde X,\widetilde D)}^{1,\mathrm{as}}$ determines
$\foD_{(X,D)}$. Thus, we just need to trace through the identifications.

Before starting, we note that for $(\fod,f_{\fod})\in
\foD^1_{(\widetilde X,\widetilde D)}$
the exponent $1/\ind(\mathrm{Cone}(\fod))$ arising in \eqref{eq:cone index}
in fact is always $1$. Indeed, by radiance, $\Lambda_{\mathrm{Cone}(\fod)}$ 
always
contains $(0,1)$, and hence $\tilde p_*:\Lambda_{\mathrm{Cone}(\fod)}\rightarrow
\ZZ$ is surjective. Thus in passing from $\foD^1_{(\widetilde X,
\widetilde D)}$ to $\foD^{1,\mathrm{as}}_{(\widetilde X,\widetilde D)}$,
we do not need to worry about this exponent.

From Remark \ref{rem:T0 observations} (2), there are three types
of walls of $\foD^1_{(\widetilde X,\widetilde D)}$
which contribute to $\foD^{1,\mathrm{as}}_{(\widetilde X,\widetilde D)}$.
Indeed, walls of type \eqref{eq:slab walls in} do not
contribute to the asymptotic scattering diagram: these arise
from walls in $\foD_{(\widetilde X,\widetilde D)}$ with support
$\widetilde\rho'$, and $\dim\widetilde\rho'\cap\widetilde B_0=n-2$.

On the other hand, walls of the type \eqref{eq:slab walls out}
do contribute, and contribute precisely the walls of
$\Upsilon(\foD_{(X_{\Sigma},H),\inc})$ as defined in
\eqref{eq:Upsilon in}. This is the first type of wall which contributes
to the asymptotic scattering diagram.

Next, let $(\fod,f_{\fod})\in\foD_{(X_{\Sigma},H)}$ be a non-incoming
wall. Then $(\fod,\nu_*(f_{\fod}))$ is a wall in
$T_0\foD^1_{(\widetilde X,\widetilde H)}$ and gives rise to
a wall in the asymptotic scattering diagram in two possible ways.
Let $\sigma\in\Sigma$ be the minimal cone containing $\fod$. The two possibilities are that $\rho_i\not\subseteq \sigma$ for any $i$,
or $\rho_i\subseteq\sigma$ for some $i$.
Recall from \eqref{Eq: mu} we have a piecewise linear identification
$\mu:M_{\RR}\rightarrow \widetilde B_0$.

First consider the case that $\rho_i\not\subseteq\sigma$ for any
$i$. In this case, the corresponding wall of the asymptotic scattering
diagram is 
\[
(\mu(\fod),\nu_*(f_{\fod})),
\]
where we now need to keep in mind that $\nu_*(f_{\fod})$ has to
be written, as indicated in Remark \ref{rem:T0 observations} (2),
in terms of parallel transport from $\shP_0^+$ to $\shP_x^+$ for
$x\in\Int(\sigma)$. Let us keep track of the behaviour of
a monomial in $f_{\fod}$.

Under the map $\nu_*$, a monomial of the form
$\prod_{i,j} (t_{ij}z^{m_i})^{a_{ij}}$ in $f_{\fod}$ is mapped to
\[
z^{\left(-m_{\mathbf{A}},\sum_{i,j}a_{ij}(\varphi_0(m_i)+F_i-E^j_i)\right)}.
\]
If we then parallel transport this monomial into the interior
of $\sigma$ and apply $\mu_*=\Upsilon_*$, we get
\[
t^{(d\varphi_0|_{\sigma})(m_{\mathbf{A}}) +\sum_{i,j} a_{ij}(\varphi_0(m_i)+F_i-E^j_i) }z^{-\Upsilon_*(m_{\mathbf{A}})}.
\]

By Lemma \ref{lem:beta recovery}, (3), bearing in mind that 
$m_i\in\rho_i$ so that $\varphi_0(m_i)=(d\varphi_0|_{\rho_i})(m_i)$, 
\[
\bar\beta_{\mathbf{A},\sigma}=(d\varphi_0|_{\sigma})(m_{\mathbf{A}})+
\sum_{i,j}a_{ij}\varphi_0(m_i).
\]
We may now write
\[
t^{(d\varphi_0|_{\sigma})(m_{\mathbf{A}}) +\sum_{ij} a_{ij}(\varphi_0(m_i)+F_i-E^j_i) }z^{-\Upsilon_*(m_{\mathbf{A}})}
=t^{\bar\beta_{\mathbf{A},\sigma}+\sum_{ij} a_{ij}(F_i-E^j_i)}z^{-\Upsilon_*(m_{\mathbf{A}})}.
\]
Here we view $\bar\beta_{\mathbf{A},\sigma}\in N_1(\widetilde X)$ via the inclusion
$X_{\Sigma}\hookrightarrow \widetilde X$ as a component of the central
fibre. Thus we have a curve class
$\bar\beta_{\mathbf{A},\sigma}+\sum_{ij} a_{ij}(F_i-E^j_i)\in N_1(\widetilde X)$.
The intersection number of this curve class
with any irreducible component of $\epsilon_{\mathbf{P}}^{-1}(0)$ is zero,
and hence by Lemma \ref{lem:iota injective},
this curve class lies in $N_1(X)$ (as must be the case by 
Proposition \ref{Prop: Asymptotic equivalence}).
Recalling the definition of $\beta_{\mathbf{A},\sigma}$ from
\eqref{eq:beta A def}, we note that
\begin{equation}
\label{eq:iota beta A sigma}
\iota(\beta_{\mathbf{A},\sigma})=
\bar\beta_{\mathbf{A},\sigma}
+\sum_{ij} a_{ij}(F_i-E_i^j)
\end{equation}
in $N_1(\widetilde X)$.
Indeed, it is enough to intersect both classes with horizontal
divisors on $\widetilde X$, e.g., the closure $\widetilde D_{\rho}$ of the
divisors $D_{\rho}\times\GG_m \subseteq X\times \GG_m \subseteq \widetilde X$
and the exceptional divisors $\widetilde E_{ij}$ over $\widetilde H_{ij}$.
Then noting that
\begin{align*}
\widetilde D_{\rho}\cdot (\bar\beta_{\mathbf{A},\sigma} +
\sum_{ij} a_{ij}(F_i-E_i^j)) = {} & \begin{cases}
D_{\Sigma,\rho}\cdot \bar\beta_{\mathbf{A},\sigma}& \hbox{$\rho\not=\rho_i$
for any $i$}\\
0&\rho=\rho_i
\end{cases}\\
\widetilde E_{k\ell}\cdot(\bar\beta_{\mathbf{A},\sigma}
+\sum_{ij} a_{ij}(F_i-E_i^j)) = {} & a_{k\ell},
\end{align*}
we obtain \eqref{eq:iota beta A sigma}.

Thus the wall of $\foD^{1,\mathrm{as}}_{(\widetilde X,\widetilde D)}$ 
arising from $(\fod, f_{\fod})\in \foD_{(X_{\Sigma},H)}$ is
\[
\big(\Upsilon(\fod), f_{\fod}(t^{\iota(\beta_{\mathbf{A},\sigma})}
z^{-\Upsilon_*(m_{\mathbf A})})\big),
\]
and agrees, up to applying $\iota$ to the curve classes, with 
the corresponding wall in $\Upsilon(\foD_{(X_{\Sigma},H)})$ given
by the definition of this latter scattering diagram.

We next consider the case that $\rho_i\subseteq\sigma$. In this case
Remark \ref{rem:T0 observations} (2) shows we obtain two
walls in $\foD^1_{(\widetilde X,\widetilde D)}$, one
contained in $\widetilde\sigma'$ and one contained in $\widetilde\sigma$.
But only the wall in $\widetilde\sigma$ contributes to
the asymptotic scattering diagram. While a monomial in $f_{\fod}$
of the form $\prod_{i,j} (t_{ij}z^{m_i})^{a_{ij}}$ still
contributes a term of the form
\[
t^{(d\varphi_0|_{\sigma})(m_{\mathbf{A}}) +\sum_{i,j} a_{ij}(\varphi_0(m_i)+F_i-E^j_i) }
z^{-(\Psi|_{\widetilde B_1})^{-1}_*(m_{\mathbf{A}})}
\]
to the corresponding wall in $\widetilde\sigma'$, we still need
to take into account the kink in $\varphi$ to describe the 
corresponding monomial for the wall in $\widetilde\sigma$. We obtain,
following the same argument as in the previous case,
a monomial of the form
\[
t^{\bar\beta_{\mathbf{A},\sigma} +\sum_{i,j} a_{ij}(F_i-E^j_i) 
+\langle n_{\sigma},m_{\mathbf{A}}\rangle F_i}
z^{-(\Psi|_{\widetilde B_1})^{-1}_*(m_{\mathbf{A}})},
\]
where $n_{\sigma}\in N$ is a primitive normal vector to
the facet of $\sigma$ not containing $\rho_i$ and positive on
$\rho_i$.

The support of the corresponding wall in the asymptotic scattering diagram
is still $\Upsilon(\fod)$, and the monomial contribution above
becomes
\[
t^{\bar\beta_{\mathbf{A},\sigma} +\sum_{i,j} a_{ij}(F_i-E^j_i)
+\langle n_{\sigma},m_{\mathbf{A}}\rangle F_i}
z^{-\Upsilon_*(m_{\mathbf{A}})}.
\]
We then find analogously to \eqref{eq:iota beta A sigma}
that
\[
\iota(\beta_{\mathbf{A},\sigma})=
\bar\beta_{\mathbf{A},\sigma} +\sum_{i,j} a_{ij}(F_i-E^j_i)
+\langle n_{\sigma},m_{\mathbf{A}}\rangle F_i.
\]
The only difference with the previous calculation is that
the intersection number of $\widetilde D_{\rho_i}$ with the right-hand side
is $\langle n_{\sigma},m_{\mathbf{A}}\rangle$. However,
this integer is precisely the coefficient of $m_i$ in
the expression of $m_{\mathbf{A}}$ as a linear combination of
generators of $\sigma$. By \eqref{eq:Drho dot beta},
this is also the intersection number of $\widetilde D_{\rho_i}$
with the left-hand side. Thus the argument finishes as in the
first case.
\end{proof}

\section{Example: The blow-up of $\PP^3$ with center two lines}
\label{sec:Examples}
Let $\ell_1 \subset D_{1}$ and $\ell_2\subset D_{2}$ be two general lines in $\PP^3$ as in Example  \ref{Ex: P3 with two lines}. Denote the blow-up of $\PP^3$ with center $\ell_1\cup \ell_2$ by $\Bl_{\ell_1\cup \ell_2}(\PP^3)$, the strict transform of $D$ by $\widetilde D$, similarly the strict transform of $D_i$ by $\widetilde D_i$, and the fibers of the exceptional divisors over $\ell_i$ by $E_i$, for $1\leq i \leq 2$. Recall the walls of $\foD_{(\PP^3,\ell_1\cup \ell_2)}$ are displayed in Table \ref{Table: walls of final HDTV P3}. 
We remark that this example does not precisely satisfy the hypotheses
we placed on $\Sigma$ and $\mathbf{P}$ in 
\S\ref{Sec: Pulling singularities out}, as the rays of the fan
$\Sigma$ for $\PP^3$ corresponding to $D_1$ and $D_2$ are contained in
a common cone. We note this does not affect the construction
of the scattering diagram $\foD_{(X_{\Sigma},H)}$, but the discussions
of \S\ref{Sec: Pulling singularities out} would need to be significantly
modified. If the reader prefers, they may refine $\Sigma$ by performing a 
toric blow-up along $D_1\cap D_2$, which resolves this issue, and
then suitably modify the discussion below. However, this only makes
the geometry more complicated.

The walls of the minimal scattering diagram equivalent to the canonical scattering diagram $\foD_{(\Bl_{\ell_1\cup \ell_2}(\PP^3),\widetilde D)}$ are obtained from $\foD_{(\PP^3,\ell_1\cup \ell_2)}$ by applying the piecewise linear isomorphism in Theorem \ref{Thm: HDTV}.
These are displayed in Table \ref{Table: final walls canonical P3}, where 
\begin{align}
    \label{Eq: x and y}
    \nonumber
x & = z^{e_1} = z^{(1,0,0)},    \\
\nonumber
y & = z^{e_2} = z^{(0,1,0)},
\end{align}
and  $L$ is the curve class corresponding to the strict transform of a 
general line in $\PP^3$.
\begin{table}[ht]
 \begin{tabular}{ll} \hline
 $\fod$ & $f_{\fod}$ \\ \hline
 $\langle e_1,e_2 \rangle,\langle e_1,e_3 \rangle, \langle e_1,e_4 \rangle$ & ~~ $1+t^{E_1}x^{-1}$  \\ 
$\langle e_2,e_1 \rangle,\langle e_2,e_3 \rangle, \langle e_2,e_4 \rangle$ & ~~$1+t^{E_2}y^{-1}$ \\
 $\langle e_3,-e_1 \rangle, \langle e_4,-e_1 \rangle$ & ~~ $1+t^{L-E_1}x$  \\
 $\langle e_3,-e_2 \rangle,\langle e_4,-e_2 \rangle,$ & ~~ $1+t^{L-E_2}y$  \\
  $\langle -e_1,-e_1-e_2 \rangle,\langle -e_2,-e_1-e_2 \rangle,  \langle e_3,-e_1-e_2 \rangle,  \langle e_4,-e_1-e_2 \rangle$ & ~~$1+t^{L-E_1-E_2}xy$ \\
 $\langle e_1,-e_2 \rangle$ & ~~ $1+t^{L-E_2}y+t^{L-E_1-E_2}xy$  \\ 
  $\langle e_2,-e_1 \rangle$ & ~~ $1+t^{L-E_1}x+t^{L-E_1-E_2}xy$  \\ 
\hline
  \end{tabular}
   \caption{Walls of $\foD_{\Bl_{\ell_1\cup \ell_2}(\PP^3),\tilde D)}$}
    \label{Table: final walls canonical P3}
  \end{table}
 We investigate each of these walls of $\foD_{(\Bl_{\ell_1\cup \ell_2}(\PP^3),\widetilde D)}$ and the associated punctured  Gromov-Witten invariants, by using the description of the walls of the canonical scattering diagram provided in \eqref{Eq: walls of canonical}:
\begin{equation}
\nonumber
  \big( h(\tau_{\out}), \exp(k_{\tau}N_{\tilde\tau}t^{\ul{\beta}} z^{-u})\big).
\end{equation}
Here $\ul \beta$ is a curve class in $\Bl_{\ell_1\cup \ell_2}(\PP^3)$ and $\tilde \tau= (\tau, \ul \beta)$ records the type $\tau$ of a punctured map $f:C^{\circ}/W \to \Bl_{\ell_1\cup \ell_2}(\PP^3)$, with one punctured point mapping to $\widetilde D$ defined as in Definition \ref{Def: typr of a punctured curve}.
The cell containing the image of the leg $L_{\mathrm{out}}$ corresponding to the punctured point is denoted by $\tau_{\mathrm{out}}$ as in \S\ref{Subsec: punctured}. The map $h:\Gamma \to \Sigma(\Bl_{\ell_1\cup \ell_2}(\PP^3))$ is the universal family of tropical maps obtained from $f:C^{\circ}/W \to \Bl_{\ell_1\cup \ell_2}(\PP^3)$, and $\kappa_{\tau}$ is defined as in \eqref{Eq: k_tau}. From the associated moduli space in \S\ref{Subsec: moduli spaces} we obtained the description of the punctured log invariants $N_{\tilde\tau}$. Now we are ready to investigate each wall in Table \ref{Table: final walls canonical P3}.
\begin{itemize}
    \item For $(\fod,f_{\fod})= (\langle e_1,e_i \rangle,1+t^{E_1}x^{-1})$, $2\leq i \leq 4:$ 
    \[ \mathrm{log}(f_{\fod}) =  \sum_{k\geq 1}k\frac{(-1)^{k+1}}{k^2}t^{kE_1}x^{-k}. \]
 Consider the moduli space of punctured maps of class $kE_1$ with one punctured point having contact order $k$ with $\widetilde D_1$. Such maps arise as $k:1$ multiple covers of fibers of the exceptional locus over $\ell_1$. In the associated tropical moduli space, there is a one-parameter family of tropical curves $h:\Gamma \to Q_{\tau,\RR}^{\vee}$, where each curve has one vertex with image mapping onto the cone generated by $e_i$, and one leg $L_{\mathrm{out}}$ with direction vector $u=(k,0,0)$ corresponding to the punctured point. In this case we get $h(\tau_{\mathrm{out}})=\langle e_1,e_i \rangle$ as desired. This tropical family defines a type $\tau$ with $Q_{\tau}=\NN$. The corresponding punctured map of type $\tau$ is a multiple cover of the fiber of the exceptional locus over the intersection point $\ell_1\cap D_i$. From \eqref{Eq: k_tau} we obtain $\kappa_{\tau}=k$ and hence 
    \begin{equation}
    \label{Eq: N-tau first row two lines}
        N_{\tilde \tau} = \frac{(-1)^{k+1}}{k^2}.
    \end{equation}
    This gives the same multiple cover formula
as in \cite[Prop 6.1]{GPS}, but in higher dimension, and without recourse to localization
techniques.
  \item For $(\langle e_2,e_i \rangle,1+t^{E_2}y^{-1})$, for $i\in \{1,3,4\}$,  we obtain the multiple cover contribution \eqref{Eq: N-tau first row two lines} for $k:1$ multiple covers of the fiber of the exceptional locus over $\ell_2 \cap D_i$. 
    \item For $(\langle -e_1,e_i \rangle,1+t^{L-E_1}x)$, $3\leq i \leq 4$, we obtain the multiple cover contribution \eqref{Eq: N-tau first row two lines} for $k:1$ multiple covers of strict transforms of lines passing through the points $ D_2\cap D_3\cap D_4$ and $\ell_1 \cap D_i$.
         \item $(\langle -e_2,e_i \rangle,1+t^{L-E_2}x)$, for $3\leq i \leq 4$, we obtain the multiple cover contribution \eqref{Eq: N-tau first row two lines} for $k:1$ multiple covers of strict transforms of lines passing through the points $ D_1\cap D_3\cap D_4$ and $\ell_2\cap D_i$.
   \item For $(\langle -e_i,-e_1-e_2 \rangle,1+t^{L-E_1-E_2}xy)$, $1\leq i \leq 2$, we obtain the multiple cover contribution \eqref{Eq: N-tau first row two lines} for $k:1$ multiple covers of strict transforms of lines passing through the points $D_i \cap D_3 \cap D_4$ and $\ell_j \cap D_i$, where $1\leq j \leq 2$ and $i \neq j$.
   \item $(\langle e_i,-e_1-e_2 \rangle,1+t^{L-E_1-E_2}xy)$, $3\leq i \leq 4$, we obtain the multiple cover contribution \eqref{Eq: N-tau first row two lines} for $k:1$ multiple covers of strict transforms of lines passing through the points $\ell_1 \cap D_i$ and $\ell_2 \cap D_i$.
   \item For $(\fod,f_{\fod}) = (\langle e_1,-e_2 \rangle,1+t^{L-E_2}y+t^{L-E_1-E_2}xy)$ we have
   \begin{equation}
       \label{Eq: log 1+ y+xy}
       \mathrm{log}(f_{\fod}) =  \sum_{k\geq 1}\sum_{\ell= 0}^k k\frac{(-1)^{k+1}}{k^2}{k\choose \ell} t^{k(L-E_2)-\ell E_1}x^{\ell}y^k.
   \end{equation}
   Consider $k:1$ multiple covers of the line of class $L-E_1-E_2$ obtained as the strict transform of the line passing through the points $D_1 \cap D_3 \cap D_4$ and $\ell_2 \cap D_1$. The inverse image of the point of intersection of this line with $\ell_1$ consists of $k$ points on the domain. Attach transversally to $k-\ell$ points among these $k$ points, a copy of a projective line and impose each of these additional projective lines to map to a fiber of the exceptional locus over $\ell_1$, so that the image will be of class $k(L-E_1-E_2)+(k-\ell)E_1$. The choice of the $k-\ell$ points among these $k$ points contributes the coefficient ${k\choose \ell}$ in \eqref{Eq: log 1+ y+xy}, and the $k:1$-multiple cover contribution is $\frac{(-1)^{k+1}}{k^2}$ as in \eqref{Eq: N-tau first row two lines}. The associated tropical family consists of tropical curves where $h(L_{\mathrm{out}})$ has support in $\langle e_1,(-\ell,-k,0) \rangle$ as illustrated on the left hand-side in Figure \ref{Fig: Xaxis}, where the vertex at the origin corresponds to images of the vertices corresponding to the transversally attached projective lines and the other vertex on the $e_1$-axis corresponds to the image of the $k:1$-multiple cover. However, we also have a second family of punctured maps of the same class 
   \[ ~~~~~~  \,\ \,\ \,\  \,\ \,\ \,\ k(L-E_1-E_2)+(k-\ell)E_1 = (k-\ell)(L-E_2)+\ell(L-E_1-E_2) \] 
   whose domain is a chain of three projective lines with images given as follows: the image of the first line is a $(k-\ell)$-fold multiple cover of the strict transform of a line passing through the points $p=D_1\cap D_3 \cap D_4$ and a point in $\ell_2$, the second line gets contracted to $p$, and the third line is an $\ell$-fold cover of the strict transform of the line passing through the points $p$ and $D_1 \cap \ell_2$. In the right hand side of Figure \ref{Fig: Xaxis} the image of the first line corresponds to the vertex at the origin, the image of the second line is the trivalent vertex moving along the $e_2$-axis and the image of the third line corresponds to the vertex moving along the $e_1$-axis. In this case the legs of the one parameter family of tropical curves trace out the wall with support $\langle -e_2,(-\ell,-k,0) \rangle$. By the scattering computation after all cancellations we obtain a wall with support $\langle e_1,-e_2 \rangle$ as desired. This gives us the prediction of the punctured log invariant corresponding to the second family to be equal to $\frac{(-1)^{k}}{k^2}{k\choose \ell}$. The analysis of the last row of Table \ref{Table: final walls canonical P3} is analogous, and obtained by exchanging $x$ and $y$.
\end{itemize}
\begin{figure}
\center{\input{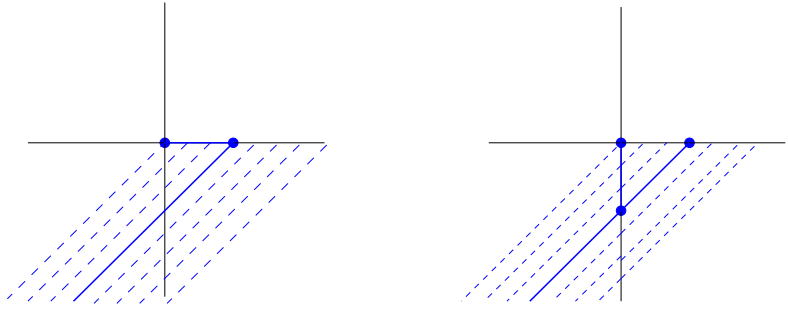}}
\caption{Cancellations of walls}
\label{Fig: Xaxis}
\end{figure}

\begin{remark}
The four walls with attached wall function $1+t^{L-E_1-E_2}xy$
may be viewed as providing additional geometric information.
For each point $x\in D_3\cap D_4$, there is precisely one line
through $x$, $\ell_1$ and $\ell_2$, and so there is a one-dimensional
family of lines through $\ell_1$, $\ell_2$ and $D_3\cap D_4$
parameterized by the intersection point with $D_3\cap D_4$.
We may blow up $D_3\cap D_4$, with exceptional divisor $E\cong
\PP^1\times\PP^1$. This
corresponds to a star subdivision
of the fan for $\PP^3$ along $\RR_{\ge 0}(-e_1-e_2)$. In this
case, the intersection points with $E$ of the strict transforms of the lines 
in the one-dimensional family trace out a curve $C$ in $E$. The four
walls instruct us that $C$ intersects each boundary divisor of $E$
once, and hence $C$ is a curve of class $(1,1)$ in $\PP^1\times\PP^1$. 

\end{remark}

We would like to note that in general, it is not obvious how to recover the punctured Gromov--Witten invariants $N_{\tilde{\tau}}$ of a log Calabi--Yau $(X,D)$ as in \eqref{Eq: blow up} from the data of the toric scattering diagram $\foD_{(X_{\Sigma},H)}$. Indeed, from the main theorem \ref{Thm: HDTV} we can only deduce that the canonical scattering diagram, which encodes the data of $N_{\tilde{\tau}}$, is equivalent to the image of $\foD_{(X_{\Sigma},H)}$ under a PL isomorphism. Thus, we can recover the canonical scattering diagram from $\foD_{(X_{\Sigma},H)}$ only up to equivalence. This enables us in practice to determine explicit equations for the coordinate rings of mirrors to $(X,D)$ in arbitrary dimension, and writing such equations will be focus of future work \cite{ACG}. On the other hand, generally this does not immediately enable us to read off the punctured invariants for arbitrary $(X,D)$, as while defining equivalent scattering diagrams we allow cancellations of walls which have the same support. However, to read off the punctured log Gromov--Witten invariant associated to a wall produced in the scattering algorithm described in the proof of Theorem \ref{thm:higher dim KS}, we need to know how to read off the data of the combinatorial types $\tau$ from the supports of the walls we add. If in the final consistent scattering diagram $\foD_{(X_{\Sigma},H)}$ we have cancellations of walls that are on top of each other, to detect the combinatorial type is not trivial and one may need to do appropriate perturbations of the widgets corresponding to the tropical hypersurfaces forming $H$ to do it. This could be done analogously to the two dimensional case in \cite[\S 1.4]{GPS}.


\bibliographystyle{plain}
\bibliography{bibliography}

\end{document}